\documentclass[a4paper, 12pt]{article}
\usepackage{amsmath}
\usepackage{amsfonts}
\usepackage{amssymb}
\usepackage{amscd}
\usepackage{amsthm}
\usepackage{aprod}

\renewcommand{\P}{{\mathbb P}}
\newcommand{\Z}{{\mathbb Z}}

\newcommand{\G}{{Gers}}
\newcommand{\Kb}{{\mathbf K}}
\newcommand{\F}{{\mathcal F}}
\renewcommand{\H}{{\mathcal H}}
\newcommand{\E}{{\mathcal E}}

\newcommand{\Ac}{{\mathcal A}}
\newcommand{\V}{{\mathcal V}}

\newcommand{\Gc}{{\mathcal G}}
\newcommand{\M}{{\mathcal M}}
\newcommand{\Pc}{{\mathcal P}}
\newcommand{\Cc}{{\mathcal C}}
\newcommand{\N}{{\mathcal N}}

\newcommand{\K}{{\mathcal K}}
\newcommand{\LL}{{\mathcal L}}

\newcommand{\A}{{\mathbf A}}
\newcommand{\Ab}{{\mathbb A}}

\newcommand{\OO}{{\mathcal O}}

\newcommand{\Alb}{{\rm Alb}}
\newcommand{\Sym}{{\rm Sym}}
\newcommand{\codim}{{\rm codim}}
\newcommand{\Pic}{{\rm Pic}}

\newcommand{\Ker}{{\rm Ker}}

\newcommand{\Imm}{{\rm Im}}
\newcommand{\Spec}{{\rm Spec}}

\newcommand{\Tor}{{\rm Tor}}
\renewcommand{\dim}{{\rm dim}}

\newcommand{\sing}{{\rm sing}}
\renewcommand{\div}{{\rm div}}

\theoremstyle{plain}
\newtheorem{theor}{Theorem}[section]
\newtheorem{prop}[theor]{Proposition}
\newtheorem{claim}[theor]{Claim}
\newtheorem{defin-prop}[theor]{Definition-Proposition}
\newtheorem{corol}[theor]{Corollary}
\newtheorem{lemma}[theor]{Lemma}
\newtheorem{quest}[theor]{Question}

\theoremstyle{remark}
\newtheorem{rmk}[theor]{Remark}
\newtheorem{examp}[theor]{Example}
\newtheorem{examps}[theor]{Examples}
\theoremstyle{definition}
\newtheorem{defin}[theor]{Definition}
\title{Adelic resolution for homology sheaves}
\date{}
\author{Sergey Gorchinskiy\footnote{The author was partially supported
by the grants RFBR 04-01-00613, 05-01-00455, and INTAS
05-1000008-8118.}\\ \small{Steklov Mathematical Institute}\\
\small{\it e-mail: gorchins@mi.ras.ru}}
\begin{document}
\maketitle

\begin{abstract}
A generalization of the usual ideles group is proposed, namely, we
construct certain adelic complexes for sheaves of $K$-groups on
schemes. More generally, such complexes are defined for any abelian
sheaf on a scheme. We focus on the case when the sheaf is associated
to the presheaf of a homology theory with certain natural axioms,
satisfied by $K$-theory. In this case it is proven that the adelic
complex provides a flasque resolution for the above sheaf and that
the natural morphism to the Gersten complex is a quasiisomorphism.
The main advantage of the new adelic resolution is that it is
contravariant and multiplicative in contrast to the Gersten
resolution. In particular, this allows to reprove that the
intersection in Chow groups coincides up to sign with the natural
product in the corresponding $K$-cohomology groups. Also, we show
that the Weil pairing can be expressed as a Massey triple product in
$K$-cohomology groups with certain indices.
\end{abstract}

\tableofcontents

\section{Introduction}
Classical adeles have been defined by A.\,Weil and C.\,Chevalley for
a global field $K$ in the following way: {\it adeles} are elements
of the set $\A_K=\displaystyle\aprod_v{K_v}$, where $v$ runs over
all valuations of the field $K$, $K_v$ is the completion of $K$ at
the point $v$, and the {\it restricted product} $\aprod{}$ means
that we consider only collections $\{f_v\}\in \displaystyle\prod_v
K_v$ such that $f_v\in \hat\OO_v$ for almost all $v$, where
$\hat\OO_v\subset K_v$ is the complete local ring corresponding to
$v$. The set of adeles has a natural structure of a topological
ring. Its group of units is called the group of {\it ideles} and is
equal to $\A^*_K=\displaystyle\aprod_v{K_v^*}$, where the restricted
product is taken in the same sense as above with $K_v$ replaced by
$K_v^*$ and $\hat\OO_v$ replaced by $\hat\OO_v^*$ (we omit the
details about a topological structure of the ideles). The group of
ideles plays a central role in the classical one-dimensional global
class field theory. One of its main properties is the relation with
the class group $Cl(K)$ of $K$. Actually there is a natural
surjective homomorphism $\A_K^*\to Cl(K)$ defined by the formula
$\{f_v\}\mapsto \sum\limits_v \nu_v(f)\cdot[v]$, where the sum is
taken over all non-archemedian valuations $v$ and $\nu_v$ denotes a
discrete valuation corresponding to $v$.

Further, J.-P.\,Serre used in \cite{Ser} {\it non-complete} adeles
on a {\it curve} $X$ over a field $k$ to prove the Riemann--Roch
theorem for $X$. Namely he considered collections $\{f_x\}\in
\displaystyle\prod_{x\in X} k(X)$ such that $f_x\in \OO_{X,x}$ for
almost all closed points $x\in X$, where $\OO_{X,x}$ is the local
ring of $X$ at $x$. Similarly, one may construct a non-complete
version of ideles. Moreover, Serre was using a certain complex of
adeles though he did not mention this explicitly.

A.\,N.\,Parshin introduced in \cite{Par76} non-complete (also called
{\it rational}) adeles on a {\it surface}, and has constructed a
certain {\it adelic complex}. In \cite{Par78} there is also a
multiplicative version of complete adeles on a surface, related to
the $K_2$-functor, and have been proposed a natural two-dimensional
generalization of the classical class field theory. Later,
A.\,A.\,Beilinson defined in \cite{Bei} a complex of adeles
${\mathbb A}(X,\OO_X)^{\bullet}$ for any Noetherian scheme $X$ using
simplicial language (in fact, the adelic complex is defined for any
quasicoherent sheaf $\F$ on $X$ instead of $\OO_X$). Let us describe
explicitly the complexes of rational adeles in low dimensions. For a
curve $X$ over a field $k$, it looks like
$$
0\to k(X)\oplus \prod_{x\in X}\OO_{X,x}\to
\aprod_{x\in X}{k(X)}\to 0,
$$
where the restricted product is taken in the
above sense and the differential is defined using the natural inclusions by
the formula $(f_X,\{f_x\})\mapsto \{f_x- f_X\}$. For a surface $X$, the
rational adelic complex has the following form:
$$
0\to k(X)\oplus \prod_{x\in
X}\OO_{X,x}\oplus\prod_{C\subset X}\OO_{X,C}\to \aprod_{C\subset X}{k(X)}\oplus
\aprod_{x\in X}{k(X)}\oplus \aprod_{x\in C}{\OO_{X,C}}\to \aprod_{x\in C\subset
X}{k(X)}\to 0,
$$
where the restricted products can be defined explicitly in
terms of the poles of functions, and the differentials are defined using the
natural inclusion by the formulas
$$
(f_X,\{f_x\},\{f_C\})\mapsto
(\{f_C-f_X\}, \{f_x-f_X\},\{f_x-f_C\})
$$
and
$$
(\{f_{XC}\},\{f_{Xx}\},\{f_{Cx}\})\mapsto \{f_{Cx}-f_{Xx}+f_{XC}\}.
$$
The Beilinson--Huber theorem tells that for any Noetherian scheme, the cohomology
groups of the adelic complex ${\mathbb A}(X,\OO_X)^{\bullet}$ are canonically isomorphic
to the cohomology groups $H^i(X,\OO_X)$, see \cite{Bei}, \cite{Hub} (the
analogous result holds true for any quasicoherent sheaf $\F$ on $X$).

The aim of this paper is to give a version of these constructions
for a class of sheaves of abelian groups on schemes different from
quasicoherent sheaves, in particular, for {\it sheaves of
$K$-groups}. Recall that a sheaf of $K$-groups $\K_n^X=\K_n(\OO_X)$
is associated with the presheaf given by $U\mapsto K_n(k[U])$, $n\ge
0$, where $U\subset X$ is an open subset in the scheme $X$ and
$K_n(-)$ denotes the Quillen $K$-group. Recall that for a regular
Noetherian separable scheme $X$ of finite type over a field, there
are Gersten (or Cousin) complexes $\G(X,n)^{\bullet}$ whose
cohomology groups are equal to $H^i(X,\K^X_n)$. On the other hand,
for the case of a regular curve $X$, there is a natural
quasiisomorphism of complexes
$$
\begin{array}{ccccccc}
0&\to& K_n(k(X))\oplus \displaystyle\prod_{x\in
C}K_n(\OO_{X,x})&\to& \displaystyle\aprod_{x\in
X}{K_n(k(X))}&\to&0\\ &&\downarrow &&\downarrow&&\\ 0&\to&
K_n(k(X)) &\to&\bigoplus\limits_{x\in X}
K_{n-1}(k(x))&\to&0\end{array}\\
$$
where the restricted product is taken in the same way as before for
rational adeles with $k(X)$ replaced by $K_n(k(X))$ and $\OO_{X,x}$
replaced by $K_n(\OO_{X,x})$, $k(x)$ stands for the residue field at
$x$, the complex on the bottom is the Gersten complex for the curve
$X$, the first vertical morphism is the projection on the first
summand, and the second one is given by residue maps. We give a
higher-dimensional generalization of this. Recall that in general
the Cousin complex of an abelian sheaf consists of direct sums over
schematic points of fixed codimension. Following the general
definition of adeles, we replace these direct sums by {\it adelic
groups}, which are certain restricted products over {\it flags} of
fixed length, i.e., sequences of schematic points
$\eta_0\ldots\eta_p$ on a scheme $X$ such that
$\eta_i\in\overline{\eta}_{i-1}$ for all $1\le i\le p$. For the
explicit construction of these groups in the simplest cases see
Examples \ref{simple_examples}. Next, we construct an {\it adelic
complex} $\A(X,\F)^{\bullet}$, which consists of adelic groups.
Under certain natural conditions, there is a canonical morphism of
complexes $\nu_X:\A(X,\F)^{\bullet}\to Cous(X,\F)^{\bullet}$, where
$Cous(X,\F)^{\bullet}$ is the Cousin complex of $\F$ on $X$. We
restrict our attention to a special type of abelian sheaves, namely,
sheaves associated with the presheaves of a homology theory that
satisfies certain axioms (see Section \ref{homology-theories}). Our
main result is that for such sheaves, the morphism $\nu_X$ is a
quasiisomorphism on smooth varieties over an infinite perfect field.
In particular, we get the following result.

\begin{theor}
There is a canonical morphism of complexes $\nu_X:
\A(X,\K_n^X)^{\bullet}\to \G(X,n)^{\bullet}$. This morphism is a
quasiisomorphism if $X$ is a smooth variety over an infinite perfect field.
\end{theor}

Recall that one of the main advantages of the Gersten complex is that it
allows to relate explicitly cohomology of the sheaves of $K$-groups,
called {\it $K$-cohomology}, with the (algebraic) geometry of $X$. In
particular, the famous Bloch--Quillen formula says that
$H^n(X,\K_n(\OO_X))=CH^n(X)$, see \cite{Q}. At the same time there is a
canonical product between the sheaves of $K$-groups, induced by the
product in $K$-groups themselves. This product structure
cannot be prolonged to the Gersten complex: otherwise there would
exist an intersection theory for algebraic cycles without taking
them modulo rational equivalence.

The main advantage of the adelic construction is that the flag
simplicial structure, involved in the definition of the adelic
complexes, allows to define a product on them, i.e., the complex
$\bigoplus\limits_{n\ge 0}\A(X,\K^X_n)^{\bullet}$ is a DG-ring. Note
that the general theory of sheaves provides many multiplicative
simplicial resolutions of sheaves, i.e., resolutions carrying the
product structure, for example, Chech or Godement resolutions.
Nevertheless these resolutions seem to be too general to reflect the
algebro-geometric structure of a scheme, for instance, relations of
$K$-cohomology to algebraic cycles or direct images for proper
morphisms. Though the adelic complex as presented here also does not
have direct images, the covariant Gersten complex turns out to be a
right DG-module over the DG-ring of adeles. Roughly speaking, the
difference between the adelic complex $\A(X,\K^X_n)^{\bullet}$ and
the Gersten complex $\G(X,n)^{\bullet}$ consists of all possible
{\it systems of local $K$-group equations along flags} for each
irreducible subvariety on $X$. The main idea is that in order to get
an intersection-product on the groups of algebraic cycles we enlarge
them by systems of equations instead of taking them modulo rational
equivalence.

Analysis of the adelic complex provides certain explicit formulas
for products and also Massey higher products in $K$-cohomology. In
particular, we obtain a new direct proof of the coincidence up to
sign of the intersection product in Chow groups and the natural
product in $K$-cohomology. Another example is the triple product
$m_3(\alpha,l,\beta)$, where $\alpha\in CH^d(X)_l$, $\beta\in
\Pic^0(X)_l$. It occurs that this triple is equal to the Weil
pairing of $\alpha$ and $\beta$. In the case of a curve the equality
of the corresponding explicit adelic formula with the Weil pairing
was proved by different methods in \cite{How}, \cite{Maz}, and
\cite{Gor}.

The paper is organized as follows. First in Section
\ref{defin-basic-adeles} we define adelic groups for abelian sheaves
on arbitrary schemes and study their basic properties, such as
multiplicativity and contravariancy. Section \ref{Cousin-relation}
shows an important relation of the adelic complex with the Cousin
complex. In Sections \ref{section-control-support} and
\ref{section-1-pure} we establish more properties of adelic groups
imposing rather mild conditions on sheaves. In Section
\ref{section-prime-adeles} we define a new type of adelic groups,
called $\A'$-adeles, which do not have a multiplicative structure
but provide a flasque resolution for any Cohen--Macaulay sheaf.

In Section \ref{homology-theories} we introduce the notion of a
homology theory locally acyclic in fibrations (l.a.f. homology
theory). We discuss the basic properties of an l.a.f. homology
theory and define the associated homology sheaves. Then, in Section
\ref{K-theory} we define strongly locally effaceable pairs of closed
subvarieties on a smooth variety with respect to an l.a.f. homology
theory. This is a globalization of the method from Quillen's proof
of the Gersten conjecture in the geometric case. Section
\ref{section-existence-addition} contains the existence and the
addition results for strongly locally effaceable pairs. In
particular, we get some uniform version of the local exactness of a
Gersten resolution, which might have interest in its own right (see
Corollary \ref{uniform_Gersten}). In Section
\ref{patching-systems-section} we introduce the notion of patching
systems of closed subvarieties on a smooth variety. This is our key
tool for studying the relation between the adelic and the Gersten
complexes. Section \ref{section-main-theorem} contains the main
result (Theorem \ref{quasiis}): for any l.a.f. homology theory the
adelic complex of the homology sheaves is quasiisomorphic to the
Gersten complex on any smooth variety over an infinite perfect
field. After several easy reductions the proof of the main theorem
is reduced to a certain approximation result, namely Lemma
\ref{approximation}, whose proof uses the developed technique of
patching systems. Section \ref{explicit-classes} is devoted to the
explicit construction of cocycles in the adelic complex representing
cocycles in the Gersten complex.

Then we consider our main example of an l.a.f. homology theory,
namely the \mbox{$K'$-theory} of schemes. We recall some general
facts on $K$-sheaves and $K$-cohomology in Section
\ref{section-K-generalities}. We also give some examples of the
explicit $K$-adelic cocycles for algebraic cycles and we study the
link between the $K$-adeles and the coherent differential adeles of
Parshin and Beilinson. Section \ref{Euler_charact} provides an
explicit construction of an Euler characteristic map from the
$K$-groups of the exact category of complexes of coherent sheaves on
a scheme $T$ that are exact outside of a closed subscheme $S\subset
T$ to the $K'$-groups of $S$. This map can be also constructed using
$R$-spaces introduced in \cite{Blo84} or general preperties of
Waldhausen $K$-theory of perfect complexes given in \cite{Tho}.
Next, explicit formulas for products in $K$-cohomology in terms of
Gersten cocycles are obtain in Section
\ref{product_cocycles_section} as a consequence of the product
structure on the adelic complex~(Theorem \ref{intersecting_cycles}).
The case of certain Massey triple products is treated in Section
\ref{section-triple-formula}. We also show the coincidence of the
considered Massey triple product with the Weil pairing between
zero-cycles and divisors (Proposition~\ref{Weil_Massey}).

It is a pleasure for the author to thank A.\,N.\,Parshin for his
constant attention to this work and many helpful suggestions,
C.\,Soul\'e for many useful discussions and remarks, D.\,Grayson,
and D.\,Kaledin for important comments, which appeared after they
read the manuscript of this text. The author is very grateful for
the hospitality during his visits to Institut de Math\'ematiques de
Jussieu, where many parts of the text had been written down.

\section{Generalities on adeles}

\subsection{Definition and first properties}\label{defin-basic-adeles}

We use Beilinson's simplicial approach to
higher-dimensional
adeles, first defined by Parshin in the two-dimensional case
(see \cite{Bei}, \cite{Par76}). Besides, we follow most notations from \cite{Hub}.
For a cosimplicial group $A^*$, let $A^{\bullet}$ be the associated cochain
complex. We define the differential in the tensor product
$A^{\bullet}\otimes B^{\bullet}$ of two complexes $A^{\bullet}$
and $B^{\bullet}$ by the formula $d(a\otimes b)=da\otimes
b+(-1)^{\deg a}a\otimes db$. For a scheme $X$, by $X^{(p)}$ denote the set
of all schematic points of codimension $p$ on $X$.

Let $X$ be a scheme and let $P(X)$ be the set of all its schematic points. By
$\overline{\eta}$ denote the closure of a point $\eta\in P(X)$. By definition, put
$$
S(X)_p=\{(\eta_0,\eta_1\ldots,\eta_p):\eta_i\in
P(X),\eta_{i}\in\overline{\eta}_{i-1}\mbox{ for all } 1\le i \le p\}.
$$
An element $F=(\eta_p\ldots\eta_p)\in S(X)_p$ is called
a {\it flag of length} $p$. The assignment $X\mapsto S(X)_{*}$ is
a covariant functor from the category of schemes to
the category of simplicial sets. Let
$\delta^{p}_i:S(X)_{p}\to S(X)_{p-1}$, $0\le i\le p$ and
$\sigma^p_i:S(X)_p\to S(X)_{p+1}$, $0\le i\le p$ be
the natural boundary and degeneracy maps, respectively.

There is an exact additive functor $\F\mapsto S(X,\F)$ from the category of abelian
sheaves on $X$ to the category of cohomological abelian systems of coefficients
on $S(X)_*$, given by $S(X,\F)(\eta_0\ldots\eta_p)=\F_{\eta_0}$ for any flag
$(\eta_0\ldots\eta_p)\in S(X)_p$. By $C(X,\F)^*$ denote the cosimplicial group
associated to the system of coefficients $S(X,\F)$ on $S(X)_*$. We have
$C(X,\F)^p=\prod\limits_{\eta_0\ldots\eta_p}\F_{\eta_0}$. Explicitly, the
differential in the complex $C(X,\F)^{\bullet}$ is given by the formula
$(df)_{\eta_0\ldots\eta_{p+1}}=\sum\limits_{i=0}^p(-1)^i
f_{\eta_0\ldots\hat{\eta}_i\ldots\eta_{p+1}}\in F_{\eta_0}$ for any element
$f\in C(X,\F)^p$, where the hat means that we omit the corresponding element
in the flag. By definition, put $C(M,\F)=
\prod\limits_{(\eta_0\ldots\eta_p)\in M}\F_{\eta_0}$ for a subset $M\subset
S(X)_p$. In particular, we have $C(X,\F)^p=C(S(X)_p,\F)$.

Evidently,
$S(X,\F\otimes\Gc)=S(X,\F)\otimes S(X,\Gc)$ for any two sheaves $\F$ and $\Gc$ on $X$
(we consider a point wise multiplication for systems of coefficients). Consequently
there is a canonical morphism of complexes
$C(X,\F)^{\bullet}\otimes C(X,\Gc)^{\bullet}\to C(X,\F\otimes\Gc)^{\bullet}$ given by
$(f\cdot g)_{\eta_0\ldots\eta_{p+q}}=
f_{\eta_0\ldots\eta_p}\otimes g_{\eta_p\ldots\eta_{p+q}}$ for any elements
$f\in C(X,\F)^p$, $g\in C(X,\Gc)^q$.

If $f:X\to Y$ is a morphism of schemes, then there is a natural morphism of
systems of coefficients $S(Y,f_*\F)\to f_*S(X,\F)$ on $S(Y)_*$, where
$f_*S(X,\F)(G)= \prod\limits_{F:f(F)=G}S(X,\F)(F)$ for any flag $G\in S(Y)_*$.
Consequently there is a morphism of cosimplicial groups $C(Y,f_*\F)^*\to
C(X,\F)^*$.

Let $\F$ be a sheaf of abelian groups on a scheme $X$. We put $\F_U=(i_U)_*i_U^*\F$
for any open embedding $i_U:U\hookrightarrow X$. Note that $(\F_U)_V=\F_{U\cap V}$
for two open subsets $U$ and $V$ in $X$. For a point $\eta\in X$, we put
$[\F]_{\eta}=j_*j^*(\F)$, where $j:\Spec\OO_{X,\eta}\to X$ is the natural
morphism. We have
$[\F]_{\eta}=\lim\limits_{\longrightarrow}\F_{U}$, where
the limit is taken over all open subsets $U\subset X$ containing the point $\eta$.

For $M\subset S(X)_p$, $\eta\in P(X)$, by
$_{\eta}M$ denote the following set:
$$
{}_{\eta}M=\{(\eta_1,\ldots,\eta_p)\in
S(X)_{p-1}:(\eta,\eta_1,\ldots,\eta_p)\in M\}.
$$
We define inductively the {\it adelic groups} $\A(M,\F)$, $M\subset S(X)_p$
of $\F$ on $X$ in the following way.

\begin{defin}
For a subset $M\subset P(X)=S(X)_0$, we put
$$
\A(M,\F)=C(M,\F)=\prod_{\eta\in M}\F_{\eta}.
$$
For a subset $M\subset S(X)_p$, $p>0$, we put
$$
\A(M,\F)=\prod_{\eta\in P(X)}\widetilde{\A}(_{\eta}M,[\F]_{\eta}),
$$
and
$$
\widetilde{\A}(M,[\F]_{\eta})=\lim\limits_{\longrightarrow\atop U} \A(M,\F_U),
$$
where the limit is taken over all open subsets $U\subset X$ containing the point $\eta$.
Also, we put $\A(\emptyset,\F)=0$.
Elements of the adelic groups $\A(M,\F)$ are called {\it adeles}.
\end{defin}

\begin{rmk}
The definition of adelic groups
does not use the ring structure of the sheaf $\OO_X$. In fact, all generalities
about adeles that are discussed below make sense for abelian sheaves on any
topological space such that every closed subset has a unique generic point.
\end{rmk}

\begin{rmk}
This definition is analogous to the definition of adeles from \cite{Bei}. The main
differences with \cite{Bei} are as follows: we replace coherent sheaves by
sheaves of type $\F_V$ and we use no completion in the construction.
Consequently our adelic condition is more rough: if $\F$ is a coherent sheaf
on a Noetherian scheme $X$, then the defined above group $\A(M,\F)$
contains the corresponding group of
{\it rational adeles} (see \cite{Par76} and \cite{Hub}). However there is a
comparison in the backward direction, see Proposition \ref{Parshin-Beilinson}.
\end{rmk}

\begin{rmk}
The scheme $X$ is not included in our notation for adelic groups $\A(M,\F)$.
Nevertheless in what follows $X$ could be always reconstructed from
the notation for a sheaf $\F$.
\end{rmk}

\begin{rmk}\label{limit-prod}
It follows from the definition that for any subset $M\subset S(X)_p$, $p>0$ and
for any open subset $V\subset X$, we have
$$
\A(M,\F_V)=\prod_{\eta\in P(X)}\lim\limits_{\longrightarrow\atop U_{\eta}}
\A(_{\eta}M,\F_{V\cap U_{\eta}})=\lim\limits_{\longrightarrow\atop \{U_{\eta}\}}
\prod_{\eta\in P(X)}\A(_{\eta}M,\F_{V\cap U_{\eta}}),
$$
where the second limit is taken over the set of systems $\{U_{\eta}\}$
of open subsets in $X$ parameterized by
schematic points $\eta$ such that $\eta\in U_{\eta}$ for any $\eta\in P(X)$, and
$\{U_{\eta}\}\le\{U'_{\eta}\}$ if and only if $U_{\eta}\supseteq U'_{\eta}$ for
all $\eta\in P(X)$.
\end{rmk}

Evidently, $\F\mapsto \A(M,\F)$ is a covariant functor from the category of
abelian sheaves on $X$ to the category of abelian groups for any subset $M\subset
S(X)_p$. It is easily shown that $\A((\eta_0\ldots\eta_p),\F)=\F_{\eta_0}$
for any element $(\eta_0\ldots\eta_p)\in S(X)_p$. For any subset
$M\subset S(X)_p$, there is a natural morphism
$\theta:\A(M,\F)\to C(M,\F)=\prod\limits_{(\eta_0\ldots\eta_p)\in M}\F_{\eta_0}$.

\begin{prop}\label{Properties-M}
\hspace{0cm}
\begin{itemize}
\item[(i)]
For any subsets $M,N\subset S(X)_p$, $p\ge 0$ such that $M\cap N=\emptyset$, there is
a decomposition $\A(M\cup N,\F)=\A(M,\F)\oplus\A(N,\F)$;
\item[(ii)]
for any subset $M\subset S(X)_{p}$, $p>0$,
there are boundary maps $d^{p}_i:\A(\delta^{p+1}_i(M),\F)\to\A(M,\F)$, $0\le i\le p$;
\item[(iii)]
for any subset $M\subset S(X)_{p}$,
there are isomorphisms $s^p_i:\A(\sigma^p_i(M),\F)\to\A(M,\F)$, $0\le i\le p$;
\item[(iv)]
for any subset $M\subset S(X)_p$ and for any sheaves
$\F$, $\Gc$ on $X$, there is a morphism
$\A(M,\F)\otimes\A(M,\Gc)\to\A(M,\F\otimes\Gc)$;
\item[(v)]
for any subset $M\subset S(Y)_p$ and
for any morphism of schemes $f:X\to Y$, there is a natural morphism
$f^*:\A(M,f_*\F)\to \A(f^{-1}(M),\F)$;
\item[(vi)]
all the morphisms constructed in $(i)-(v)$ commute via the map $\theta$ with
their natural counterparts for the corresponding direct product groups.
\end{itemize}
\end{prop}
\begin{proof}
The proof of $(i)$ and the constructions in $(ii)$, $(iii)$
are the same as the proof of
Propositions 2.1.5, Definition 2.2.2, and the proof of Proposition 2.3.1,
respectively, from \cite{Hub}. The only difference is that we use sheaves of type $\F_U$
in place of coherent sheaves.

The proof of $(iv)$ is by induction on $p$ and uses that there is a
natural morphism of sheaves $\F_U\otimes
\Gc_V\to(\F\otimes\Gc)_{U\cap V}$ for any open subsets $U,V\subset
X$. Besides, for $p=0$, the morphisms in question equals the point
wise multiplication in stalks of sheaves.

The proof of $(v)$ is also by induction on $p$ and uses that there is a natural
morphism of sheaves $(f_*\F)_U\to f_*(\F_{f^{-1}(U)})$ for any open subset $U\subset Y$.
Besides, for $p=0$, the morphism in question equals the point wise map on
stalks of sheaves.

The proof of $(vi)$ is straightforward.
\end{proof}

For a closed or an open subscheme $Y\subset X$, let $i_Y$ be the corresponding
embedding. For any subset $M\subset S(X)_p$, let $M(Y)$ be the set of flags
on $Y$ that are in $M$. The following consequences of Proposition \ref{Properties-M}
are needed for the sequel.

\begin{corol}\label{open-local}
\hspace{0cm}
\begin{itemize}
\item[(i)]
For any open subset $U\subset X$ and any subset $M\subset S(X)_p$, we have
$$
\A(M(U),\F)=\A(M(U),i^*_U\F),
$$
where the right hand side is the adelic group on $U$. In particular,
$\A(M,\F)=\A(M(U),i_U^*\F)\oplus A$ for some subgroup $A\subset\A(M,\F)$;
\item[(ii)]
for any schematic point $\eta\in X$ and any subset $M\subset S(X)_p$, we have
$$
\A(M(\eta),\F)=\A(M(\eta),j_{\eta}^*\F),
$$
where $j_{\eta}:X_{\eta}=\Spec(\OO_{X,{\eta}})\to X$
is the natural morphism of schemes, $M(\eta)$ is the set of flags on
$X_{\eta}$ that are from $M$, and the right hand side is the adelic group on
$X_{\eta}$. In particular, $\A(M,\F)=\A(M(\eta),j_{\eta}^*\F)\oplus B$
for some subgroup $B\subset\A(M,\F)$;
\item[(iii)]
for any closed subset $Z\subset X$ and any subset $M\subset S(X)_p$, we have
$$
\A(M(Z),\F)=\A(M(Z),i_Z^*\F),
$$
where the right hand side is the adelic group on $Z$. In particular, for any
sheaf $\Gc$ on $Z$ we have $\A(M(Z),(i_Z)_*\Gc)=\A(M(Z),\Gc)$;
\item[(iv)]
suppose that $Y,Z$ are closed subschemes in $X$ such that $X=Y\cup Z$; then for
any sheaf $\F$ on $X$ and for any subset $M\subset S(X)_p$ we have
$$
\A(M,\F)=[\A(M(Y),i_Y^*\F)\oplus\A(M(Z),i_Z^*\F)]/\A(M(Y\cap Z),i_{Y\cap
Z}^*\F);
$$
\item[(v)]
consider a point $\eta\in X$ and a subset $M\subset S(X)_p$
such that any flag in $M$ starts with ${\eta}$; then for any sheaf $\F$ on $X$
and for any open subset $U\subset X$ containing $\eta$ we have
$$
\A(M,\F)=\A(M,\F_U).
$$
\end{itemize}
\end{corol}

We put $\A_{s}(X,\F)^p=\A(S(X)_p,\F)$.
Using Proposition \ref{Properties-M}, we get the following statement.

\begin{corol}\label{Properties-S}
\hspace{0cm}
\begin{itemize}
\item[(i)]
There is a natural structure of a cosimplicial group on $\A_s(X,\F)^*$ such that
the natural morphism $\theta:\A_s(X,\F)^*\to C(X,\F)^*$ is a morphism of cosimplicial
groups;
\item[(ii)]
for any two sheaves $\F$ and $\Gc$ on $X$ there is a morphism of complexes
$\A_s(X,\F)^{\bullet}\otimes\A_s(X,\Gc)^{\bullet}\to
\A_s(X,\F\otimes\Gc)^{\bullet}$, which commutes via $\theta$ with the
morphism of complexes $C(X,\F)^{\bullet}\otimes C(X,\Gc)^{\bullet}\to
C(X,\F\otimes\Gc)^{\bullet}$;
\item[(iii)]
for any morphism of schemes $f:X\to Y$ and for any sheaf $\F$ on $X$
there is a morphism of cosimplicial groups $\A_s(Y,f_*\F)^*\to\A_s(X,\F)^*$,
which commutes via $\theta$ with the morphism of cosimplicial groups
$C(Y,f_*\F)^*\to C(X,\F)^*$.
\end{itemize}
\end{corol}

For a cosimplicial group $A^*$, we put
$A_{deg}^{p}=\sum\limits_{i=0}^{p} {\rm Im}(s^p_i)$, where $s^p_i:A^{p+1}\to A^{p}$
are the degeneracy maps; then $A_{deg}^{\bullet}$ is a subcomplex in the complex
$A^{\bullet}$ (however there is no
analogous inclusion of cosimplicial groups). It is well known that the
quotient map $A^{\bullet}\to A^{\bullet}/A_{deg}^{\bullet}$ is a
quasiisomorphism. We put $A_{red}^{\bullet}=A^{\bullet}/A_{deg}^{\bullet}$. Any
morphism of simplicial groups $f:A^*\to B^*$ induces a morphism of
complexes $f:A_{red}^{\bullet}\to B_{red}^{\bullet}$.

\begin{defin}
For a scheme $X$ and an abelian sheaf $\F$ on $X$, let the {\it adelic
complex} $\A(X,\F)^{\bullet}$ be $\A_{s}(X,\F)_{red}^{\bullet}$.
\end{defin}

It follows from Corollary \ref{Properties-S} that for any sheaves
$\F$ and $\Gc$ on $X$ there is a morphism of complexes
$\A(X,\F)^{\bullet}\otimes\A(X,\Gc)^{\bullet}\to
\A(X,\F\otimes\Gc)^{\bullet}$ and that for any morphism of schemes
$f:X\to Y$ there is a morphism of complexes
$f^*:\A(Y,f_*\F)\to\A(X,\F)$. In particular, if $\Ac$ is a sheaf of
associative rings on $X$, then $\A(X,\Ac)^{\bullet}$ is a DG-ring.
Given a morphism of schemes $f:X\to Y$, we get a homomorphism of
DG-rings $\A(Y,f_*\Ac)^{\bullet}\to \A(X,\Ac)^{\bullet}$. In
addition, for any sheaf $\F$, there is a natural inclusion
$\Gamma(X,\F)\hookrightarrow H^0(\A(X,\F)^{\bullet})$.

Suppose that the scheme $X$ is Noetherian of finite dimension $d$.
Then it is easily shown that
$$
\A_{s}(X,\F)^p=\prod_{0\le i_0\le \ldots\le i_p\le d}\A((i_0\ldots i_p),\F),
$$
$$
\A(X,\F)^{p}=\prod_{0\le i_0<\ldots <i_p\le d}
\A((i_0 \ldots i_p),\F),
$$
where the expression $(i_0\ldots i_p)$ stands for the set of all flags
$\eta_0\ldots\eta_p$ on $X$ such
that for any $j,0\le j\le p$, we have $\codim(\eta_j)=i_j$. We say
that such flags {\it are of type} $(i_0\ldots i_p)$. For example,
$\A(X,\F)^0=\prod\limits_{0\le p\le d}\A((p),\F)$
and $\A((p),\F)=C((p),\F)=\prod\limits_{\eta\in X^{(p)}}\F_{\eta}$.
Thus the adelic complex is bounded and has length $d$. In fact, the adelic
complex is bounded for any Noetherian scheme $X$ of finite Krull dimension
by the maximal dimension of the irreducible components of $X$.

We may sheafify the above construction. Namely for any scheme $X$,
an abelian sheaf $\F$ on $X$, and a subset $M\subset S(X)_p$
there is an abelian presheaf $\underline{\A}(M,\F)^*$ defined by the formula
$U\mapsto\A_s(M(U),i^*_U\F)$ (see Corollary \ref{open-local}$(i)$).

\begin{prop}
If the scheme $X$ is Noetherian, then the presheaf $\underline{\A}_s(X,\F)$ is actually a
flasque sheaf.
\end{prop}
\begin{proof}
The flasqueness of this presheaf follows from Corollary \ref{open-local}$(i)$.
Clearly, it is enough to prove the sheaf property for the case of a finite open
covering $\cup_{\alpha}V_{\alpha}$ of $X$. In this case we proceed by induction
on $p$, using that for any collection of open subsets $U_{\alpha}\subset
V_{\alpha}$ containing a fixed point $\eta\in X$, the open subset
$\cap_{\alpha}U_{\alpha}\subset X$ also contains $\eta$.
\end{proof}

Thus if $X$ is Noetherian, then we get
the flasque cosimplicial abelian sheaf $\underline{\A}_s(X,\F)^*$ and
the complexes of flasque sheaves $\underline{\A}_s(X,\F)^{\bullet}$,
$\underline{\A}_s(X,\F)_{deg}^{\bullet}$, and $\underline{\A}(X,\F)^{\bullet}$.
Moreover, there is a morphism of complexes $\F\to \underline{\A}(X,\F)^{\bullet}$,
where $\F$ is considered as a complex concentrated in the zero term.

\begin{quest}
Under which conditions on $\F$ the complex of sheaves
$\underline{\A}(X,\F)^{\bullet}$ is a flasque resolution of $\F$?
\end{quest}

\begin{rmk}\label{commute-adelic}
\hspace{0cm}
\begin{itemize}
\item[(i)]
Let $\F$, $\Gc$ be two sheaves on a Noetherian scheme $X$; then the composition
of the morphisms of complexes $\F\otimes\Gc\to\underline{\A}(X,\F)^{\bullet}\otimes
\underline{\A}(X,\Gc)^{\bullet}\to
\underline{\A}(X,\F\otimes\Gc)^{\bullet}$ is the natural map described above.
\item[(ii)]
Let $f:X\to Y$ be a morphism of Noetherian schemes, $\F$ be a sheaf on $X$;
then the composition of the morphisms of complexes $f_*\F\to
\underline{\A}(Y,f_*\F)^{\bullet}\to f_*\underline{\A}(X,\F)^{\bullet}$ coincides
with the value of the functor $f_*$ at the natural morphism
$\F\to\underline{\A}(X,\F)^{\bullet}$ described above.
\end{itemize}
\end{rmk}

In particular, if $\Ac$ is a sheaf of associative rings, then the map
$\bigoplus\limits_{i\ge 0}H^{i}(X,\Ac)\to
\bigoplus\limits_{i\ge 0}H^{i}(\A(X,\F)^{\bullet})$ is a homomorphism of rings.

We will use the following explicit description of adelic
groups on Noetherian schemes. Let us introduce the following notation.

\begin{defin}\label{defin-Z(.)}
Let $Z$ be a closed subscheme in a Noetherian scheme $X$ and $\eta$
be a schematic point in $X$; then by $Z(\eta)\subset X^{(1)}$ denote
the set of irreducible reduced divisors on $X$ that are contained in
$Z$ and pass through $\eta$. Analogously, for a closed subset
$W\subset X$, by $Z(W)\subset X^{(1)}$ denote the set of irreducible
reduced divisors on $X$ that are contained in $Z$ and contain $W$.
\end{defin}

\begin{prop}\label{systemofd}
For any subset $M\subset S(X)_p$, we have
$$
\A(M,\F)=
\lim\limits_{\longrightarrow\atop\{D_{\eta_0\ldots\eta_k}\}}\,
\prod_{(\eta_0\ldots\eta_p)\in M}(\F_{X\backslash
D_{\eta_0\ldots\eta_{p-1}}})_{\eta_p},
$$
where the limit is taken over the set of systems $\{D_{\eta_0\ldots\eta_k}\}$, $0\le k< p$
of effective, reduced, possibly reducible divisors on $X$ parameterized
by flags $\eta_0\ldots\eta_k$ that can be extended to the right
to a flag from $M$ with the following property. For any $k,0< k< p$ and for any
``left part'' $(\eta_0\ldots\eta_k)$ of a flag from $M$, we have
$$
D_{\eta_0\ldots\eta_{k-1}}(\eta_{k})\supseteq
D_{\eta_0\ldots\eta_{k-1}\eta_{k}}(\eta_{k})\eqno (*)
$$
and $D_{\eta_0}(\eta_0)=\emptyset$. The partial order on the systems
$\{D_{\eta_0\ldots\eta_k}\}$, $0\le k<p$
is given by the flag wise embedding of divisors in $X$.
\end{prop}

\begin{proof}
Using Remark \ref{limit-prod}, one proves by induction on $p$ that
$$
\A(M,\F)=\lim\limits_{\longrightarrow\atop\{U_{\eta_0\ldots\eta_k}\}}
\prod_{(\eta_0\ldots\eta_p)\in M}
(\F_{U_{\eta_0}\cap\ldots\cap U_{\eta_0\ldots\eta_{p-1}}})_{\eta_p},
$$
where the limit is taken over the set of systems
$\{U_{\eta_0\ldots\eta_k}\}$, $0\le k<p$ of open subsets in $X$ parameterized
by ``left parts'' of flags from $M$ such that for any ``left part''
$(\eta_0\ldots\eta_k)$ of a flag from $M$, we have $\eta_k\in U_{\eta_0\ldots\eta_k}$.
The partial order on the systems $\{U_{\eta_0\ldots\eta_k}\}$,
$0\le k<p$ is given as the inverse to
the flag wise embedding of open subsets in $X$.

Since the scheme $X$ is Noetherian, enlarging the complement
$X\backslash U_{\eta_0\ldots\eta_k}$, we may assume that this complement is reduced,
has pure codimension one, and does not contain $\eta_k$.
Finally, we put $D_{\eta_0}=X\backslash U_{\eta_0}$ and
$D_{\eta_0\ldots\eta_k}=D_{\eta_0\ldots\eta_{k-1}}\cup X\backslash U_{\eta_0\ldots\eta_k}$
for $1\le k<p$.
\end{proof}

\begin{claim}\label{integrality}
Condition $(*)$ implies that
$$
D_{\eta_0\ldots\eta_k}(\eta_{j})\supseteq
D_{\eta_0\ldots\eta_k\ldots\eta_{k+l}}(\eta_{j}) $$ for any $0\le
j\le k+1$, $0\le l\le p-k-1$.
\end{claim}
\begin{proof}
It is enough to show this for $l=1$. Each irreducible divisor
$D\in D_{\eta_0\ldots\eta_k,\eta_{k+1}}(\eta_{j})$ contains
$\eta_{k+1}$ and thus belongs to
$D_{\eta_0\ldots\eta_k}(\eta_{k+1})$. Moreover, since $D$ contains
$\eta_j$, we see that $D$ also belongs to
$D_{\eta_0\ldots\eta_k}(\eta_{j})$.
\end{proof}

\subsection{Relation with the Cousin complex}\label{Cousin-relation}

Let $X$ be a Noetherian catenary scheme such that all irreducible
components of $X$ have the same finite dimension $d$. For a closed
subset $Z\subset X$ and a sheaf $\F$ on $X$, denote by $\gamma_Z\F$
the sheaf on $Z$ such that if $U\subset X$ is an open subset, then
$(\gamma_Z\F)(U\cap Z)$ consists of all sections in $\F(U)$ with
support on $U\cap Z$. Thus, $R\gamma_Z=i_Z^!$. One can also apply
functors $R^p\gamma_Z$ to complexes of sheaves on $X$. Following the
notations from \cite{Har}, for any sheaf $\F$ on $X$ and any point
$\eta\in X^{(p)}$, we put $H_{\eta}^p(X,\F)= (R^p
\gamma_{\overline{\eta}}\F)_{\eta}$. Let
$\nu_{\eta\xi}:H^p_{\eta}(X,\F)\to H^{p+1}_{\xi}(X,\F)$ be the
natural map defined for any two points $\eta,\xi\in X$ such that
$\xi\in\overline{\eta}$ and $\overline{\xi}$ has codimension one in
$\overline{\eta}$ (see \cite{Har}). Let $Cous(X,\F)^{\bullet}$ be
the {\it Cousin complex} of $\F$ on $X$, i.e.,
$$
Cous(X,\F)^p=\mbox{$\bigoplus\limits_{\eta\in X^{(p)}}H_{\eta}^p(X,\F)$},
$$
where the differential is the sum of the maps $\nu_{\eta\xi}$.
Let us sheafify the Cousin complex, namely, consider the complex of sheaves
$\underline{Cous}(X,\F)^{\bullet}$ given by
$\underline{Cous}(X,\F)^p=
\bigoplus\limits_{\eta\in X^{(p)}}(i_{\overline{\eta}})_*H^p_{\eta}(X,\F)$, where for
each point $\eta\in X^{(p)}$ we consider $H^p_{\eta}(X,\F)$ as a constant sheaf
on $\overline{\eta}$.
There is a natural map of complexes $\F\to\underline{Cous}(X,\F)^{\bullet}$, where
we consider $\F$ as a complex concentrated in the zero term.

Let $0\le i_0<\ldots< i_p$ be a strictly increasing sequence of natural numbers; then
the {\it depth} of $(i_0\ldots i_p)$ is the maximal natural number $l\ge 0$ such that
$(i_0\ldots i_l)=(0\ldots l)$ if $i_0=0$. Otherwise, the depth of $(i_0\ldots i_p)$
equals $-1$.

\begin{prop}\label{residue}
Let $(i_0\ldots i_p)$ be a strictly increasing sequence of natural numbers such that
$i_p\le d$. Suppose that the depth of $(i_0\ldots i_p)$ is $l\ge 0$; then there
exists a natural map
$$
\nu_{0\ldots l}\colon
\A((01\ldots l i_{l+1} \ldots i_p),\F)\to
\mbox{$\bigoplus\limits_{\eta\in X^{(l)}}
\A((0(i_{l+1}-l)\ldots (i_p-l)),R^l\gamma_{\overline{\eta}}
\underline{Cous}(X,\F)^{\bullet})$},
$$
where the right hand side is the direct sum of adelic groups on $\eta\in
X^{(l)}$. Moreover, for any adele
$f\in \A((01\ldots l i_{l+1} \ldots i_p),\F)$ and any flag
$\eta_l\eta_{i_{l+1}}\ldots\eta_{i_p}$ of type $(li_{l+1}\ldots i_p)$ on $X$, we have
$$
\theta(\nu_{0\ldots l}(f))_{\eta_l\eta_{i_{l+1}}\ldots\eta_{i_p}}=
\sum_{\eta_0\ldots\eta_l}
(\nu_{\eta_{l-1}\eta_l}\circ\ldots\circ\nu_{\eta_0\eta_1})
(\theta(f)_{\eta_0\ldots\eta_l\eta_{i_{l+1}}\ldots\eta_{i_p}}),
$$
where the sum in the left hand side is actually finite.
\end{prop}
\begin{proof}
The proof is by induction on $l$. For $l=0$, by Proposition
\ref{open-local}$(iii),(iv),(v)$, we have the natural map
$$
\A((0i_1\ldots i_p),\F)=
\mbox{$\bigoplus\limits_{\eta\in X^{(0)}}\A((0i_1\ldots i_p)(\overline{\eta}),
i_{\overline{\eta}}^*\F)$}= \mbox{$\bigoplus\limits_{\eta\in
X^{(0)}}\A((0i_1\ldots i_p),\gamma_{\overline{\eta}}\F)$}\to
$$
$$
\to\mbox{$\bigoplus\limits_{\eta\in
X^{(0)}}\A((0i_1\ldots i_p),\gamma_{\overline{\eta}}
\underline{Cous}(X,\F)^{\bullet})$}.
$$

Further, note that for the composition of
closed embeddings $Z'\subset Z\subset X$, where $Z$ has pure codimension $l$
in $X$ and $Z'$ has pure codimension one in $Z$, the natural morphism of complexes
$$
R^l\gamma_Z \underline{Cous}(X,\F)^{\bullet}\to \gamma_Z
\underline{Cous}(X,\F)^{\bullet}[l]
$$
induces the morphism of sheaves
$$
R^{1}\gamma_{Z'Z}(R^l\gamma_Z\underline{Cous}(X,\F)^{\bullet})
\to R^1\gamma_{Z'Z}(\gamma_Z\underline{Cous}(X,\F)^{\bullet}[l])=
R^{l+1}\gamma_{Z'}\underline{Cous}(X,\F)^{\bullet}.
$$
Therefore, to prove the proposition by
induction on $l$, it is enough to consider the case $l=1$ and $X$ is irreducible.
Recall that for any closed subscheme $D\subset X$, there is a morphism
of sheaves $\F_U\to (i_D)_*R^1\gamma_{D}\F$, where $U=X\backslash D$.
Hence, using the same argument as for the case $l=0$, we get the map
$$
\A((01i_2\ldots i_p),\F)=\lim\limits_{\longrightarrow\atop U}\A((1i_2\ldots
i_p),\F_U)\to\lim\limits_{\longrightarrow\atop D=X\backslash U}\A((1i_2\ldots i_p),
(i_D)_*R^1\gamma_{D}\F)=
$$
$$
=\lim\limits_{\longrightarrow\atop D}\A(0(i_2-1)\ldots (i_p-1)),
R^1\gamma_{D}\F)=\mbox{$\bigoplus\limits_{\eta\in X^{(1)}}\A((0(i_2-1)\ldots (i_p-1)),
R^1\gamma_{\overline{\eta}}\F)$}\to
$$
$$
\to\mbox{$\bigoplus\limits_{\eta\in X^{(1)}}\A((0(i_2-1)\ldots (i_p-1)),
R^1\gamma_{\overline{\eta}}\underline{Cous}(X,\F)^{\bullet})$},
$$
where the second limit is taken over all closed subschemes $D\subset X$ of
pure codimension one.
\end{proof}

\begin{rmk}
It seems that it is impossible to replace in the formulation of
Proposition \ref{residue} the sheaf $R^l\gamma_Z\underline{Cous}(X,\F)^{\bullet}$
by a more natural sheaf $R^l\gamma_Z\F$. At least the induction step in the above
proof will not be valid, because in general there is no map
$R^1\gamma_{Z'Z}(R^l\gamma\F)\to R^{l+1}\gamma_{Z'}\F$
in notations from the proof of Proposition \ref{residue}.
\end{rmk}

\begin{examp}\label{resexamp}
Suppose that $l=p$; then we get the map $\nu_p=\nu_{0\ldots p}\colon
\A((0\ldots p),\F)\to \mbox{$\bigoplus\limits_{\eta\in
X^{(p)}} H_{\eta}^p(X,\F)$}$.
\end{examp}

There is a morphism of complexes
$$
\nu_X:\A(X,\F)^{\bullet}\to Cous(X,\F)^{\bullet}
$$
that is equal to the map $(-1)^{\frac{p(p+1)}{2}}\nu_p$
on the $(0\ldots p)$-type components of the adelic
complex and equals zero on all the other components of the adelic complex.

Also, for any two sheaves $\F$ and $\Gc$ on $X$ and a point $\eta\in X^{(p)}$,
we have the natural morphism
$$
H^p_{\overline{\eta}}(X,\F)\otimes \A(X,\Gc)^q=
\A((p),(i_{\overline{\eta}})_*R^p\gamma_{\overline{\eta}}\F)\otimes\A(X,\Gc)^q\to
\A(\overline{\eta},R^p\gamma_{\overline{\eta}}(\F\otimes\Gc))^q
\stackrel{\nu_{\overline{\eta}}}\longrightarrow
$$
$$
\stackrel{\nu_{\overline{\eta}}}\longrightarrow
\mbox{$\bigoplus\limits_{\xi\in \overline{\eta}^{(q)}}
H^q_{\xi}(\overline{\eta},R^p\gamma_{\overline{\eta}}(\F\otimes\Gc))$}\subset
\mbox{$\bigoplus\limits_{\xi\in X^{(p+q)}}
H^{p+q}_{\xi}(X,\F\otimes\Gc)$}.
$$
It is easily checked that this defines a morphism of complexes
$$
\mu:Cous(X,\F)^{\bullet}\otimes \A(X,\Gc)^{\bullet}\to
Cous(X,\F\otimes\Gc)^{\bullet}
$$
given by the formula
$$
\mu(f\otimes g)_{\eta}=
(-1)^{\epsilon(p,q)}\sum_{\eta_0\ldots\eta_{q-1}}(\nu_{\eta_{q-1}\eta}\circ\ldots\circ\nu_{\eta_0\eta_1})
(f_{\eta_0}\cdot \theta(g)_{\eta_0\ldots\eta_{q-1}\eta})
$$
for any
$f\in Cous(X,\F)^p$, $g\in\A(X,\Gc)^q$, and $\eta\in X^{(p+q)}$,
where the sum is taken over all flags
$\eta_0\ldots\eta_{q-1}$ of type $(p,p+1\ldots,p+q-1)$
such that $\eta\in\overline{\eta}_{q-1}$ and
$f_{\eta_0}\cdot\theta(g)_{\eta_{0}\ldots\eta_{q-1}\eta}\in
H^p_{\eta}(X,\F\otimes\Gc)$, and $\epsilon(p,q)=pq+\frac{q(q+1)}{2}$.

\begin{rmk}
The analogous product is well defined of one replaces $\F$ by a complex of
abelian sheaves $\F^{\bullet}$.
\end{rmk}

\begin{rmk}
A coherent version of the product between the Cousin and the adelic
complex was considered in \cite{Yek}.
\end{rmk}

\begin{examp}
Multiplication of the adelic complex on the right by $1\in\Z=\underline{\Z}(X)$
coincides with the morphism $\nu_X$.
\end{examp}

In particular, if $\Ac$ is a sheaf of associative rings on $X$, then
$Cous(X,\Ac)^{\bullet}$ is a right DG-module over the DG-ring
$\A(X,\Ac)^{\bullet}$.

\begin{rmk}\label{surj-cohom}
Evidently, we also have the morphism of complexes of sheaves
$\underline{\nu}_X:\underline{\A}(X,\F)^{\bullet}\to
\underline{Cous}(X,\F)^{\bullet}$.
Suppose that the sheaf $\F$ on $X$ is {\it Cohen--Macaulay} in the sense of \cite{Har},
i.e., that the composition $\F\to\underline{\A}(X,\F)^{\bullet}
\to\underline{Cous}(X,\F)^{\bullet}$ is a quasiisomorphism;
then for any $i\ge 0$, the cohomology group $H^i(X,\F)$ is a direct summand in
the group $H^i(\A(X,\F)^{\bullet})$.
One may expect that the map $\F\to\underline{\A}(X,\F)^{\bullet}$
is a quasiisomorphism for any Cohen--Macaulay sheaf $\F$.
For the particular case of this statement see Theorem \ref{quasiis}.
In particular, if $\Ac$ is a Cohen--Macaulay sheaf of associative rings,
then the ring
$\bigoplus\limits_{i\ge 0}H^{i}(X,\Ac)$ is a direct summand as a ring in the
associative ring $\bigoplus\limits_{i\ge 0}H^{i}(\A(X,\Ac)^{\bullet})$.
\end{rmk}

The next statement is needed for the sequel.

\begin{lemma}\label{surjective-nu}
Suppose that $X$ is a Noetherian catenary scheme such that all irreducible
components of $X$ have the same finite dimension $d$. Let $\F$ be a
Cohen--Macaulay sheaf on $X$ (see \cite{Har});
then for any $p$, $0\le p\le d$, the map
$\nu_{p}:\bigoplus\limits_{\eta_0\ldots\eta_p}\F_{\eta_0}\to
\bigoplus\limits_{\eta\in X^{(p)}}H^p_{\eta}(X,\F)$ is surjective, where the first
direct sum is taken over all flags of type $(0\ldots p)$ on $X$.
\end{lemma}
\begin{proof}
The proof is by induction on $p$. For $p=0$, there is nothing to prove.
Suppose that $p>0$. Consider a collection $\{f_{\eta}\}\in
\bigoplus\limits_{\eta\in X^{(p)}}H^p_{\eta}(X,\F)$.
Note that for any point $\eta\in X^{(p)}$, the sheaf $j_{\eta}^*\F$ is Cohen--Macaulay
on $X_{\eta}$, where $j_{\eta}:X_{\eta}=\Spec(\OO_{X,\eta})\to X$
is the natural morphism. Hence for each point $\eta\in X^{(p)}$, there exists
a collection
$\{g_{\xi}\}_{(\eta)}\in \bigoplus\limits_{\eta\in X_{\eta}^{(p-1)}}
H^{p-1}_{\xi}(X_{\eta},j_{\eta}^*\F)$ such that $d_{\eta}\{g_{\xi}\}_{(\eta)}=
f_{\eta}$,
where $d_{\eta}$ denotes the differential in the Cousin complex on
$X_{\eta}$. We may suppose that $\{g_{\xi}\}_{(\eta)}=0$ for almost all
$\eta\in X^{(p)}$. By the induction hypothesis, for each point $\eta\in X^{(p)}$,
there exists a collection
$\{g_{\xi_0\ldots\xi_{p-1}}\}_{(\eta)}\in\bigoplus\limits_{\xi_0\ldots\xi_{p-1}}\F_{\xi_0}$ such that
$\nu_{p-1}\{g_{\xi_0\ldots\xi_{p-1}}\}_{(\eta)}=\{g_{\xi}\}_{(\eta)}$, where the direct sum is taken
over all flags of type $(0\ldots p-1)$ on $X_{\eta}$.
Again, we may suppose that $\{g_{\xi_0\ldots\xi_{p-1}}\}_{(\eta)}=0$
for almost all $\eta\in X^{(p)}$. Finally, we put
$\{f_{\eta_0\ldots\eta_p}\}=\{g_{\eta_0\ldots\eta_{p-1}}\}_{(\eta_p)}$.
\end{proof}

\subsection{Projection formula}

Let $X$, $Y$ be Noetherian catenary irreducible schemes,
$f:X\to Y$ be a morphism such
that for any point $\eta\in X$, we have
$\dim(\overline{\eta})\ge\dim(\overline{f(\eta)})$.
Under the above hypothesis, for any sheaf $\F$ on $X$,
there is a canonical morphism of complexes
$Cous(X,\F)^{\bullet}\to Cous(Y,Rf_*\F[d])^{\bullet}$,
where $d=\dim(f)=\dim(X)-\dim(Y)$. The definition of this morphism uses
inclusions of complexes
$\Gamma_{Z}(X,C(\F)^{\bullet})\hookrightarrow
\Gamma_{\overline{f(Z)}}(Y,f_*C(\F)^{\bullet})$
for any closed subset $Z\subset X$, where $C(\F)^{\bullet}$
is a flasque resolution of $\F$ on $X$. The morphism
$Cous(X,\F)^{\bullet}\to Cous(Y,Rf_*\F[d])^{\bullet}$ consists of
homomorphisms of type $f_*:H^p_{\eta}(X,\F)\to H^{p-d}_{f(\eta)}(X,Rf_*\F[d])$, where
$\dim(\overline{\eta})=\dim(\overline{f(\eta)})$.

The following adelic projection formula holds true.

\begin{prop}\label{product_formula}
Let $f:X\to Y$ be as above and let $\F$, $\Gc$ be two sheaves on $X$;
then the following natural diagram commutes up to the
sign $(-1)^{d\cdot \deg_{\A}}$,
where $\deg_{\A}$ is the degree of the components in the adelic complex:
$$
\begin{array}{ccc}
Cous(X,\F)^{\bullet}\otimes\A(Y,f_*\Gc)^{\bullet}&=&
Cous(X,\F)^{\bullet}\otimes\A(Y,f_*\Gc)^{\bullet}\\
\downarrow&&\downarrow\\
Cous(X,\F)^{\bullet}\otimes\A(X,\Gc)^{\bullet}&&
Cous(Y,Rf_*\F[d])^{\bullet}\otimes\A(Y,\Gc)^{\bullet}\\
\downarrow&&\downarrow\\
Cous(X,\F\otimes\Gc)^{\bullet}&\longrightarrow&
Cous(Y,Rf_*(\F\otimes\Gc)[d])^{\bullet}.\\
\end{array}
$$
\end{prop}
\begin{proof}
For any point $\eta\in X^{(p)}$ such that
$\dim(\overline{\eta})=\dim(\overline{f(\eta)})$, the following diagram commutes
$$
\begin{array}{ccc}
H^p_{\eta}(X,\F)\otimes (f_*\Gc)_{f(\eta)}&=&
H^p_{\eta}(X,\F)\otimes (f_*\Gc)_{f(\eta)}\\
\downarrow&&\downarrow\\
H^p_{\eta}(X,\F)\otimes \Gc_{\eta}&&
H^{p-d}_{f(\eta)}(Y,Rf_*\F[d])\otimes (f_*\Gc)_{f(\eta)}\\
\downarrow&&\downarrow\\
H^p_{\eta}(X,\F\otimes\Gc)&\longrightarrow&
H^{p-d}_{f(\eta)}(Y,Rf_*(\F\otimes\Gc)[d]),\\
\end{array}
$$
where $f(\eta)\in Y^{(q)}$.
Hence the proposition follows from Lemma \ref{sum_over_flags} and
the explicit formula for the product between
the Cousin and the adelic complexes.
\end{proof}

\begin{lemma}\label{sum_over_flags}
Let $f:X\to Y$ be as above with $\dim(X)=\dim(Y)$,
$\H$ be a sheaf on $X$, and
$(\xi_0\ldots\xi_r)$ be a flag on $Y$ such
that $\xi_l\in Y^{(l)}$ for all $l,0\le l\le r$; then for any element
$h\in \H_X$, we have
$$
f_*\left(
\sum_{\eta_0\ldots\eta_r}\nu_{\eta_0\ldots\eta_r}(h)\right)=
\nu_{\xi_0\ldots\xi_r}(f_*(h)),
$$
where $f_*:H^*_{\eta}(X,\H)\to H^{*-d}_{f(\eta)}(Y,Rf_*\H[d])$
are the natural homomorphisms and the sum is taken
over all flags $\eta_0\ldots\eta_r$ on $X$ such that
$f(\eta_0\ldots\eta_r)=(\xi_0\ldots\xi_r)$ and $\eta_l\in X^{(l)}$
for all $l,0\le l\le r$.
\end{lemma}
\begin{proof}
The proof is by induction on $r$. Note that there are only finitely
many schematic points $\eta_1\in X^{(1)}$ such that
$f(\eta_1)=\xi_1$. Hence, if we replace $X$ and $Y$ by
$\overline{\eta}_1$ and $\overline{\xi}_1$, respectively, we see that the
induction step is equivalent to the case $r=1$.
On the other hand, when $r=1$ the residue maps
$\nu_{\eta_0\eta_1}$ and $\nu_{\xi_0\xi_1}$ correspond to the
differentials in the Cousin complexes $Cous(X,\H)^{\bullet}$ and
$Cous(Y,Rf_*\H[d])^*$. Therefore
Lemma \ref{sum_over_flags} is equivalent to the fact that the
homomorphisms $f_*$ define a morphism of the corresponding Cousin complexes.
\end{proof}

\subsection{Sheaves with controllable support}\label{section-control-support}

\begin{defin}
We say that a sheaf $\F$ on a scheme $X$ {\it has controllable support} if for
any point $\eta\in X$, there exists a closed subset $Z_{\eta}\subset X$ such that
$\eta\notin Z_{\eta}$ and $\lim\limits_{\longrightarrow}(i_Z)_*\gamma_Z\F=
(i_{Z_{\eta}})_*\gamma_{Z_{\eta}}\F$,
where the limit is taken over all closed subsets $Z\subset X$ such that $\eta\notin Z$.
\end{defin}

\begin{rmk}\label{properties-controled}
\hspace{0cm}
\begin{itemize}
\item[(i)]
If $\F$ has controllable support, then for any open subset
$U\subset X$ the sheaf $\F_U$ has also controllable support (indeed,
for any closed subset $Z\subset X$ we have
$(i_Z)_*\gamma_Z(\F_U)=((i_Z)_*\gamma_Z\F)_U$;
\item[(ii)]
any subsheaf in a sheaf with controllable support has also
controllable support.
\end{itemize}
\end{rmk}

\begin{examps}
\hspace{0cm}

1) If $X$ has finitely many irreducible components, then a constant
sheaf on $X$ has controllable support.

2) If $\F$ is a coherent sheaf on a Noetherian scheme $X$, then $\F$
has controllable support. Indeed, all sheaves $(i_Z)_*\gamma_Z\F$ in
the definition are coherent subsheaves in the sheaf $\F$.

3) Any Cohen--Macaulay sheaf on a Noetherian scheme has controllable support.

4) Let $X$ be an irreducible one-dimensional scheme with infinitely many closed points.
Consider the sheaf $\bigoplus\limits_{x\in X}(i_x)_*\Z$, where $x$ ranges over all
closed points in $X$ and $i_x:\Spec k(x)\hookrightarrow X$ is the
closed embedding of a point. Then the sheaf $\bigoplus\limits_{x\in
X}(i_x)_*\Z$ does not have controllable support.
\end{examps}

\begin{claim}\label{controled}
If a sheaf $\F$ on a scheme $X$ has controllable support, then for any point
$\eta\in X$ there exists an open subset $U_{\eta}\subset X$ containing $\eta$
such that for any open subsets $V\subset U\subset U_{\eta}$ containing $\eta$
the natural morphism of sheaves $\F_{U}\to \F_V$ is injective.
\end{claim}
\begin{proof}
We have
${\rm Ker}\{\F_{U}\to \F_V\}=((i_Z)_*\gamma_Z\F)_{U}$,
where $Z=X\backslash V$. We put $U_{\eta}=X\backslash Z_{\eta}$;
since $Z_{\eta}\subset Z$ and $U\subset X\backslash Z_{\eta}$,
we get $((i_Z)_*\gamma_Z\F)_{U}=((i_{Z_{\eta}})_*\gamma_{Z_{\eta}}\F)_{U}=0$.
\end{proof}

\begin{prop}\label{theta-emb}
Suppose that the sheaf $\F$ on a scheme $X$ has controllable support;
then the map $\theta$ is injective for any subset $M\subset S(X)_p$:
$$
\theta:\A(M,\F)\hookrightarrow C(M,\F)=\prod_{(\eta_0\ldots\eta_p)\in M}\F_{\eta_0}.
$$
\end{prop}
\begin{proof}
The proof is by induction on $p$. For $p=0$, there is nothing to prove. Suppose
that $p>0$; then, by the induction hypothesis, we have
$$
\A(M,\F)=\prod_{\eta\in P(X)}\lim\limits_{\longrightarrow\atop U_{\eta}}
\A(_{\eta}M,\F_{U_{\eta}})\hookrightarrow
\prod_{\eta\in P(X)}\lim\limits_{\longrightarrow\atop U_{\eta}}
\prod_{(\eta_1\ldots\eta_p)\in_{\eta}M}(\F_{U_{\eta}})_{\eta_1}\hookrightarrow
$$
$$
\hookrightarrow \prod_{\eta\in P(X)}\prod_{(\eta_1\ldots\eta_p)\in_{\eta}M}
\lim\limits_{\longrightarrow\atop U_{\eta}}(\F_{U_{\eta}})_{\eta_1}=
\prod_{(\eta_0\ldots\eta_p)\in M}\F_{\eta},
$$
where the injectivity of the second map follows from Claim \ref{controled}.
\end{proof}

Thus when $\F$ has controllable support, each adele $f\in \A(M,\F)$ is
uniquely determined by its components $f_{\eta_0\ldots\eta_p}\in \F_{\eta_0}$, where
$(\eta_0\ldots\eta_p)$ runs over flags in $M$.

\begin{rmk}
Let $\F$ be a sheaf with controllable support on a scheme $X$; then for any subset
$M\subset S(X)_p$, there is a natural inclusion
$\bigoplus\limits_{(\eta_0\ldots\eta_p)\in M}\F_{\eta_0}\subset \A(M,\F)$.
\end{rmk}

\begin{examps}\label{simple_examples}
Let $\F$ be a sheaf with controllable support on $X$. For an irreducible
closed subscheme $Z\subset X$, by $\F_Z$ denote the stalk of $\F$ at the
generic point of $Z$.

1) Suppose that $\dim X=1$; then the adelic group $\A((01),\F)\subset
\prod\limits_{x\in X}{\F_{X}}$ consists of all collections
$\{f_{Xx}\}\in \prod\limits_{x\in X}{\F_X}$ such that
$f_{Xx}\in {\rm Im}(\F_x\to\F_X)$ for almost all $x\in X$. In particular, we get
rational adeles on a curve if $\F=\OO_X$ (see \cite{Ser}, where rational adeles
are called {\it repartitions}).

2) Suppose that $\dim X=2$. Let us describe explicitly the arising
adelic groups. The adelic group
$\A((01),\F)\subset\prod\limits_{C\subset X}\F_X$ consists of all
collection $\{f_{XC}\}$ such that $f_{XC}\in {\rm Im}(\F_C\to\F_X)$
for almost all irreducible curves $C\in X$. The adelic group
$\A((12),\F)\subset\prod\limits_{x\in C}\F_C$ consists of all
collections $\{f_{Cx}\}$ such that $f_{Cx}\in {\rm Im}(\F_x\to\F_C)$
for almost all points $x\in C$ for a fixed $C$. The adelic group
$\A((02),\F)\subset \prod\limits_{x\subset X}\F_X$ consists of all
collections $\{f_{Xx}\}$ such that there exists a divisor $D\subset
X$ such that $f_{Xx}\in {\rm Im}((\F_{X\backslash D})_{x}\to \F_X)$
for any closed point $x\in X$. The adelic group $\A((012),\F)\subset
\prod\limits_{x\in C\subset X}\F_X$ consists of all collections
$\{f_{XCx}\}$ satisfying the following condition. There exists a
divisor $D\subset X$ and for each irreducible curve $C\subset X$,
there is a divisor $D_C$ such that $D_C(C)=D(C)$ (see
Definition~\ref{defin-Z(.)}), and $f_{XCx}\in {\rm
Im}((\F_{X\backslash D_C})_{x}\to \F_X)$ for all flags $x\in
C\subset X$. This is analogous to the construction given in
\cite{Par76}, p. 751.
\end{examps}

\subsection{$1$-pure sheaves}\label{section-1-pure}

\begin{defin}
We say that a sheaf $\F$ on a Noetherian scheme $X$ is {\it $1$-pure} if for any point
$\eta\in X^{(i)}$, we have $H^0_{\eta}(X,\F)=0$ if $i\ge 1$ and
also $H^1_{\eta}(X,\F)=0$ if $i\ge 2$.
\end{defin}

Equivalently, a sheaf $\F$ is $1$-pure
if the following complex of sheaves is exact:
$$
0\to\F\to \underline{Cous}(X,\F)^0\to\underline{Cous}(X,\F)^1.
$$

\begin{rmk}\label{explicit-sheaves-pure}
For any open subset $U\subset X$, any point $\xi\in P(X)$, and a $1$-pure sheaf
$\F$ on $X$, there are exact sequences of sheaves
$$
0\to\F_U\to\mbox{$\bigoplus\limits_{\eta\in X^{(0)}}(i_{\overline{\eta}})_*
\F_{\eta}\to\bigoplus\limits_{\eta\in D(X\backslash U)}
(i_{\overline{\eta}})_*H^1_{\eta}(X,\F)$},
$$
$$
0\to[\F]_{\xi}\to\mbox{$\bigoplus\limits_{\eta\in X^{(0)}}(i_{\overline{\eta}})_*
\F_{\eta}\to\bigoplus\limits_{\eta\in D(\xi)}(i_{\overline{\eta}})_*H^1_{\eta}(X,\F)$},
$$
where $D(X\backslash U)$ is the set of codimension one points in $X$ that belong to
the complement $X\backslash U$ and $D(\xi)$ is the set of codimension one
points $\eta$ in $X$ such that $\xi\in\overline{\eta}$. In particular, the sheaves
$\F_U$ and $[\F]_{\xi}$ are $1$-pure for any $1$-pure sheaf $\F$.
\end{rmk}

\begin{rmk}
It is easily shown that for a sheaf $\F$ on a Noetherian scheme $X$,
the sheaf $\bigoplus\limits_{\eta\in X^{(0)}}(i_{\overline{\eta}})_*
\F_{\eta}$ has controllable support. Hence any $1$-pure sheaf on a
Noetherian scheme has controllable support.
\end{rmk}

For any flag $F=(\eta_0\ldots\eta_i)$ and a subset $M\subset S(X)_p$, we put $_F M=
{_{\eta_i}(_{\eta_{i-1}}(\ldots_{\eta_0}(M)))}$.

\begin{prop}\label{intersect}
Let $\F$ be a $1$-pure sheaf on a Noetherian irreducible scheme $X$.
Suppose that the subset
$M\subset S(X)_p$ and a natural number $l$, $0\le l\le p$ satisfy the following
condition: either $p\le l+1$, or for any flag $F=(\eta_0\ldots\eta_{l+1})\in
S(X)_{p+1}$,
we have $_F M=
{_{\eta_{l+1}}(\delta^{p-l}_0\ldots\delta^p_l(M))}$ or
$_F M=\emptyset$.
Then the boundary maps $C(\delta^p_l(M),\F)\to C(M,\F)$ and
$\A(\delta^p_l(M),\F)\to \A(M,\F)$, $0\le l\le p$
(see Proposition \ref{Properties-M}$(ii)$) are injective and we have
$$
\A(\delta^p_l(M),\F)=\A(M,\F)\cap C(\delta^p_l(M),\F),
$$
where the intersection is taken inside the group $C(M,\F)$.
\end{prop}
\begin{proof}
Since $X$ is irreducible, the sheaf $\F$ is a subsheaf in a constant sheaf and
for any points $\eta,\xi\in X$ and any open subset $U$ such that
$\xi\in\overline{\eta}$, $\eta\in U$, and $U$ is connected,
the natural morphisms $\F_{\xi}\to\F_{\eta}$ and $\F(U)\to \F_{\eta}$ are injective.
Thus we get the injectivity of the boundary maps for the direct product groups
and hence, by Proposition \ref{theta-emb},
we get the injectivity of the boundary maps
for the adelic groups. In particular, the left hand side of the equality is
contained in the right hand side.

To prove the backward inclusion we use induction on $p$ and $l$.
For $p=1$ and $l=0,1$, we have $\A(\delta^1_l(M),\F)=C(\delta^1_l(M),\F)$ and the
assertion is clear.

Suppose that $p>1,l=0$. Denote by $A$ the right hand side of the needed equality.
For each $\eta\in P(X)$, consider the image $A_{\eta}$ of $A$ under
the natural map $C(\delta^p_0(M),\F)\to C(_{\eta}M,\F)$,
see Proposition \ref{Properties-M}$(i)$.
The group $A_{\eta}$ coincides with the projection of $A$ on
the $\eta$-part in the direct product
$C(M,\F)=\prod\limits_{\eta\in P(X)}C(_{\eta}M,[\F]_{\eta})$ and hence
$A_{\eta}=\lim\limits_{\longrightarrow\atop U_{\eta}}
\A(_{\eta}M,\F_{U_{\eta}})\cap C(_{\eta}M,\F)$, where the intersection is taken
in the group $C(_{\eta}M,[\F]_{\eta})$.
Therefore, by Lemma \ref{intersect-sheaves-adelic}, $A_{\eta}=\A(_{\eta}M,\F)$.
Further, note that $\delta^0_p(M)=\bigcup\limits_{\eta\in P(X)}{_{\eta}M}$ and
for any point $\xi\in P(X)$ there exists $\eta\in P(X)$ such that
$_{\xi}(_{\eta}M)=_{\xi}(\delta^p_0(M))$. Therefore,
by Lemma \ref{projections-adelic}, $A=\A(\delta^p_0(M),\F)$.

Suppose that $p>1$, $l>0$; then, by
definition, the right hand side is equal to
$$
\prod_{\eta\in P(X)}\lim\limits_{\longrightarrow\atop U_{\eta}}
\A(_{\eta}M,\F_{U_{\eta}})\cap C(\delta^p_l(M),\F)=
\prod_{\eta\in P(X)}\left(\lim\limits_{\longrightarrow\atop U_{\eta}}
\A(_{\eta}M,\F_{U_{\eta}})\cap
C(_{\eta}\delta^p_{l}(M),[\F]_{\eta})\right).
$$
Since $l>0$, for each point $\eta\in P(X)$, we have $_{\eta}\delta^p_{l}(M)=
\delta^p_{l-1}(_{\eta}M)$. On the other hand, by Proposition \ref{theta-emb},
$\A(_{\eta}M,\F_{U_{\eta}})\subset C(_{\eta}M,\F_{U_{\eta}})$ and
the intersection corresponding to the point $\eta$
is actually contained in the subgroup
$\lim\limits_{\longrightarrow\atop U_{\eta}}C(\delta^p_{l-1}(_{\eta}M),\F_{U_{\eta}})
\subset C(\delta^p_{l-1}(_{\eta}M),[\F]_{\eta})$. Hence, by the inductive assumption
for $\F_{U_{\eta}}$, $_{\eta}M$, $p-1$, and $l-1$,
the intersection considered above is equal to
$$
\prod_{\eta\in P(X)}\lim\limits_{\longrightarrow\atop U_{\eta}}
\left(\A(_{\eta}M,\F_{U_{\eta}})\cap C(\delta^p_{l-1}(_{\eta}M),\F_{U_{\eta}})\right)=
\prod_{\eta\in P(X)}\lim\limits_{\longrightarrow\atop U_{\eta}}
\A(\delta^p_{l-1}(_{\eta}M),\F_{U_{\eta}})=\A(\delta^p_l(M),\F).
$$
\end{proof}

\begin{rmk}
It follows from the proof of Proposition \ref{intersect} that for any subsheaf $\F$ of
a constant sheaf on a scheme $X$ with finitely many irreducible components, for any
$M\subset S(X)_p$, and for any $l$, $0\le l\le p$, there is the inclusion
$$
\A(\delta^p_l(M),\F)\subset\A(M,\F)\cap C(\delta^p_l(M),\F).
$$
\end{rmk}

\begin{rmk}
The condition that $\F$ is $1$-pure in Proposition \ref{intersect} might be replaced by a weaker condition
but actually we do not need such improvement: all sheaves that we consider
further are $1$-pure. The same is true about the condition on the set $M$.
\end{rmk}

\begin{examp}
Suppose that $X$ is irreducible and $\dim(X)=2$. Let $D\subset X$ be a closed
irreducible codimension one subset with infinitely many
closed points. We put $G=\bigoplus\limits_{y\in D}\Z$,
where the sum is taken over all closed points $y$ on $D$. Let
$\F$ be the kernel of the natural map
$\underline{G}\to\bigoplus\limits_{y\in D}(i_y)_*\Z$ of sheaves on $X$.
Then $\F$ is a subsheaf in a constant sheaf but is not $1$-pure on $X$.
Further, consider the adele
$f\in C((12),\F)$ defined by $f_{Cx}=\{1_x\}\in G$
if $C=D$ and $f_{Cx}=0$ otherwise. For $C=D$, we have
$f_{Cx}\in (\F_{X\backslash \{x\}})_x$ and there are two inclusions
$(\F_{X\backslash \{x\}})_x\subset(\F_{X\backslash C})_x$ and
$(\F_{X\backslash \{x\}})_x\subset\F_C$.
Therefore, $f\in \A((012),\F)\cap C((12),\F)$ but $f\notin \A((12),\F)$.
\end{examp}

\begin{lemma}\label{intersect-sheaves-adelic}
Let $\F$ be $1$-pure sheaf on a Noetherian scheme $X$.
Consider the sheaf $\Gc=\F_V$ for some
open subset $V\subset X$. Then for any subset $M\subset S(X)_p$, we have
$$
\A(M,\Gc)\cap C(M,\F)=\A(M,\F),
$$
where the intersection is taken in the group $C(M,\Gc)$.
\end{lemma}
\begin{proof}
The proof is by induction on $p$. For $p=0$, there is nothing to prove.
Suppose that $p>0$; then, by definition, the left hand side is equal to
$
\prod\limits_{\eta\in P(X)}\left(\lim\limits_{\longrightarrow\atop U_{\eta}}
\A(_{\eta}M,\Gc_{U_{\eta}})\cap C(_{\eta}M,[\F]_{\eta})\right).
$

Consider a point $\eta\in P(X)$ and an open subset $U\subset X$ containing $\eta$. Suppose that $U$ is
small enough so that all irreducible components of $U\backslash V$ contain $\eta$ and the natural
morphisms $\F_U\to [\F]_{\eta}$, $\Gc_U\to [\Gc]_{\eta}$ are injective. It follows from
the explicit description given in Remark \ref{explicit-sheaves-pure} that
$\Gc_{U}\cap [\F]_{\eta}=\F_U$, where the intersection is taken inside the sheaf
$[\Gc]_{\eta}$. Therefore $C(N,\Gc_{U})\cap C(N,[\F]_{\eta})=C(N,\F_U)$
for any subset $N\subset S(X)_q$.

Consequently the intersection $\lim\limits_{\longrightarrow\atop U_{\eta}}
\A(_{\eta}M,\Gc_{U_{\eta}})\cap C(_{\eta}M,[\F]_{\eta})$ is actually contained in the
subgroup $\lim\limits_{\longrightarrow\atop U_{\eta}}C(_{\eta}M,\F_{U_{\eta}})
\subset C(_{\eta}M,[\F]_{\eta})$ and,
by the inductive hypothesis, we get the needed statement.
\end{proof}

\begin{lemma}\label{projections-adelic}
Let $\F$ be a sheaf with controllable support on a scheme $X$ and let
$N$ be a subset in $S(X)_p$ such that
$N=\bigcup\limits_{\alpha} N_{\alpha}$, where $N_{\alpha}\subset S(X)_p$
for each $\alpha$. Suppose that for any point $\eta\in P(X)$ there exists $\alpha$
such that $_{\eta}(N_{\alpha})={_{\eta}N}$. Let $A$ be a subgroup
in $C(N,\F)$ such that for any $\alpha$, the image of $A$ under the natural map
$C(N,\F)\to C(N_{\alpha},\F)$ is contained inside $\A(N_{\alpha},\F)$.
Then $A\subset \A(N,\F)$.
\end{lemma}
\begin{proof}
For a point $\eta\in P(X)$ let $\alpha$ be such that $_{\eta}(N_{\alpha})={_{\eta}N}$.
Then the projection of $A$ to the group $C(_{\eta}N,[\F]_{\eta})=
C(_{\eta}(N_{\alpha}),[\F]_{\eta})$
from the direct product $C(N,\F)=\prod\limits_{\eta\in P(X)}C(_{\eta}N,[\F]_{\eta})$
is contained in the group $\lim\limits_{\longrightarrow\atop U_{\eta}}
\A(_{\eta}(N_{\alpha}),\F_{U_{\eta}})=\lim\limits_{\longrightarrow\atop U_{\eta}}
\A(_{\eta}N,\F_{U_{\eta}})$. Thus we get that $A\subset \A(N,\F)$.
\end{proof}

Let $X$ be a Noetherian scheme.
Consider an increasing sequence of natural numbers $i_0\le \ldots \le i_p$.
Recall that $(i_0\ldots i_p)$
denotes the set of all flags from $S(X)_p$ of type $(i_0\ldots i_p)$. For any $l$,
$0\le l\le p$, we have
$\delta^p_l(i_0\ldots i_p)=(i_0\ldots \hat{\imath}_l\ldots i_p)$.
From this one deduces that the subset $(i_0\ldots i_p)\subset S(X)_p$ satisfies the condition
from Proposition \ref{intersect} for any $l$, $0\le l\le p$.
Let $(j_0\ldots j_q)$ be a subsequence in $(i_0\ldots i_p)$, $q\le p$.
There are canonical maps $\alpha:C((j_0\ldots j_q),\F)\to C((i_0\ldots i_p),\F)$ and
$\beta:\A((j_0\ldots j_q),\F)\to \A((i_0\ldots i_p),\F)$
that are compositions of the corresponding boundary maps.
By Proposition \ref{intersect}, we get the following.

\begin{corol}\label{intersect2}
Let $\F$ be a $1$-pure sheaf on a Noetherian irreducible scheme $X$.
Then for any sequences
$(i_0\ldots i_p)$ and $(j_0\ldots j_q)$ as above,
the maps $\alpha$ and $\beta$ are injective and we have
$$
\A((j_0\ldots j_q),\F)=\A((i_0\ldots i_p),\F)\cap
C((j_0\ldots j_q),\F),
$$
where the intersection is taken inside the group $C((i_0\ldots i_p),\F)$.
\end{corol}

In what follows we will always imply the inclusion $\beta$ when
comparing adelic groups with different indices given by type.

\subsection{$\A'$-adelic groups}\label{section-prime-adeles}

We introduce a new type of adelic groups. Let $X$ be a Noetherian
catenary scheme such that all irreducible components of $X$ have the same
finite dimension $d$ and let $\F$ be an abelian sheaf
on $X$.

Consider a strictly increasing sequence of natural numbers $(i_0,\ldots,i_p)$
such that $i_p\le d$. Recall that $C((i_0\ldots i_p),\F)=
\displaystyle\prod_{\eta_0\ldots\eta_p}\F_{\eta_0}$, where the product is taken
over all flags of type $(i_0\ldots i_p)$ on $X$. We put $l$ to be the depth of
$(i_0\ldots i_p)$ (see Proposition \ref{residue}).

\begin{defin}
Let the subgroup $\A'((i_0\ldots i_p),\F)\subseteq C((i_0\ldots i_p),\F)$
consist of all elements
$f\in C((i_0\ldots i_p),\F)$ such that for any number
$m,0\le m\le l$ and a flag
$\eta_{m+1}\ldots\eta_p$ of type $(i_{m+1}\ldots i_p)$ on $X$,
there are only finitely
many flags $\eta_0\ldots\eta_{m}$ of type $(0\ldots m)$ on $X$ such that
the composition of the residue maps
$(\nu_{\eta_{m-1}\eta_{m}}\circ\ldots\circ\nu_{\eta_{0}\eta_{1}})
(f_{\eta_0\ldots\eta_{m}\eta_{m+1}\ldots\eta_p})\in
H^{m}_{\eta_{m}}(X,\F)$ is not zero.
These new adelic groups are called
{\it $\A'$-adelic groups}, while the old adelic groups will be called
{\it $\A$-adelic groups}.
\end{defin}

When $l=-1$, we have $\A'((i_0\ldots i_p),\F)=C((i_0\ldots i_p),\F)$.
When $l\ge 0$, for any number $m,0\le m\le l$, there is a map
$$
\nu'_{0\ldots m}\colon \A'((i_0\ldots i_p),\F)\to
\prod_{\eta_{m+1}\ldots\eta_p}
\left(\mbox{$\bigoplus\limits_{\eta\in X^{(m)}}$} H^{m}_{\eta}(X,\F)\right),
$$
where for each flag $\eta_{m+1}\ldots\eta_p$ of type $(i_{m+1}\ldots i_p)$ on $X$,
the direct sum is taken over all
points $\eta\in X^{(m)}$ such that $\eta_{m+1}\in\overline{\eta}$
(compare with Proposition \ref{residue}).

\begin{rmk}\label{tilde-properties}
\hspace{0cm}
\begin{itemize}
\item[(i)]
By Proposition \ref{residue}, the image
$\theta(\A((i_0\ldots i_p),\F))\subset C((i_0\ldots i_p),\F)$ is contained in
$\A'(i_0\ldots i_p),\F)$.
\item[(ii)]
It is readily seen that the analogue of Corollary \ref{open-local}$(i),(ii)$
with $M=(i_0\ldots i_p)$ holds for the $\A'$-adelic groups.
\item[(iii)]
By reciprocity law, for any $j,0\le j\le p$, the boundary map
$$
d^p_j:C((i_0\ldots \hat{\imath}_j\ldots i_p),\F)\to C((i_0\ldots i_p),\F)
$$
induces the map of the corresponding $\A'$-groups;
define the {\it $\A'$-adelic complex} $\A'(X,\F)^{\bullet}$ by formula
$$
\A'(X,\F)^p=\prod_{0\le i_0<\ldots <i_p\le d}\A((i_0 \ldots i_p),\F)
$$
and with the differential induced from the complex $C(X,\F)^{\bullet}_{red}$;
we have the morphism of complexes
$\A'(X,\F)^{\bullet}\stackrel{\nu'_X}
\longrightarrow Cous(X,\F)^{\bullet}$ defined in the same way as for the
$\A$-adelic complex; the composition
$\A(X,\F)^{\bullet}\to\A'(X,\F)^{\bullet}\to Cous(X,\F)^{\bullet}$
is equal to $\nu_X$.
\item[(iv)]
Given two sheaves $\F$, $\Gc$ on $X$,
the morphism of complexes $C(X,\F)_{red}^{\bullet}\otimes
C(X,\Gc)_{red}^{\bullet}\to C(X,\F\otimes\Gc)_{red}^{\bullet}$ does {\it not} induce
a morphism of complexes $\A'(X,\F)^{\bullet}\otimes
\A'(X,\Gc)^{\bullet}\to \A'(X,\F\otimes\Gc)^{\bullet}$;
also, the analogue of Corollary \ref{intersect2} is not true for the
$\A'$-adelic groups.
\item[(v)]
We may consider the sheafified version $\underline{\A'}(X,\F)^{\bullet}$
of the $\A'$-adelic complex; there is a natural morphism of complexes
$\F\to\underline{\A}'(X,\F)^{\bullet}$.
\end{itemize}
\end{rmk}

\begin{lemma}\label{inductive-tilde}
For any natural numbers $i_0,\ldots,i_l,p$ such that
$0\le i_0<\ldots< i_l< p\le d$ and $(i_0\ldots i_l)\ne (0\ldots (p-1))$,
we have
$$
\A'((i_0\ldots i_l p),\F)=
\prod_{\eta\in X^{(p)}}\A'((i_0\ldots i_l\eta),\F),
$$
where the index $(i_0\ldots i_l \eta)$ stands for
the set of all flags $\eta_0\ldots\eta_l\eta_p$ on $X$ of type $(i_0\ldots i_l p)$
such that $\eta_p=\eta$. Also, we have
$$
\A'((0\ldots (p-1) p),\F)=
\aprod_{\eta\in X^{(p)}}{\A'((0\ldots (p-1)\eta),\F)},
$$
where the restricted product means that we consider the set of
all collections $\{f_{\eta}\}\in\prod\limits_{\eta\in X^{(p)}}
\A'((0\ldots(p-1)\eta),\F)$ such that
for almost all points $\eta\in X^{(p)}$, the $\A'$-adele
$f_{\eta}\in \A'((0\ldots (p-1)\eta),\F)\subset
\A'((0\ldots (p-1)p),\F)$ belongs
to $\Ker(\nu_{0\dots p})$.
\end{lemma}

\begin{proof}
This follows immediately from the definition of
$\A'$-adelic groups combined with the $\A'$-analogue of
Corollary \ref{open-local}$(ii)$ (see Remark \ref{tilde-properties}(ii)).
\end{proof}

The lack of the multiplicative structure is the main disadvantage of the
$\A'$-adelic complex (see Remark \ref{tilde-properties}(iv)).
Nevertheless, the main advantage of the $\A'$-adelic complex
is the following statement.

\begin{theor}\label{prime-quasiis}
Suppose that $X$ is a Noetherian
catenary scheme such that all irreducible components of $X$
have the same finite dimension $d$ and
the sheaf $\F$ on $X$ is Cohen--Macaulay in the sense of \cite{Har}.
Then the morphism $\underline{\nu}'_X:
\underline{\A}'(X,\F)^{\bullet}\to
\underline{Cous}(X,\F)^{\bullet}$ is a quasiisomorphism.
\end{theor}
\begin{corol}
Under the assumptions from Theorem \ref{prime-quasiis},
the natural morphism $\F\to
\underline{\A}'(X,\F)^{\bullet}$ is a quasiisomorphism.
\end{corol}

\begin{proof}[Proof of Theorem \ref{prime-quasiis}]
It is enough to prove that the morphism
$\nu'_U:\A'(U,\F|_U)^{\bullet}\to
Cous(U,\F|_U)^{\bullet}$ is a quasiisomorphism for any open subset $U\subset X$.
The sheaf $\F|_U$ is Cohen--Macaulay on $U$,
hence we may suppose that $U=X$.

Using several intermediate complexes, we transform the $\A'$-adelic complex
into the Cousin complex. For each number $p$, $0\le p\le d$,
consider the following complex:
$$
C_p^{\bullet}\::0\to\prod_{0\le i\le p}\A'((i),\F)\to\ldots\to
\prod_{0\le i_0<\ldots i_l\le p}\A'((i_0\ldots i_l),\F)\to\ldots\to
\A'((0\ldots p),\F)\to $$
$$
\to\mbox{$\bigoplus\limits_{\eta\in
X^{(p+1)}}H^{p+1}_{\eta}(X,\F)$}\to\ldots\to \mbox{$\bigoplus\limits_{\eta\in
X^{(d)}}H^d_{\eta}(X,\F)$}.
$$
The differential in the first part of this
complex coincides with that in the $\A'$-adelic complex, the differential in the
second part coincides with that in the Cousin complex, and the differential in
the middle $\A'((0\ldots p),\F)\to \mbox{$\bigoplus\limits_{\eta\in
X^{(p+1)}}H^{p+1}_{\eta}(X,\F)$}$ is equal to the composition of
the boundary map $\A'((0\ldots p),\F)\to \A'((0\ldots p,(p+1)),\F)$
with the map $\nu'_{0\ldots (p+1)}$.  Thus
$C_0^{\bullet}=Cous(X,\F)^{\bullet}$ and
$C_d^{\bullet}=\A'(X,\F)^{\bullet}$.

For each $p$, $1\le p\le d$, there is a natural morphism of complexes
$\varphi_p\colon C_p^{\bullet}\to C_{p-1}^{\bullet}$ that is equal
to the natural projection for
the first part of the complex $C^{\bullet}_{p}$, is equal to
$\nu'_{0\ldots p}$ for the $p$-th terms of $C_p^{\bullet}$,
and is equal to the identity maps for
the second part of the complex $C^{\bullet}_p$.
We have ${\varphi}_{1}\circ\ldots\circ{\varphi}_{d}={\nu}'_X$.
We prove by induction on $p$ that the
morphism $\varphi_p$ is actually a quasiisomorphism
for all $X$ and $\F$ as above.

For $p=0$, there is nothing to prove. Suppose that $1\le p\le d=\dim(X)$.
By Lemma \ref{surjective-nu}, the morphism $\varphi_p$ is surjective.
So we need to show that the kernel of $\varphi_p$ is an
exact complex. By construction, $\Ker(\varphi_p)^{\bullet}$ is equal to the complex
$$
0\to\A'((p),\F)\to\prod_{0\le i<p}\A'((ip),\F)\to\ldots\to
\prod_{0\le i_0<\ldots <i_l< p} \A'((i_0\ldots i_l p),\F)\to\ldots
$$
$$
\ldots\to\prod_{0\le i< p}\A'((0\ldots\hat{\imath}\ldots p),\F) \to
\Ker(\nu'_{0\ldots p})\to 0.
$$
For each schematic point $\eta\in X^{(p)}$, consider the natural morphism
$j_{\eta}:X_{\eta}=\Spec(\OO_{X,\eta})\to X$.
Note that $j_{\eta}^*\F$ is Cohen--Macaulay on $X_{\eta}$. Hence,
by the induction hypothesis, the following complex is exact:
$$
0\to \F_{\eta}\to\prod_{0\le i<p}
\A'((i),j_{\eta}^*\F)\to\ldots\to \prod_{0\le i_0<\ldots<i_l< p}
\A'((i_0\ldots i_l),j_{\eta}^*\F)\to\ldots
$$
$$
\ldots\to\A'((0\ldots (p-1)),j_{\eta}^*\F)\to
H^p_{\eta}(X,\F)\to 0.
$$

Now we take the product of these complexes over all points $\eta\in X^{(k)}$.
We get the exact complex
$$
B_p^{\bullet}:\:0\to \prod_{\eta\in
X^{(p)}}\F_{\eta}\to\ldots\to\prod_{\eta\in X^{(p)}}
\left(\prod_{0\le i_0<\ldots<i_l< p}
\A'((i_0\ldots i_l\eta),j^*_{\eta}\F)\right)\to \ldots
$$
$$
\ldots\to
\aprod_{\eta\in X^{(p)}}
{\A'((0\ldots (p-1)\eta),j_{\eta}^*\F)}
\stackrel{\nu'_{0\ldots p}}\longrightarrow
\mbox{$\bigoplus\limits_{\eta\in X^{(p)}}H^p_{\eta}(X,\F)$}\to 0,
$$
where
the restricted product is taken in the same sense as in Lemma
\ref{inductive-tilde} and, as above, the index $(i_0\ldots i_l \eta_p)$ means that we
consider
the set of all flags $\eta_0\ldots\eta_l\eta_p$ on $X$ of type $(i_0\ldots i_l p)$
with fixed $\eta_p$. Finally, by Lemma \ref{inductive-tilde}, we see that
$\Ker(\varphi_p)^{\bullet}=\tau_{\le p}(B_p^{\bullet})$, where
$\tau_{\le p}$ is the canonical truncation of a complex, and thus the complex
$\Ker(\varphi_p)^{\bullet}$ is exact.
\end{proof}

The following technical result is needed for the sequel. For any adele $h$
and an increasing sequence of natural numbers $(j_0\ldots j_q)$,
by $h_{j_0\ldots j_q}$ denote the component of $h$ that has type $(j_0\ldots j_q)$.

\begin{lemma}\label{cobound}
Under the assumptions from Theorem \ref{prime-quasiis}, consider an increasing
sequence $(0\ldots li_{l+1}\ldots i_p)$ of depth $l$ and an adele
$g\in\A'((0\ldots li_{l+1}\ldots i_p),\F)$
such that $\nu'_{0\ldots(l+1)}(g)=0$. Then there exists an adele
$h\in \prod\limits_{0\le i\le l}
\A'((0\ldots \hat{\imath}\ldots li_{l+1}\ldots i_p),\F)$ such that
$(dh)_{0\ldots li_{l+1}\ldots i_p}=g$, where $d$ is the differential in the
$\A'$-adelic complex.
\end{lemma}

\begin{proof}
We use notations from the proof of Theorem \ref{prime-quasiis}.
Fix a flag $\eta_{l+1}\ldots\eta_p$ of type
$(i_{l+1}\ldots i_p)$. We have $g_{0\ldots l\eta_{l+1}\ldots\eta_p}\in
\A'((0\ldots l),j^*_{\eta_{l+1}}\F)$. By the
condition of the lemma, $g_{0\ldots l\eta_{l+1}\ldots\eta_p}$ is a degree $l$
cocycle in the complex $C_l^{\bullet}$
constructed for the local scheme $X_{\eta_{l+1}}$ and the sheaf $j^*_{\eta_{l+1}}\F$.
It follows from the proof of Theorem \ref{quasiis} that
the complex $C_l^{\bullet}$ is exact for $X_{\eta_{l+1}}$ and $j^*_{\eta_{l+1}}\F$.
Therefore there exists an adele $h_{\eta_{l+1}\ldots\eta_p}\in
\prod\limits_{0\le i\le l}\A'((0\ldots\hat{\imath}\ldots l),j^*_{\eta_{l+1}}\F)$
such that
$d_{\eta_{l+1}}(h_{\eta_{l+1}\ldots\eta_p})_{0\ldots l}=
f_{0\ldots l\eta_{l+1}\ldots\eta_p}$,
where $d_{\eta_{l+1}}$ is the differential in the $\A'$-adelic complex
on $X_{\eta_{l+1}}$.
By Lemma \ref{inductive-tilde}, the collection
$$
h=\{h_{\eta_{l+1}\ldots\eta_p}\}\in
\prod_{\eta_{l+1}\ldots\eta_p}\prod_{0\le i\le l}
\A'((0\ldots\hat{\imath}\ldots l),j^*_{\eta_{l+1}}\F)
$$
belongs to the
$\A'$-adelic group $\prod\limits_{0\le i\le l}
\A'((0\ldots\hat{\imath}\ldots l i_{l+1}\ldots i_p),\F)$ and
$h$ satisfied the needed condition.
\end{proof}

\section{Adeles for homology sheaves}\label{main_theorem}

We give some particular example of a class of sheaves
on a smooth variety $X$ over a field such that for any sheaf $\F$ from this class
the morphism of complexes of sheaves
$\F\to\underline{\A}(X,\F)^{\bullet}$ is a quasiisomorphism, i.e.,
the adelic complex is in fact a flasque resolution for the sheaf $\F$.
This class of sheaves naturally arises from homology theories.
In addition, these sheaves are Cohen--Macaulay in the sense of \cite{Har}
(see Corollary \ref{corol-Gerst}$(i)$).

\subsection{Homology theories}\label{homology-theories}

Let $k$ be a field and $\V_k$ be the category of varieties over $k$.

\begin{defin}
A {\it weak homology theory over $k$} is a presheaf on $\V_k$ in the Zariski topology
with value in the category of graded abelian groups
$$
\mbox{$X\mapsto \bigoplus\limits_{n\in\Z}F_n(X)$},\quad
(j:U\hookrightarrow X) \mapsto(j^*:F_n(X)\to F_n(U))
$$
such that for any closed
embedding $f:X'\hookrightarrow X$, there is a functorial
homomorphism of graded abelian groups $f_*:F_n(X')\to F_n(X)$
satisfying the following axioms:
\begin{itemize}
\item[(WH1):]
for any open embedding $j:U\hookrightarrow X$ and a closed embedding
$f:X'\hookrightarrow X$, the following diagram commutes:
$$
\begin{array}{ccc}
F_n(U')&\stackrel{(j')^*}\longleftarrow&F_n(X')\\
\downarrow\lefteqn{f_*}&&\downarrow\lefteqn{f_*}\\
F_n(U)&\stackrel{j^*}\longleftarrow&F_n(X),\\
\end{array}
$$
where $j':U'=f^{-1}(U)\hookrightarrow X'$ is an open embedding;
\item[(WH2):]
for any closed embedding $i:Z\hookrightarrow X$
there exists a long exact localization sequence
$$
\ldots\to F_n(Z)\stackrel{i_*}\longrightarrow F_n(X)\stackrel{j^*}\longrightarrow
F_n(X\backslash Z)\stackrel{\partial_{XZ}}\longrightarrow F_{n-1}(Z)\to\ldots,
$$
where $j:X\backslash Z\hookrightarrow X$ is an open embedding; in addition,
for any for any closed embedding $f:X'\hookrightarrow X$ and any
pair of closed subsets $i:Z\hookrightarrow X$, $i':Z'\hookrightarrow X'$ such that
$f(Z')\subseteq Z$,
the following diagram commutes:
$$
\begin{array}{ccccccccc}
\ldots\to& F_n(Z')&\stackrel{i'_*}\longrightarrow&F_n(X')&\stackrel{(j')^*}
\longrightarrow
&F_n(X'\backslash Z')&\stackrel{\partial_{X'Z'}}\longrightarrow &
F_{n-1}(Z')&\to\ldots\\
&\downarrow\lefteqn{f_*}&&\downarrow\lefteqn{f_*}&&\downarrow\lefteqn{f_*\alpha^*}&
&\downarrow\lefteqn{f_*}&\\
\ldots\to& F_n(Z)&\stackrel{i_*}\longrightarrow&F_n(X)&\stackrel{j^*}\longrightarrow
&F_n(X\backslash Z)&\stackrel{\partial_{XZ}}\longrightarrow &F_{n-1}(Z)&\to\ldots,\\
\end{array}
$$
where $\alpha:f^{-1}(X\backslash Z)\hookrightarrow X\backslash Z$ is an open embedding.
\end{itemize}
\end{defin}

Note that this is a modified version of the notion of a twisted
homology theory from~\cite{BO}.

Let $F_*$ be a weak homology theory over a field $k$. For an irreducible variety
$X$ over $k$ and $n\in\Z$, we put
$F_n(k(X))=\lim\limits_{\longrightarrow\atop U}F_n(U)$,
where the limit is taken over all open non-empty subsets $U\subset X$.
Evidently, this definition is correct, i.e., $F_n(k(X))$ depends only on
the birational class of $X$.

The same reasoning as in \cite{BO}, Proposition 3.7 shows that
for any variety $X$ over $k$, there is a
homological type spectral sequence
$E_{p,q}^1(X,F_*)=\bigoplus\limits_{\eta\in X_{(p)}}F_{p+q}(k(\eta))
\Rightarrow F_{p+q}(X)$, where $X_{(p)}$ is the set of all points $\eta$ on $X$
such that $\dim(\overline{\eta})=p$. The corresponding ascending filtration on
$F_{n}(X)$ is defined by
${\rm Im}(\lim\limits_{\longrightarrow\atop Z\in X_{\le p}}
F_{n}(Z)\to F_{n}(X))\subseteq F_{n}(X)$,
where $X_{\le p}$ is the set of all closed subsets $Z$ in $X$
of dimension at most $p$. For $p,q\in \Z$, we put
$\G(X,F_*,q)_{p}=E^1_{p,q}(X,F_*)$. Thus $\G(X,F_*,q)_{\bullet}$ is a
homological type complex.

Let us say that a variety $X$ is
{\it equidimensional} if all irreducible components of $X$ have the same
dimension. For an equidimensional variety $X$ of dimension $d$, we put
$\G(X,F_*,n)^{p}=E^1_{d-p,n-d}(X,F_*)=
\bigoplus\limits_{\eta\in X^{(p)}}F_{n-p}(k(\eta))$.
The cohomological type complex $\G(X,F_*,n)^{\bullet}$
is called the {\it Gersten complex}
associated to the weak homology theory $F_*$. Given a collection
$\{f_{\eta}\}\in \bigoplus\limits_{\eta\in X^{(p)}}F_{n-p}(k(\eta))$, the set of
all schematic points $\eta\in X^{(p)}$ such that $f_{\eta}\ne 0$ is called the
{\it support of} the collection $\{f_{\eta}\}$.

It is readily seen that for any $q\in\Z$
the functors $F(q)_n(X)=H_n(\G(X,F_*,q)_{\bullet})$, $n\in\Z,n\ge 0$
also form a weak homology theory. For an irreducible variety $X$ of dimension $d$,
we have $F(q)_d(k(X))=F_{d+q}(k(X))$ and $F(q)_n(k(X))=0$ for $n\ne d$. Therefore,
$\G(X,F(q)_*,0)_{\bullet}=\G(X,F_*,q)_{\bullet}$ and
$\G(X,F(q)_*,m)_{\bullet}=0$ if $m\ne 0$.

\begin{defin}
For a weak homology theory $F_*$,
the {\it homology sheaves} $\F_n,n\in\Z$ are the sheaves on $\V_k$
in the Zariski topology associated to the presheaves $F_n,n\in\Z$.
\end{defin}

We put $\F^X_n$ to be the restriction of the sheaf $\F_n$ to the
variety $X$ if we need to distinguish sheaves on different spaces.
Thus for an irreducible variety $X$, we have $(\F_n^X)_X=F_n(k(X))$.
Note that for an open subset $i_U:U\hookrightarrow X$, we have
$(i_U)^*\F_n^X=\F_n^U$. We also denote by $\F_n^U$ the sheaf
$(\F^X_n)_U=(i_U)_*\F^U_n$ on $X$.

It is readily seen that for any irreducible variety $X$ of dimension
$d$ and any $q\in\Z$, the sheaf $\F(q)^X_d$ is $1$-pure (see Section
\ref{section-1-pure}).

\begin{rmk}
Suppose that $X$ is an equidimensional variety of dimension $d$ over
$k$, $\eta$ is a schematic point on $X$, and $D\subset X$ is a
divisor; then any element $f\in (\F(n)^{X\backslash D}_d)_{\eta}$,
$n\in\Z$, is equal to the restriction of an element from
$F_{n+d}(X_{\eta}\backslash (D\cup R))$, where $R$ is a closed
subset in $X$ such that all irreducible components of $R$ have
codimension at least two in $X$. This follows directly from
definitions and the localization sequence.
\end{rmk}

For an equidimensional variety $X$,
we may also consider the sheafified Gersten
complex $\underline{\G}(X,F_*,n)^{\bullet}$, namely
$\underline{\G}(X,F_*,n)^{p}=\bigoplus\limits_{\eta\in X^{(p)}}
(i_{\overline{\eta}})_*F_{n-p}(k(\eta))$, where for each point $\eta\in X^{(p)}$,
we consider $F_{n-p}(k(\eta))$ as a constant sheaf on $\overline{\eta}$.
There is a morphism of complexes of sheaves
$\F^X_n\to\underline{Cous}(X,\F^X_n)^{\bullet}\to \underline{\G}(X,F_*,n)^{\bullet}$,
where $\F^X_n$ is considered as a complex concentrated in the zero term. Also, we have
$H^q(\underline{\G}(X,F_*,n)^{\bullet})=\F(n-d)^X_{d-q}$ and there is
a natural morphism of sheaves $\F^X_n\to \F(n-d)^X_{d}$, where $d=\dim(X)$.

For any equidimensional closed subset $Z\subset X$ of codimension
$p$ in $X$, we have $\gamma_Z\underline{\G}(X,F_*,n)^{\bullet}=
\underline{\G}(Z,F_*,n-p)^{\bullet}[-p]$. In particular,
$R^q\gamma_Z\underline{\G}(X,F_*,n)=\F(n-d)^Z_{d-q}$, where $d=\dim(X)$. Also,
there are natural morphisms of sheaves $\F^Z_{n-p}\to
R^p\gamma_Z\underline{\G}(X,F_*,n)=\F(n-d)^Z_{d-p}$
and $R^p\gamma_Z\F^X_n\to R^p\gamma_Z\underline{\G}(X,F_*,n)$. However,
in general there is no natural morphism between the sheaves
$\F^Z_{n-p}$ and $R^p\gamma_Z\F^X_n$.

\begin{defin}
We say that a weak homology theory $F_*$ over a field $k$ is a
{\it homology theory locally acyclic in fibrations} (l.a.f. homology theory)
if the following two conditions
are satisfied:
\begin{itemize}
\item[(H):]
for any finite morphism $f:X'\to X$, there is a functorial homomorphism of graded
abelian groups $f_*:F_n(X')\to F_n(X)$ extending the one
for closed embeddings and such that the axioms (WH1) and (WH2)
are satisfied with $f$ being a finite morphism;
\item[(LAF):]
if for a closed embedding $i:Z\hookrightarrow X$ and a point
$\eta\in Z$ there exists a morphism
$\pi:X\to Z$ such that $\pi\circ i=id_Z$ and $\pi$ smooth at $\eta$, then
there exists an open subset $U\subset X$
containing $\eta$ such that the composition $F_n(Z)\to F_n(X)\to F_n(U)$ is zero
for any $n\in\Z$.
\end{itemize}
\end{defin}

We say that a scheme $Y$ is {\it of geometric type over $k$} if
$Y=\cap_{\alpha} U_{\alpha}$, where $\{U_{\alpha}\}$ is a collection
of open subsets in a variety $X$ over $k$. For such $Y$, we put
$F_n(Y)= \lim\limits_{\longrightarrow\atop U}F_n(U)$, where the
limit is taken over all open subsets $U\subset X$ containing $Y$. It
follows easily that $F_n(Y)$ is independent in $X$. If $Y$ is an
equidimensional scheme of geometric type over $k$, then we put
$F_n(Y^{\ge p})=\lim\limits_{\longrightarrow\atop Z}F_n(Z)$, where
the limit is taken over all closed subsets $Z\subset Y$ such that
all irreducible components of $Z$ have codimension at least $p$ in
$Y$.

The proof of the
next statement is the same as the proof of Theorem 5.11 in \cite{Q}
or Theorem 4.2 in \cite{BO}.

\begin{prop}\label{short-exact-sequence}
Let $F_*$ be an l.a.f. homology theory over a field $k$;
then for any regular irreducible local scheme $Y$ of geometric type over $k$ and
any $n\in\Z$, there is a natural short exact sequence
$$
0\to F_n(Y^{\ge p})\to\mbox{$\bigoplus\limits_{\eta\in
Y^{(p)}}F_n(k(\eta))$}\to F_{n-1}(Y^{\ge (p+1)})\to 0,
$$
where the first map takes each element $\alpha\in F_n(Z)$ to the
restrictions of $\alpha$ to the generic points of all components in
$Z$ that have codimension $p$ in $Y$.
\end{prop}

\begin{rmk}
The differential in the Gersten complex for $Y$ equals to the
composition of two corresponding maps in the exact triples from
Proposition~\ref{short-exact-sequence} for $n$ and for $n+1$.
\end{rmk}

\begin{corol}\label{K-integrality}
Under the notations from Proposition \ref{short-exact-sequence},
any cocycle $\{f_{\eta}\}\in \bigoplus\limits_{\eta\in
Y^{(p)}}F_n(k(\eta))$ in the Gersten complex on $Y$
is equal to the restriction of an element $\alpha\in F_n(Z)$ such that the set
of all codimension $p$ components of $Z$ is equal to the support of $\{f_{\eta}\}$.
\end{corol}
\begin{proof}
By Proposition \ref{short-exact-sequence}, any cocycle
$\{f_{\eta}\}\in \bigoplus\limits_{\eta\in
Y^{(p)}}F_n(k(\eta))$ is defined by an element
$\alpha\in F_n(Z)$ for a certain closed subset
$Z\subset Y$ of codimension at least $p$. Let $Z_0\subset Z$ be
a codimension $p$ irreducible component in $Y$ that is not from the support
of $\{f_{\eta}\}$. Note that $F_n(k(Z_0))=
\lim\limits_{\longrightarrow}F_n(U)$, where the limit is taken over all open
subsets $U\subset Z_0$ that are also open in $Z$. From the localization
sequence it follows that, in fact, $\alpha$ belongs to $F_n(Z')$,
where the closed subset $Z'\subset Z$ has the same irreducible components as
$Z$ except for $Z_0$ and
$Z_0$ is replaced by a proper closed subset. This concludes the proof.
\end{proof}

Proposition \ref{short-exact-sequence} also implies the following important statement.

\begin{prop}\label{Gersten-resolution}
Let $F_*$ be an l.a.f. homology theory over a field $k$; then for
any equidimensional regular variety $X$ over $k$ and for any
$n\in\Z$, the morphism of complexes of sheaves
$\F^X_n\to\underline{\G}(X,F_*,n)^{\bullet}$ is a quasiisomorphism.
\end{prop}

\begin{corol}\label{corol-Gerst}
Let $F_*$ be an l.a.f. homology theory over a field $k$, $X$ be an
equidimensional regular variety $X$ over $k$, and $n\in\Z$; then we
have:
\begin{itemize}
\item[(i)]
the sheaf $\F^X_n$
on $X$ is Cohen--Macaulay in the sense of \cite{Har} and the natural morphism
of complexes of sheaves
$\underline{Cous}(X,\F^X_n)^{\bullet}\to\underline{\G}(X,F_*,n)^{\bullet}$
is an isomorphism; in particular, the sheaves $\F^X_n$ are
$1$-pure on $X$.
\item[(ii)]
for any equidimensional closed subset $Z\subset X$ of codimension $p$,
the natural morphism of sheaves $R^p\gamma_Z\F^X_n\to
R^p\gamma_Z\underline{\G}(X,F_*,n)^{\bullet}=\F(n-d)^Z_{d-p}$
is an isomorphism; in particular, there is a natural morphism
$\F^Z_{n-p}\to R^p\gamma_Z\F^X_n$, which is an isomorphism at the generic points
of $Z$;
\item[(iii)]
for any $q\in\Z$, we have $\F(q)^X_d=\F^X_{d+q}$ and $\F(q)^X_m=0$ if $m\ne 0$.
\end{itemize}
\end{corol}

\begin{corol}
Let $F_*$ be an l.a.f. homology theory. Suppose $X$ is an
equidimensional variety that is regular outside of a codimension two
closed subset; then for any point $\eta\in X$ and any $n\in\Z$, we
have
$$
(\F(n)^X_d)_{\eta}=\hbox{$\bigcap\limits_D F_{n+d}(\OO_{X,D})\subset F_{n+d}(k(X))$},
$$
where $d=\dim(X)$ and
$D$ runs over all irreducible divisors in $X$ containing $\eta$.
\end{corol}
\begin{proof}
This follows from the regularity of the discrete valuation ring $\OO_{X,D}$
for any $D$ and the exactness of the Gersten complex for $X_D=\Spec(\OO_{X,D})$.
\end{proof}

\begin{examps}\label{homology_examples}
\hspace{0cm}

The following examples are homology theories
locally acyclic in fibrations:

1) $F_n(X)=K_n'(X)=\pi_{n+1}(BQ\M(X))$ for $n\ge 0$ and $F_n(X)=0$ for $n<0$,
where $\M(X)$ is the exact category of
coherent sheaves on $X$ (see \cite{Q}, proof of Theorem 5.11).

2) $F_n(X)=H_n(X,i)$ for some $i\in \Z$,
where $(H^*,H_*)$ is a Poincar\'e duality theory with
supports in the sense of Bloch--Ogus (see \cite{BO}, Proposition 4.5).

3) $F_n(X)=A_n(X;M)$, where $M$ is a cycle module over $k$ in the sense of Rost
(see \cite{Ros}, proof of Proposition 6.4). In this case we have
$\G(X,F_*,0)_{\bullet}=C_{\bullet}(X;M)$ in notations from \cite{Ros}
and $\G(X,F_*,m)_{\bullet}=0$ if $m\ne 0$.
\end{examps}

\begin{rmk}
The sheaf $\K_n(\OO_X)$ on a smooth variety $X$ over $k$ from both
Examples \ref{homology_examples}, 1) and Examples \ref{homology_examples}, 3)
for $M=\bigoplus\limits_{n\ge 0}K_n$ (compare with Corollary \ref{corol-Gerst},$(iii)$).
\end{rmk}

In the next three sections we develop some technique necessary for
the proof of Lemma~\ref{approximation}.

\subsection{Strongly locally effaceable pairs}\label{K-theory}

Let $F_*$ be a homology theory locally acyclic in fibrations over a
field $k$ and $X$ be an equidimensional variety. Consider
equidimensional subvarieties $Z\subset X$ and $\widetilde{Z}\subset
X$ of codimensions $p$ and $p-1$ in $X$, respectively, such that
$Z\subset \widetilde{Z}$.

\begin{defin}\label{defin-s.l.e.}
Suppose that for each (not necessary closed) point
$x\in Z$ and for any open subset $V\subset X$ containing $x$, there exists a
smaller open subset $W\subset X$ containing $x$ such that the natural map
$$
F_n(V\cap Z)\to F_n(W\cap\widetilde{Z})
$$
is zero for all $n\in\Z$. In addition, suppose that for any $q\ge 0$,
there exists an
assignment $R\mapsto\Lambda(R)$, where $R$ is an equidimensional
subvariety of codimension $q$ in $Z$, $\Lambda(R)$ is an
equidimensional subvariety of codimension $q$ in $\widetilde{Z}$ such that
$R\subset\Lambda(R)$ and for any (not necessary closed) point
$x\in Z$, for any open subset $V\subset X$ containing $x$, there exists a
smaller open subset $W\subset X$ containing $x$ such that the composition
$$
F_n(V\cap (Z\backslash R))\to F_n(V\cap (Z\backslash
\Lambda(R)))\to F_n(W\cap(\widetilde{Z}\backslash\Lambda(R)))
$$
is zero for all $n\in\Z$ (in fact, this condition makes sense
whenever $x\in R$). Then we say that the pair of subvarieties
$(Z,\widetilde{Z})$ is a {\it strongly locally effaceable pair}, or
an {\it s.l.e. pair} (developing the terminology from \cite{BO}).
\end{defin}

The assignment $\Lambda$ in the definition of s.l.e. pairs is needed
to establish a relation with the Gersten complex, as is stated in
the next proposition.

\begin{prop}\label{differential}
Suppose that the field $k$ is infinite and perfect. Let
$(Z,\widetilde{Z})$ be an s.l.e. pair of
subvarieties on a smooth variety $X$ over the field $k$ such
that $Z$ has codimension $p$ in $X$. Choose an arbitrary (not necessary
closed) point $x\in Z$. Suppose that the collection
$\{f_{z}\}\in\bigoplus\limits_{{z}\in Z_x^{(0)}}F_n(k(z))$ is
a cocycle in the Gersten complex on $X_{x}=\Spec(\OO_{X,x})$, i.e.,
suppose that $d_x(\{f_z\})=0$, where $d_x$ is the differential in the complex
$\G(X_x,F_*,n+p)^{\bullet}$. Then there exists a collection
$\{g_{\widetilde{z}}\}\in\bigoplus\limits_{{\widetilde{z}}\in
\widetilde{Z}_x^{(0)}}F_{n+1}(k(\widetilde{z}))$ such that
$d_x(\{g_{\widetilde{z}}\})=\{f_x\}$.
\end{prop}
\begin{proof}
It follows from Corollary \ref{K-integrality} that the collection
$\{f_z\}$ is defined by an element $\alpha\in F_n((Z\cup S)_x)$
for a certain closed subset $S\subset X$
such that each irreducible component of $S$ has
codimension at least $p+1$ in $X$ and is not contained in $Z$.
We may assume that $S$ is equidimensional
of codimension $p+1$ in $X$. Hence the
intersection $Z\cap S$ is contained in some equidimensional
subvariety $R\subset Z$ of codimension $p+2$ in $X$.
Furthermore, $\alpha\in F_n(V\cap (Z\cup S))$
for some open subset $V\subset X$ containing $x$.

By Corollary \ref{global_addition}, there is an equidimensional subvariety
$\widetilde{S}$ of codimension $p$ in $X$ such that the
pair $(\Lambda(R)\cup S,\widetilde{S})$ is strongly locally
effaceable. Consider the following commutative diagram, whose middle column is exact
in the middle term:
$$
\begin{array}{ccccc}
&&F_n(\Lambda(R)\cup S)&\longrightarrow& F_n(\widetilde{S})\\
&&\downarrow&&\downarrow\\ F_n(Z\cup S)&\longrightarrow&
F_n(\widetilde{Z}\cup S)&\longrightarrow& F_n(\widetilde{Z}\cup
\widetilde{S})\\ \downarrow&&\downarrow&&\\ F_n(Z\backslash R)&
\longrightarrow& F_n(\widetilde{Z}\backslash (\Lambda(R)\cup S))&&
\end{array}
$$
The map in the bottom raw is the composition
$$
F_n(Z\backslash R)\to
F_n(\widetilde{Z}\backslash\Lambda(R))\to
F_n(\widetilde{Z}\backslash(\Lambda(R)\cup~S)).
$$

Since the pairs
$(Z,\widetilde{Z})$ and $(\Lambda(R)\cup S,\widetilde{S})$ are
s.l.e., for the point $x\in Z$ and the open
subset $V\subset X$ considered above, there exists a smaller open subset $W\subset X$
containing $x$ such that the map
$$
F_n(V\cap (Z\cup S))\to
F_n(W\cap(\widetilde{Z}\cup\widetilde{S}))
$$
is zero. Therefore
$\alpha$ is a coboundary of an element $\beta\in F_{n+1}(W\cap
((\widetilde{Z}\cup\widetilde{S})\backslash(Z\cup S)))$ in the
localization exact sequence associated to the closed embedding
$W\cap(Z\cup S)\hookrightarrow
W\cap(\widetilde{Z}\cup\widetilde{S})$. In particular, $\beta$
defines a collection
$\{g_{\widetilde{z}}\}\in\bigoplus\limits_{\widetilde{z}\in
\widetilde{Z}_x^{(0)}}F_{n+1}(k(\widetilde{z}))$. Note that all codimension
$p-1$ irreducible components of $\widetilde{Z}\cup\widetilde{S}$ are contained
in $\widetilde{Z}$, while all codimension $p$ irreducible components of $Z\cup S$ are in
$Z$ (as before, codimensions are taken with respect to $X$). Therefore
$d_{x}(\{g_{\widetilde{z}}\})=\{f_z\}$ and the proposition is proved.
\end{proof}

\subsection{Existence and addition of strongly locally
effaceable pairs}\label{section-existence-addition}

Let $F_*$ be an l.a.f. homology theory over a field $k$ and $X$ be a
variety over $k$.

\begin{defin}
Let $f\ge 0$ be a natural number and $(Z,\widetilde{Z})$ be an
s.l.e. pair on $X$. Suppose that for each irreducible subvariety
$C\subset X$ and an equidimensional subvariety $R\subset Z$ of
codimension $q$ in $Z$ with $C \nsubseteq R$, we can choose an
equidimensional subvariety $\Lambda_C(R)\subset\widetilde{Z}$ of
codimension $q$ in $\widetilde{Z}$ such that $C\nsubseteq
\Lambda_C(R)$, $R\subset\Lambda_C(R)$, and the following property
holds true. For any $f$ irreducible subvarieties $C_1,\ldots,
C_f\subset\widetilde{Z}$, for any $f$ equidimensional subvarieties
$R_1,\ldots,R_f\subset Z$ (maybe of different codimensions in $Z$)
with $C_i\nsubseteq R_i$ for all $i,1\le i\le f$, for any schematic
point $x\in Z$, and an open subset $V\subset X$ containing $x$,
there exists a smaller open subset $W\subset X$ containing $x$ such
that the natural map
$$
F_n(V\cap (Z\backslash (R_1\cup\ldots\cup R_f)))\to F_n(W\cap
(\widetilde{Z}\backslash(\Lambda_{C_1}(R_1)\cup\ldots \cup
\Lambda_{C_f}(R_f))))
$$
is zero for all $n\in\Z$. Then we say that the
pair of subvarieties $(Z,\widetilde{Z})$ is {\it strongly locally
effaceable with the freedom degree at least $f$} or is an {\it $f$-s.l.e. pair}.
In particular, a {\it strongly locally effaceable pair
with the freedom degree at least zero}
is the same as a strongly locally effaceable pair.
\end{defin}

\begin{rmk}\label{subpair}
If the pair $(Z,\widetilde{Z})$ is $f$-s.l.e.
and $Z'\subset Z$ is any closed equidimensional subset of
the same dimension as $Z$, then the pair $(Z',\widetilde{Z})$ is also $f$-s.l.e.
\end{rmk}

Here is the existence theorem for strongly locally effaceable pairs
with a given freedom degree.

\begin{theor}\label{freely_effaceable}
Suppose that the field $k$ is infinite and perfect.
Let $X$ be an affine smooth variety over the field $k$.
Consider an equidimensional subvariety $Z$ of codimension $p\ge 2$ in
$X$ and a finite subset of closed points
$T\subset X\backslash Z$. Then for any natural
number $f\ge 0$, there exists a subvariety $\widetilde{Z}\supset Z$ that
does not contain any point from $T$ and such that
the pair $(Z,\widetilde{Z})$ is strongly locally effaceable with the
freedom degree at least $f$.
\end{theor}

\begin{corol}\label{global_addition}
Suppose that the field $k$ is infinite and perfect.
Let $X$ be a smooth variety over the field $k$.
Consider an equidimensional subvariety $Z$ of codimension $p\ge 2$ in
$X$ and a closed subset $T\subset X$ such that
no irreducible component of $T$
is contained in $Z$ and $T$ has codimension at most $p-1$ in $X$.
Then there exists a subvariety $\widetilde{Z}\supset Z$ that
does not contain any irreducible component of $T$ and such that
the pair $(Z,\widetilde{Z})$ is s.l.e.
\end{corol}

\begin{proof}
Consider a finite open affine covering $X=\cup_{\alpha}
U_{\alpha}$. For each $\alpha$ and for each irreducible component
of $T\cap U_{\alpha}$, choose a closed point on it
outside of $Z$ and thus get a finite subset $T'_{\alpha}\subset U_{\alpha}\backslash Z$.
The application of Theorem \ref{freely_effaceable}
for the intersection of all data with $U_{\alpha}$ yields the existence
of a closed subset $\widetilde{Z}_{\alpha}\subset U_{\alpha}$ such that
$\widetilde{Z}_{\alpha}$ does not contain any point from $T'_{\alpha}$
and the pair $(Z\cap U_{\alpha},\widetilde{Z}_{\alpha})$ is s.l.e.
We put $\widetilde{Z}=\cup_{\alpha}
\overline{\widetilde{Z}_{\alpha}}$, where the bar denotes the closure in $X$.
By the codimension assumption, $\widetilde{Z}$
does not contain any irreducible component of $T$. Also, the pair
$(Z,\widetilde{Z})$ is
s.l.e., where $\Lambda(R)$ can be taken as the union
over $\alpha$ of the closures of $\Lambda_{\alpha}(R\cap U_{\alpha})$ for
an equidimensional subvariety $R\subset Z$.
\end{proof}

\begin{corol}\label{global_addition-2}
Suppose that the field $k$ is infinite and perfect.
Let $X$ be a smooth quasiprojective variety over the field $k$.
Consider an equidimensional subvariety $Z$ of codimension $p\ge 2$ in
$X$ and a closed subset $T\subset X$ such that
no irreducible component of $T$
is contained in $Z$. Then there exists a subvariety $\widetilde{Z}\supset Z$ that
does not contain any irreducible component of $T$ and such that
the pair $(Z,\widetilde{Z})$ is s.l.e.
\end{corol}
\begin{proof}
For each irreducible component of $T$ we choose a closed point on it outside of $Z$.
Thus we
get a finite set of closed points $T'\subset X\backslash Z$. Since $X$ is quasiprojective,
there exists a finite open affine covering $X=\cup_{\alpha}
U_{\alpha}$ such that for each $\alpha$ we have $T'\subset U_{\alpha}$.
To conclude the proof, we repeat the same argument as in the proof of
Corollary \ref{global_addition}.
\end{proof}

Combining Proposition \ref{differential} and Corollary \ref{global_addition},
we get the following statement,
which could be considered as a uniform version of Gersten
conjecture for smooth varieties and has interest in its own
right.

\begin{corol}\label{uniform_Gersten}
Suppose that the field $k$ is infinite and perfect.
Let $X$ be a smooth variety over the field $k$. Then for
any equidimensional subvariety $Z\subset X$ of codimension $p$ in $X$
there exists an equidimensional subvariety $\widetilde{Z}\supset Z$ of
codimension $p-1$ in $X$ with the following property. Suppose we
are given an arbitrary (not necessary closed) point $x\in Z$ and a
collection $\{f_{z}\}\in\bigoplus\limits_{{z}\in Z_x^{(0)}}F_n(k(z))$
such that $\{f_{z}\}$ is a cocycle in the local Gersten
resolution at $x$, i.e., that $d_x(\{f_z\})=0$. Then there exists a
collection $\{g_{\widetilde{z}}\}\in\bigoplus\limits_{{\widetilde{z}}\in
\widetilde{Z}_x^{(0)}}F_{n+1}(k(\widetilde{z}))$ such that
$d_x(\{g_{\widetilde{z}}\})=\{f_z\}$.
\end{corol}

\begin{rmk}
Corollary \ref{uniform_Gersten} is stronger than Theorem 4.2 in \cite{BO}
or Theorem 5.11 in \cite{Q}. Namely in \cite{BO} the analogous result was shown
for a fixed subvariety $Z$, a fixed point $x\in Z$, and a fixed
collection $\{f_z\}$ on $Z_x$. The proof in \cite{BO} does not
seem to imply directly Corollary \ref{uniform_Gersten} and that is
why we use some different geometrical method during the proof of
Theorem \ref{freely_effaceable}.
\end{rmk}

\begin{proof}[Proof of Theorem \ref{freely_effaceable}]
The proof is in two steps.

{\it Step 1.} During this step ``a point'' always means ``a closed point''.
Recall that $d=\dim X$. We say that a morphism $\pi:X\to\Ab^{d-1}$
{\it resolves} a point $x\in Z$ if $\pi$ is smooth of relative
dimension one at $x$, the restriction $\varphi=\pi|_{Z}$ is
finite, $\varphi^{-1}(\varphi(x))=\{x\}$, and
$\pi(T)\cap\pi(Z)=\emptyset$.
The following geometric result is a globalization of Quillen's
construction used in his proof of Gersten conjecture, see
\cite{Q}, Lemma 5.12 and \cite{BO}, Claim on p.191.

\begin{prop}\label{projections}
Under the above assumptions, there exists a finite set $\Sigma$ of
morphisms $\pi:X\to\Ab^{d-1}$ such that for any $f$ points
$y_1,\ldots,y_f\in X$ and any point $x\in Z$, there exists $\pi\in
\Sigma$ such that $\pi$ resolves $x$ and
$\pi(y_i)\notin\pi(Z\backslash\{y_i\})$ for all $i,1\le i\le f$.
\end{prop}

\begin{proof}
Using Claim \ref{noetherian_induction}, we prove by decreasing
induction on $e$, $-1\le e\le f$ that for any $e+1$ irreducible
subsets $Z_0\ldots,Z_{e}$ in $Z$ there exist non-empty open subsets
$U_0\subset Z_0,\ldots,U_e\subset Z_{e}$ and a finite set of
morphisms $\Sigma$ such that the statement of Proposition
\ref{projections} is true for all collections of points
$(x,y_1,\ldots,y_f,y_1',\ldots,y_f')$ satisfying
$x\in U_0$, $y_i\in U_i$ for all $i,1\le i\le e$, $y_i\in Z$ for
all $i,e+1\le i\le f$, and $y'_i\in X\backslash Z$ for all $i,0\le i\le f$.
Note that for $e=-1$ this immediately implies the needed statement of Proposition
\ref{projections}.

Suppose that $e=f$ and $y'_1,\ldots,y'_f$ are $f$ arbitrary points in
$X\backslash Z$. The application of Claim \ref{noetherian_induction} for
the points $y'_1,\ldots,y'_f$ and the closed subsets $Z_0,\ldots, Z_f$ in $Z$
yields the existence of non-empty open subsets $U_0\subset Z_0,\ldots,
U_f\subset Z_f$ and a morphism $\pi$ satisfying the conditions from
Claim \ref{noetherian_induction}. Since the conditions $y'_i\notin
\pi^{-1}(\pi(Z))$, $1\le i\le f$ are open,
there is a finite open covering $\cup_{\alpha}V_{\alpha}$ of the direct product
$(X\backslash Z)^{\times f}$ such that for all $\alpha$, any collection
$(y'_1,\ldots,y'_f)$ from the open subset $V_{\alpha}$ satisfies the
conditions from Claim \ref{noetherian_induction} with respect to some non-empty
open subsets $U^{\alpha}_0\subset Z_0,\ldots,U^{\alpha}_f\subset Z_f$ and
a morphism $\pi^{\alpha}$. Taking the finite intersections
$U_i=\cap_{\alpha}U^{\alpha}_i$ for each $i,0\le i\le f$, we get the needed
open subsets in $Z_i$, $0\le i\le f$, while the needed finite set
of morphisms is $\Sigma=\{\pi_{\alpha}\}$.

Now let us do the induction step from $e$ to $e-1$. Choose any
irreducible component $C_1$ of $Z$. By the inductive hypothesis,
there exist non-empty open subsets
$U^1_0\subset Z_0,\ldots,U^1_{e-1}\subset Z_{e-1},U^1_{e}\subset C_1$
and a finite set of morphisms $\Sigma_1$ such that they
satisfy the conditions stated above. We may
assume that the subset $U^1_e\subset C_1$ is also open in $Z$. Let $C_2$
be one of the irreducible components of $Z\backslash U^1_e$. Again, by
the inductive hypothesis, there exist other open subsets
$U^2_0\subset Z_0,\ldots,U^2_{e-1}\subset Z_{e-1},U^2_e\subset C_2$
and a finite set of morphisms $\Sigma_2$
such that they satisfy the conditions stated above.
We repeat the same step until we come
to the end of the obtained finite stratification of $Z$ by open
subsets $U_e^j$ in $C_j$. Taking the finite intersections
$U_i=\cap_jU^j_i$ for each $i,0\le i\le e-1$, we get the needed
open subsets in $Z_i$, $0\le i\le e$, while the needed finite set of morphisms
is equal to the finite union $\Sigma=\cup_j\Sigma_j$.
\end{proof}

\begin{claim}\label{noetherian_induction}
For any $f$ points $y'_1\ldots,y'_f\in X\backslash Z$ and $f+1$
closed subsets $Z_0,\ldots,Z_f\subset Z$ there exist non-empty open
subsets $U_i\subset Z_i$, $0\le i\le f$ and a morphism $\pi:X\to
\Ab^{d-1}$ such that $\pi$ resolves all point $x$ from $U_0$ (with
respect to $Z$), $\pi(y_i)\notin\pi(Z\backslash\{y_i\})$ for
any point $y_i\in U_i$, $1\le i\le f$, and $\pi(y'_i)\notin\pi(Z)$ for all
$i,1\le i\le f$.
\end{claim}

\begin{proof}
Let $\overline{X}\subset\P^N$ be a projective variety such that
$X=\overline{X}\backslash H$, where $H\subset\P^N$ is a hyperplane.
In what follows the bar denotes the projective closure in $\P^N$ and the star
denotes a join of two projective subvarieties in $\P^N$.

Without loss of generality we may assume that $T$ contains the points
$y'_1\ldots,y'_f$ and that $Z_i$ are irreducible for all $i,0\le i\le f$.
Let $x'_i$ be an arbitrary smooth point on $Z_i$ for each $i,0\le i\le f$
(such $x'_i$ exist because the field is perfect). We have the
following dimension conditions: $\dim(H\cap\overline{Z})\le d-3$,
$\dim(H\cap (T*\overline{Z}))\le d-2$,
$\dim(H\cap \overline{T_{x'_0} X})=d-1$, and $\dim(H\cap
\overline{T_{x'_i} Z_i})\le d-3$ for all $i,0\le i\le f$. Since the ground
field is infinite, there exists a projective subspace
$L'\subset H$ of codimension $d-2$ in $H$ such that $L'$ does not
intersect with $\overline{Z}$, intersects $T*\overline{Z}$ in a finite set of
points, intersects $\overline{T_{x'_0} X}$ in a line, and does not intersect
with any $\overline{T_{x'_i}Z_i}$ for $0\le i\le f$. Note that the
projection $\pi_{L'}$ with the center at $L'$ defines on
$\overline{Z}$ a finite morphism $\varphi_{L'}$. Put
$Z_i'=\varphi_{L'}^{-1}(\varphi_{L'}(\overline{Z_i}))\subset
\overline{Z}$ for $0\le i\le f$. For each $i,0\le i\le f$, let $x_i$ be an arbitrary
point on $Z_i\subset Z_i'$ such that $x_i$ is smooth on $Z'_i$, $\overline{T_{x_i}Z_i}$
does not intersect with $L'$, and $\overline{T_{x_0} X}$ intersects $L'$ in a line.

We claim that the intersection $L'\cap (x_i*Z_i')$ is a finite
set of points for all $i,0\le i\le f$.
Indeed, each join $x_i*Z'_i$ is the union two subsets. The first
one is the tangent space to $Z'_i$ at $x_i$ and does not
intersect with $L'$. The second one is the union of lines passing
through $x_i$ and other points from $Z_i'$. The intersection of
this union of lines with $L'$ corresponds to the fiber of $x_i$
under the finite morphism $\varphi_{L'}$ and therefore is finite.
Hence there exists a hyperplane $L\subset L'$ that does not
intersect with the joins $T*\overline{Z}$ and $x_i* Z_i'$ for any $i,0\le
i \le f$ and that intersects the tangent spaces
$\overline{T_{x_0}X}$ in one point. Since the variety $X$ is smooth, the projection
$\pi_L$ with the center at $L$ is smooth at $x_0$. Besides, the map $\pi_L$
can not glue points from $Z_i'$
with points from $\overline{Z}\backslash Z_i'$ for any $i,0\le i\le
f$. Therefore the application of Lemma \ref{entry_locus} with $Y=Z'_i$ yields that
there exist non-empty open subsets $U_i\subset Z_i$ containing $x_i$ such that
$\pi_L$ resolves all points $x$ from $U_0$ and
$\varphi_L^{-1}(\varphi_L(y_i))=\{y_i\}$ for all points $y_i$ from
$U_i$, $1\le i\le f$, where $\varphi_L=\pi_L|_{Z}$. In addition,
$\pi(T)$ does not intersect with $\pi(Z)$, and, in particular, $\pi(y'_i)\notin
\pi(Z)$ for all $i,1\le i\le f$.
\end{proof}

We have used the following fact from projective geometry.

\begin{lemma}\label{entry_locus}
Let $Y\subset\P^N$ be a projective variety, $x\in Y$ be a smooth
point on $Y$. Suppose that a projective subspace $M\subset \P^N$
does not intersect with the join $x*Y$. Then there exists an open
subset $W\subset Y$ containing $x$ such that
$\varphi^{-1}(\varphi(y))=\{y\}$ for all $y\in W$, where $\varphi$
is the restriction to $Y$ of the projection $\pi_{M}$ with the
center at $M$.
\end{lemma}

{\it Step 2.} In notations from Proposition \ref{projections}, consider the finite set
$\Sigma$ of morphisms $\pi:X\to\Ab^{d-1}$. Put
$\widetilde{Z}=\cup\pi^{-1}(\pi(Z))$,
where the union is taken over all $\pi\in \Sigma$. By construction,
$\widetilde{Z}$ does not contain any irreducible component of $T$.

\begin{prop}\label{BlOg}
The pair $(Z,\widetilde{Z})$ is $f$-s.l.e.
\end{prop}
\begin{proof}
Essentially, we repeat the proof of Theorem 5.11 in \cite{Q} with
some modifications.

First we note that after we choose a suitable
closed point $x'$ on $\overline{x}\subset Z$ we may suppose that the
given open subset $V\ni x$ actually contains $x'$.
Thus we may suppose $x$ to be closed.

Choose $\pi\in \Sigma$ that resolves $x$ and, as before, put
$\varphi=\pi|_{Z}$. Following the construction of Quillen,
consider the Cartesian square:
$$
\begin{array}{ccc}
Y&\stackrel{\varphi'}\longrightarrow&X \\
\downarrow\lefteqn{\pi'}&&\downarrow\lefteqn{\pi}\\
Z&\stackrel{\varphi}\longrightarrow&\Ab^{d-1}.
\end{array}
$$
Note that $\varphi'$ is finite onto its image,
$\varphi'(Y)=\pi^{-1}(\pi(Z))$ is a closed subset in
$\widetilde{Z}$, and $(\varphi')^{-1}(x)$ consists of one point, which we
denote by $z$. Besides, the morphism $\pi'$ is smooth at $z$ and admits a
canonical section $\sigma:Z\to Y$ (with $\sigma(x)=z$). Since the homology theory
$F_*$ is locally acyclic in fibrations, the composition
$F_n(Z)\stackrel{\sigma_*}\longrightarrow
F_n(Y)\to F_n(Y')$ is zero for all $n\in\Z$ and for some
suitable open subset $Y'\subset Y$ containing $z$. Hence the map
$F_n(Z)\to F_n(\widetilde{Z}\cap U)$ is also zero for all
$n\in\Z$ and for some suitable open subset $U\subset X$ containing $x$
such that $(\varphi')^{-1}(U)\subset Y'$ (such $U$ exists, since
$\varphi'$ is finite and $(\varphi')^{-1}(x)=\{z\}$).

Take an arbitrary open subset $V\subset X$ containing $x$. Since
$\varphi^{-1}(\varphi(x))=\{x\}$ and $\varphi$ is finite, there exists
an open subset $D\subset\Ab^{d-1}$ such that $x\in
\varphi^{-1}(D)\subset V$. Restricting the Cartesian diagram
from $\Ab^{d-1}$ to $D$, we get that the natural map $F_n(V\cap
Z)\to F_n(W\cap\widetilde{Z})$ is zero for all $n\in\Z$ and for some
suitable open subset $W\subset V$ containing $x$.

Further, consider an irreducible subvariety $C\subset X$ and an
equidimensional subvariety $R\subset Z$ of codimension $q$ in $Z$ such
that $C\nsubseteq R$. We put
$$
\Lambda_C(R)=\bigcup_{y\in C\backslash R}
\left(\bigcup_{\pi\in\Sigma_y}\pi^{-1}(\pi(R))\right),
$$
where $\Sigma_y$
is the set of all $\pi\in\Sigma$ such that
$\pi(y)\notin\pi(Z\backslash\{y\})$. For instance, if $C$ is not
contained in $\widetilde{Z}$, then $\Lambda_C(R)=\cup
\pi^{-1}(\pi(R))$ where the union is taken over all
$\pi\in\Sigma$. For irreducible subvarieties $C_1,\ldots, C_f$ in
$X$ and subvarieties $R_1,\ldots,R_f$ in $Z$ satisfying the needed
conditions, we choose closed points $y_i\in C_i\backslash R_i$. By
construction, for any closed point $x\in Z$, there is a morphism
$\pi\in\Sigma$ such that it resolves $x$ and belongs to
$\Sigma_{y_1}\cap\ldots\cap\Sigma_{y_f}$. The same argument
with the analogous Cartesian diagram as before leads to the
needed result.
\end{proof}

This completes the proof of Theorem \ref{freely_effaceable}.
\end{proof}

\begin{rmk}\label{rmk-irreducible-sle}
In Theorem~\ref{freely_effaceable} one may also require that each
irreducible component in $\widetilde{Z}$ contains some irreducible
component in $Z$. This follows from the fact that for each
irreducible component $Z_0$ in $Z$, the variety $\pi^{-1}(\pi(Z_0))$
is irreducible in an open neighborhood of a given point $x$, where
$\pi:X\to \Ab^{d-1}$ is a morphism that resolves the point $x$ and,
in particular, is smooth at $x$.
\end{rmk}

The following proposition allows to add strongly locally effaceable
pairs.

\begin{prop}\label{addition}
Suppose that the field $k$ is infinite and perfect.
Let $X$ be an affine smooth variety over the field $k$.
Consider two equidimensional subvarieties $Z_1$ and $Z_2$ of the same codimension
$p\ge 2$ in $X$. Suppose that we are given a subvariety
$\widetilde{Z}_1\supset Z_1$ such that the pair
$(Z_1,\widetilde{Z}_1)$ is strongly locally effaceable with the
freedom degree at least $f\ge 2$. Consider a closed subset $T\subset X$
such that all irreducible components of $T$ have codimension at
most $p-1$ in $X$, and an
irreducible subvariety $K\subset X$ such that $K$ is not contained
in $Z_2$. Then there exists a subvariety $\widetilde{Z}_2$ such
that no irreducible component of $T$ and $K$ is contained in
$\widetilde{Z}_2$ and the pair $(Z_1\cup
Z_2,\widetilde{Z}_1\cup\widetilde{Z}_2)$ is strongly locally
effaceable with the freedom degree at least $f-1$.
\end{prop}

\begin{proof}
If $K$ is not contained in $Z_1$, then the statement of the
proposition follows directly from Theorem \ref{freely_effaceable}
after we choose a closed point on each irreducible component of $T$ and $K$
outside of $Z_1\cup Z_2$ (in this case the freedom degree does
not decrease). Otherwise we use the same construction as in
the proof of Proposition \ref{differential}.

Suppose that $K\subset Z_1$. Then there is a
codimension one subvariety $Z'_2$ in $Z_1$ such that $Z'_2$ does not
contain $K$ and contains the intersection of $Z_1$ with each
irreducible component of $Z_2$ that is not contained in $Z_1$.
Put $Z_3=\Lambda_K(Z_2')\subset\widetilde{Z}_1$. By the codimension
assumption, $Z_2\cup Z_3$ does not
contain any irreducible component of $T$.
Choosing closed points on each irreducible component of $T$ and on $K$
outside of $Z_2\cup Z_3$, we see that, by Theorem \ref{freely_effaceable},
there exists
a subvariety $\widetilde{Z}_2\subset X$ such that the pair $(Z_2\cup
Z_3,\widetilde{Z}_2)$ is $(f-1)$-s.l.e. and $\widetilde{Z}_2$ does not contain any
irreducible component of $T$ and $K$.

We claim that the pair
$(Z_1\cup Z_2,\widetilde{Z}_1\cup\widetilde{Z}_2)$ is s.l.e.
This is implied by the following commutative diagram,
whose middle column is exact in the middle term:
$$
\begin{array}{ccccc}
&&F_n(Z_2\cup Z_3)&\longrightarrow& F_n(\widetilde{Z}_2)\\
&&\downarrow&&\downarrow\\ F_n(Z_1\cup Z_2)&\longrightarrow&
F_n(\widetilde{Z}_1\cup Z_2)&\longrightarrow&
F_n(\widetilde{Z}_1\cup\widetilde{Z}_2)\\
\downarrow&&\downarrow&&\\ F_n(Z_1\backslash
Z'_2)&\longrightarrow& F_n(\widetilde{Z}_1\backslash (Z_2\cup
Z_3))&&
\end{array}
$$
The map in the bottom raw is the composition
$$
F_n(Z_1\backslash
Z'_2)\to F_n(\widetilde{Z}_1\backslash Z_3)\to
F_n(\widetilde{Z}_1\backslash (Z_2\cup Z_3)).
$$
Since the pairs
$(Z_1,\widetilde{Z}_1)$ and $(Z_2\cup Z_3,\widetilde{Z}_2)$ are
s.l.e., for any point $x\in Z_1\cup Z_2$ and any
open subset $V\subset X$ containing $x$, there exists a smaller open subset
$x\in W\subset X$ such that the map
$$
F_n(V\cap (Z_1\cup
Z_2))\to F_n(W\cap(\widetilde{Z}_1\cup\widetilde{Z}_2))
$$
is zero for all $n\in\Z$.

Now consider an irreducible subvariety $C\subset X$ and an
equidimensional subvariety $R\subset Z_1\cup Z_2$ of codimension $q$ in $Z_1\cup Z_2$
such that $C\nsubseteq R$. Let $R'$ be the union of all
irreducible components in $R$ that are not contained in $Z_2'\cup Z_2$. Let
$\Lambda'_{C}(R')$ be the union of all irreducible components in
$\Lambda_{C}(R')\subset\widetilde{Z}_1$ that are not contained in $Z_2\cup Z_3$.
Then there exists an equidimensional subvariety $R''\subset Z_2\cup Z_3$ of codimension $q$
in $Z_2\cup Z_3$ such that $R''$ contains the intersection
$(\Lambda'_{C}(R')\cup R)\cap (Z_2\cup Z_3)$ and does not contain
$C$. Consider $\Lambda_C(R'')\subset\widetilde{Z}_2$, where now $\Lambda_C$ is taken with
respect to the $(f-1)$-s.l.e. pair
$(Z_2\cup Z_3,\widetilde{Z}_2)$.

We claim that $(Z_1\cup Z_2,\widetilde{Z}_1\cup\widetilde{Z}_2)$ is
an $(f-1)$-s.l.e. pair with
respect to the assignment
$(R,C)\mapsto\Lambda_C(R)=\Lambda'_C(R')\cup\Lambda_C(R'')$.
This is implied by the commutative diagram, which analogous to the previous one.
This new diagram is the combination of two following diagrams:
$$
\begin{array}{ccc}
F_n((Z_1\cup Z_2)\backslash \{R_i\})& \longrightarrow&
F_n((\widetilde{Z}_1\cup Z_2)\backslash\{\Lambda'_{C_i}(R'_i)\cup
R_i\})\\ \downarrow&&\downarrow\\ F_n(Z_1\backslash (Z'_2\cup
\{R_i\}))&\longrightarrow& F_n(\widetilde{Z}_1\backslash (Z_2\cup
Z_3\cup\{\Lambda'_{C_i}(R'_i)\})),
\end{array}
$$
$$
\begin{array}{ccc}
F_n((Z_2\cup Z_3)\backslash \{\Lambda'_{C_i}(R'_i)\cup R_i\})&
\longrightarrow&
F_n(\widetilde{Z}_2\backslash\{\Lambda_{C_i}(R_i)\})\\
\downarrow&&\downarrow\\ F_n((\widetilde{Z}_1\cup
Z_2)\backslash\{\Lambda'_{C_i}(R'_i)\cup R_i\})& \longrightarrow&
F_n((\widetilde{Z}_1\cup\widetilde{Z}_2)\backslash
\{\Lambda_{C_i}(R_i)\}).
\end{array}
$$
We glue these diagrams together using the following sequence, which is exact
in the middle term:
$$
F_n((Z_2\cup Z_3)\backslash
\{\Lambda'_{C_i}(R'_i)\cup R_i\})\to F_n((\widetilde{Z}_1\cup
Z_2)\backslash\{\Lambda'_{C_i}(R'_i)\cup R_i\})\to
$$
$$
\to F_n(\widetilde{Z}_1\backslash (Z_2\cup
Z_3\cup\{\Lambda'_{C_i}(R'_i)\})).
$$
Here $\{O_i\}$ means the union of objects $O_i$ over all
$i,1\le i\le f-1$ and the
horizontal maps in the diagrams are the compositions of direct images
under closed embedding and restrictions to open subsets. The
new diagram is analyzed in the same way as the previous one.
\end{proof}

\begin{rmk}\label{codimension-one}
For $p=1$ by Quillen's result the only possible pair $(Z,X)$ is
s.l.e. However there is no analogue of
Theorem \ref{freely_effaceable} and Proposition \ref{addition} in this case
with non-empty $T$ and $K$.
\end{rmk}


\subsection{Patching systems}\label{patching-systems-section}

Let $k$ be a field and $X$ be an equidimensional variety over $k$.
Consider an equidimensional
subvariety $Z\subset X$ of codimension $p$ in $X$.

\begin{defin}
Suppose that the
system of equidimensional subvarieties $\{Z^{1,2}_r\}$, $1\le r\le p-1$
of codimension $r$ in $X$, respectively, satisfies the following conditions:
\begin{itemize}
\item[(i)]
the variety $Z$ is contained in both varieties $Z^1_{p-1}$ and $Z_{p-1}^2$,
and the variety $Z_{r}^1\cup Z_r^2$ is contained in both varieties $Z^1_{r-1}$ and
$Z^2_{r-1}$ for all $r,2\le r\le p-1$;
\item[(ii)]
the pairs $(Z,Z^1_{p-1})$, $(Z,Z^2_{p-1})$, $(Z^1_{r}\cup
Z^2_r,Z^1_{r-1})$, and $(Z^1_{r}\cup Z^2_r,Z^2_{r-1})$ are strongly
locally effaceable with the freedom degree at least $f$ for all $r,2\le r\le p-1$;
\item[(iii)]
the varieties $Z_{r}^1$ and $Z_r^2$ have no common irreducible
components for all $r,1\le r\le p-1$.
\end{itemize}
Then we say that the system of subvarieties $\{Z^{1,2}_r\}$, $1\le r\le p-1$ is
a {\it patching system for the subvariety $Z$ with the freedom degree at least $f$}.
\end{defin}

\begin{prop}\label{patching systems}
Suppose that the field $k$ is infinite and perfect; then
for any integer $f\ge 0$ and any equidimensional subvariety $Z$ of codimension $p$
in the affine smooth variety $X$ over the field $k$, there exists a
patching system $\{Z^{1,2}_{r}\}$, $1\le r\le p-1$ on $X$ for the subvariety $Z$
with the freedom degree at least $f$.
\end{prop}
\begin{proof}
We construct the needed system of subvarieties $\{Z^{1,2}_r\}$
by decreasing induction on $r$, $1\le r\le p-1$.

Suppose that $r=p-1$. The application of Theorem \ref{freely_effaceable} with
empty $T$ yields the existence of an equidimensional subvariety
$Z_{p-1}^1=\widetilde{Z}\subset X$ of codimension $p-1$ in $X$ such that
$(Z,Z_{p-1}^1)$ is an $f$-s.l.e pair.
Choosing a closed point on each irreducible component of
$Z^1_{p-1}$, we see that the application of Theorem \ref{freely_effaceable}
yields the existence of an equidimensional subvariety $Z^2_{p-1}=\widetilde{Z}$ of
codimension $p-1$ in $X$ such that $Z^2_{p-1}$ has no common irreducible
components with $Z^1_{p-1}$ and $(Z,Z^2_{p-1})$ is an $f$-s.l.e. pair.

The induction step from $r+1$ to $r$, $1\le r< p-1$ is analogous to the case $r=p-1$
with $Z$ replaced by $Z^1_{r+1}\cup Z^2_{r+1}$.
\end{proof}

\begin{rmk}\label{global_patching_systems}
Using Corollary \ref{global_addition}
instead of Theorem \ref{freely_effaceable} in the proof of
Proposition \ref{patching systems}, it is possible to show that for
any equidimensional subvariety $Z$ on a (not necessary affine) smooth
variety $X$ over an infinite perfect field there exists a patching system
with the freedom degree at least zero.
\end{rmk}

\begin{rmk}\label{rmk-irreducible-patchingsyst}
\hspace{0cm}
\begin{itemize}
\item[(i)]
By Remark~\ref{rmk-irreducible-sle}, in Proposition~\ref{patching
systems} one may also require that each irreducible component in
$Z^1_{p-1}$ and $Z^2_{p-1}$ contains some irreducible component in
$Z$ and that for all $r$, $1\le r\le p-2$, each irreducible
component in $Z^1_r$ and $Z^2_r$ contains some irreducible component
in $Z^1_{r+1}\cup Z^2_{r+1}$.
\item[(ii)]
Let $W\subset X$ be an equidimensional subvariety such that $W$
meets $Z$ properly; if the patching system $\{Z^{1,2}_r\}$ satisfies
the condition from (i), then $W$ meets $Z^i_r$ properly for all $r$,
$1\le r\le p-1$ and $i=1,2$. Combining this fact with
Corollary~\ref{global_addition-2}, we get that under conditions of
Proposition~\ref{patching systems} one may also require that no
irreducible component in $W\cap Z^1_r$ is contained in $Z^2_r$ for
all $r$, $1\le r\le p-1$.
\end{itemize}
\end{rmk}

Let us introduce the following notation.
Suppose that $Z_i\subset X$, $p\le
i\le q$ are equidimensional subvarieties of codimension $i$ in an
equidimensional variety $X$ over $k$ such that
$Z_q\subset\ldots\subset Z_p$. Consider a collection
$f=\{f_{\eta}\}\in\bigoplus\limits_{\eta\in X^{(p)}}F_n(k(\eta))$.
We put $\nu_{Z_p\ldots Z_q}(f)=
\displaystyle\sum_{z_p\ldots z_q}(\nu_{k(z_{q-1})k(z_{q})}\circ\ldots\circ
\nu_{k(z_p)k(z_{p-1})})(f_{z_p})$, where $z_p,\ldots, z_q$ range over all
collections of generic points in $Z_p,\ldots,Z_q$ such that for any $i,p<i\le
q$, we have $z_i\in\overline{z}_{i-1}$.
For $f\in F_n(k(X))$, let $\sing(f)$ be the set of irreducible divisors
$D$ on $X$ such that $\nu_{XD}(f)\ne 0$.

Here is the main property of patching systems.

\begin{prop}\label{differential-patching-systems}
Let $\{Z^{1,2}_r\}$, $1\le r\le p-1$ be a patching system on $X$ for the
equidimensional subvariety $Z\subset X$ of codimension $p$ in $X$ with the freedom
degree at least zero. Given a (not necessary closed) point $x\in Z$, suppose
that a collection $g\in \bigoplus\limits_{\eta\in X^{(p)}}F_{n-p}(k(\eta))$
is a cocycle in the local Gersten
resolution $\G(Y,F_*,n)^{\bullet}$ at $x$, where $Y=X_{x}$, and that the
support of the collection $g$ is contained in $Z$. Then there exists a
collection $f\in\bigoplus\limits_{\eta\in X^{(0)}}F_n(k(\eta))$ such that
the subvariety $\sing(\nu_{XZ^1_1\ldots Z^1_{r-1}}(f))\subset Z^1_{r-1}$
is contained in $Z^1_r\cup Z^2_r$ for all $r,1\le r\le p-1$, and
$d_{x}\nu_{XZ^1_1\ldots Z^1_{p-1}}(f)=g$, where $d_{x}$ is the differential
in the local Gersten complex $\G(Y,F_*,n)^{\bullet}$.
\end{prop}
\begin{proof}
The proof is by induction on $p\ge 1$. For $p=1$, by Remark \ref{codimension-one},
there is nothing to prove.

Suppose that $p>1$. Then, by Proposition \ref{differential}, there exist two
collections $g^1,g^2\in\bigoplus\limits_{\eta\in X^{(p-1)}}F_{n-p+1}(k(\eta))$
with the support on $Z^1_{p-1}$ and $Z^2_{p-1}$, respectively, such that
$d_x(g^i)=g$ for $i=1,2$.
Therefore, $d_{x}(g^1-g^2)=0$ and, by the inductive assumption, there exists
a collection $f\in \G(X,F_*,n)^0$ such that
$\sing(\nu_{XZ^1_1\ldots Z^1_{r-1}}(f))\subset Z^1_r\cup Z^2_r$ for
all $r,1\le r\le p-2$ and $d_{x}\nu_{XZ^1_1\ldots
Z^1_{p-2}}(f)=g^1-g^2$, where $Z_{p-1}=Z^1_{p-1}\cup Z^2_{p-1}$. Such $f$
satisfies the needed conditions with respect to the initial collection $g$.
\end{proof}

\subsection{Main theorem}\label{section-main-theorem}

Let $F_*$ be a homology theory locally acyclic in fibrations over a field $k$.

\begin{theor}\label{quasiis}
Suppose that $k$ is an infinite perfect field and
that $X$ is an irreducible smooth variety over $k$; then for any $n\in\Z$,
the morphism $\underline{\nu}_X:\underline{\A}(X,\F^X_n)^{\bullet}\to
\underline{Cous}(X,\F^X_n)^{\bullet}=\underline{\G}(X,F_*,n)^{\bullet}$
is a quasiisomorphism.
\end{theor}

\begin{corol}
Under the assumptions from Theorem \ref{quasiis}, the natural morphism
$\F_n^X\to\underline{\A}(X,\F^X_n)^{\bullet}$ is a quasiisomorphism; in particular,
the cohomology groups $H^i(\A(X,\F^X_n)^{\bullet})$ are canonically isomorphic
to the cohomology groups $H^i(X,\F^X_n)$.
\end{corol}

\begin{rmk}
In Theorem \ref{quasiis} we make a strong restriction on the ground
field to be infinite and perfect. In fact, the only one place where
we use this is the geometric proof of
Claim~\ref{noetherian_induction}. It seems possible to prove the
same result for smooth varieties over a finite field, and, then, to
reduce the case of regular varieties over an arbitrary field to that
case by a standard argument of choosing a model. On the other hand,
the author can prove Theorem~\ref{quasiis} for $\dim X\le 3$ over an
arbitrary field, avoiding Claim~\ref{noetherian_induction}.
\end{rmk}

\begin{proof}[Proof of Theorem \ref{quasiis}]
In is enough to prove that the morphism $\nu_U:\A(U,\F^U_n)\to Cous(U,\F^X_n)$
is a quasiisomorphism for any affine open subset $U\subset X$. We may put $X=U$.
Since $\F^X_n$ is a subsheaf in a constant sheaf, by Proposition \ref{theta-emb},
the complex $\A(X,\F^X_n)^{\bullet}$ is a subcomplex in the complex
$\A'(X,\F^X_n)^{\bullet}$. By Remark \ref{surj-cohom} and Theorem \ref{prime-quasiis},
it is enough to show that for any $p$, $0\le p\le d$, the natural homomorphism
$H^p(\A(X,\F^X_n)^{\bullet})\to H^p(\A'(X,\F^X_n)^{\bullet})$ is injective.

For $p=0$, there is nothing to prove. Suppose that $1\le p\le d$ and
consider an element $f\in \A(X,\F^X_n)^{p}$. Suppose that $f=d(g')$,
where $g'\in \A'(X,\F^X_n)^{p-1}$. We want to show that there exists
$g\in \A(X,\F^X_n)^{p-1}$ such that $d(g)=f$. We prove this by
induction on the maximal depth $l$, $-1\le l\le p-1$, of types for
non-zero components of $g'$. Recall that for any adele $h$ and an
increasing sequence of natural numbers $(j_0\ldots j_q)$, by
$h_{j_0\ldots j_q}$ we denote the component of $h$ that has type
$(j_0\ldots j_q)$.

Suppose that $l=-1$. Let $(i_0\ldots i_{p-1})$ be a sequence such that
$g'_{i_0\ldots i_{p-1}}\ne 0$; then $i_0>0$ and we have
$g'_{i_0\ldots i_{p-1}}=f_{0i_0\ldots i_{p-1}}\in
\A((0i_0\ldots i_{p-1}),\F^X_n)$. By Corollary \ref{intersect2},
$g'_{i_0\ldots i_{p-1}}\in\A((i_0\ldots i_{p-1}),\F_n^X)$ and thus
$g'\in\A(X,\F^X_n)^{p-1}$.

Now suppose that $l\ge 0$. Let $(0\ldots l i_{l+1}\ldots i_{p-1})$
be a sequence such that \mbox{$g'_{0\ldots li_{l+1}\ldots
i_{p-1}}\ne 0$}. Suppose that $l<p-1$.
Combining Proposition \ref{residue} and Corollary
\ref{corol-Gerst}$(ii)$, we get that the element $\nu_{0\ldots
l+1}(f_{0\ldots l+1i_{l+1}\ldots i_{p-1}})$ belong to the group
$$
\mbox{$\bigoplus\limits_{\eta_{l+1}\in X^{(l+1)}}
\A((0(i_{l+1}-l-1)\ldots(i_{p-1}-l-1)),\F(n-d)_{d-l-1}^{\overline{\eta}_{l+1}})$}
\subset
$$
$$
\subset \prod_{\eta_{i_{l+1}}\ldots\eta_{i_p}}
\left(\mbox{$\bigoplus\limits_{\eta_{l+1}\in X^{(l+1)}}$}
F_{n-l-1}(k(\eta_{l+1}))\right),
$$
where $d=\dim(X)$ and the product is taken over all flags
$\eta_{i_{l+1}}\ldots\eta_{i_p}$ of type $(i_{l+1}\ldots i_p)$ and
$\eta_{i_{l+1}}\in\overline{\eta}_{l+1}$. On the other hand, the
reciprocity law implies that
$$
\nu_{0\ldots l+1}(f_{0\ldots l+1i_{l+1}\ldots i_{p-1}})=
\nu'_{0\ldots l+1}(-1)^{l+1}g'_{0\ldots l i_{l+1}\ldots i_{p-1}}.
$$
By Lemma \ref{approximation}, there exists an element
$g^0\in\A((0\ldots l i_{l+1}\ldots i_{p-1}),\F_n^X)$ such that
$$
\nu'_{0\ldots l+1}(g'_{0\ldots l i_{l+1}\ldots i_{p-1}}-g^0)=0.
$$
By Lemma \ref{cobound}, there exists an element
$h'\in\prod\limits_{0\le i\le l}
\A'((0\ldots \hat{\imath}\ldots l i_{l+1}\ldots i_{p-1}),\F_n^X)$
such that $d(h')_{0\ldots li_{l+1}\ldots i_{p-1}}=
g'_{0\ldots li_{l+1}\ldots i_{p-1}}- g^0$. Note that
$d(g'-g^0-d(h'))\in\A(X,\F^X_n)^p$ and the adele $g'-g^0-d(h')$
has a strictly less nonzero depth $l$ components than the adele $g'$.
Therefore, by the inductive assumption,
there exists an element $g^1\in \A(X,\F^X_n)^{p-1}$ such that
$d(g^1)=d(g'-g^0-d(h'))$ and we put $g=g^0+g^1$.

Now suppose that $l=p-1$; then, by Lemma \ref{surjective-nu}, there
exists an element $g^0\in\A((0\ldots p-1),\F^X_n)$ such that
$\nu'_{0\ldots p-1}(g'-g^0)=0$. By Lemma~\ref{cobound}, there exists
an element $h'\in\prod\limits_{0\le i<
p-1}\A'((0\ldots\hat{\imath}\ldots p-1),\F^X_n)$ such that
$d(h')_{0\ldots p-1}=(g'-g^0)_{0\ldots p-1}$. Note that
$d(g'-g^0-d(h'))\in\A(X,\F^X_n)^p$ and the adele $g'-g^0-d(h')$ has
no components of depth $p-1$. Therefore, by the inductive
assumption, there exists an element $g^{1}\in \A(X,\F^X_n)^{p-1}$
such that $d(g^1)=d(g'-g^0-d(h'))$ and we put $g=g^0+g^1$.
\end{proof}

The essential part in the proof of Theorem \ref{quasiis} is the
following approximation type lemma.

\begin{lemma}\label{approximation}
Under the assumptions from Theorem \ref{quasiis}, consider an adele
\mbox{$f\in \A'((i_0\ldots i_p),\F^X_n)$} such that the depth of the
sequence $(i_0\ldots i_p)$ is $l<p$ and
$$
\nu'_{0\ldots (l+1)}(f)\in
\mbox{$\bigoplus\limits_{\eta_{l+1}\in X^{(l+1)}}
\A((0(i_{l+1}-l-1)\ldots(i_p-l-1)),\F(n-d)_{d-l-1}^{\overline{\eta}_{l+1}})$}\subset
$$
$$
\subset \prod_{\eta_{i_{l+1}}\ldots\eta_{i_p}}
\left(\mbox{$\bigoplus\limits_{\eta_{l+1}\in X^{(l+1)}}$}
F_{n-l-1}(k(\eta_{l+1}))\right),
$$
where $d=\dim(X)$ and
the product is taken over all flags $\eta_{i_{l+1}}\ldots\eta_{i_p}$ of type
$(i_{l+1}\ldots i_p)$ and with $\eta_{i_{l+1}}\in\overline{\eta}_{l+1}$.
Then there
exists an adele $g\in \A((i_0\ldots i_p),\F_n^X)$ of the same
type as $f$ such that $\nu_{0\ldots(l+1)}(g)=\nu'_{0\ldots(l+1)}(f)$.
\end{lemma}
\begin{proof}
During the proof $\eta_s$ denotes a
schematic point on $X$ of codimension $s$ in $X$
(though sometimes points are considered on proper closed subvarieties in $X$).
Further, $d_{\eta_s}$ is the differential
in the local Gersten resolution at $\eta_s$, i.e., in the complex
$\G(X_{\eta_s},F_*,n)^{\bullet}$. For any two subvarieties $C_1$, $C_2$ in
$X$, denote by $C_1-C_2$ the union of all irreducible components of $C_1$ that
are not contained in $C_2$. Notice that for any two subvarieties $C_1$, $C_2$ in $X$,
we have $(C_1-C_2)\cup C_2=C_1\cup C_2$.

The proof is in two steps.

{\it Step 1.} Consider the collection of $\A$-adeles
$$
\{h_{\eta_{l+1}(i_{l+1}-l-
1)\ldots(i_p-l-1)}\}=\nu'_{0\ldots (l+1)}(f)\in
\mbox{$\bigoplus\limits_{\eta_{l+1}\in
X^{(l+1)}}\A((0(i_{l+1}-l-1)\ldots (i_p-l-1)),\F_{n-l-1}^{\overline{\eta}_{l+1}})$}.
$$
Let $Z_{l+1}=\cup\overline{\eta}_{l+1}$
be the union of the closures over the finite set of
schematic points $\eta_{l+1}\in X^{(l+1)}$ such that
$h_{\eta_{l+1}(i_{l+1}-l-1)\ldots(i_p-l-1)}$ is a non-zero adele on
$\overline{\eta}_{l+1}$. For
each schematic point $\eta_{l+1}\in Z^{(0)}_{l+1}$ let
$\{D_{\eta_{l+1}},D_{\eta_{l+1}\eta_{i_{l+1}}\ldots\eta_{i_{k}}}\}$, $l+1\le k< p$
be the system of divisors on $\overline{\eta}_{l+1}$
arising from the adelic condition for the $\A$-adele
$h_{\eta_{l+1}(i_{l+1}-l-1)\ldots(i_p-l-1)}$
on $\overline{\eta}_{l+1}$ (see Proposition \ref{systemofd}).

\begin{prop}\label{systemofc}
For any flag $\eta_{i_{l+1}}\ldots\eta_{i_k}$, $l+1\le k<p$ on $Z_{l+1}$, there
exists an equidimensional subvariety $Z_{l+1;\eta_{i_{l+1}}\ldots\eta_{i_{k}}}\subset
X$ of codimension $l+1$ in $X$ such that the system of subvarieties
$\{Z_{l+1;\eta_{i_{l+1}}\ldots\eta_{i_{k}}}\}$, $l+1\le k<p$ on $X$
satisfies the following conditions:
\begin{itemize}
\item[(i)]
for any point $\eta_{i_{l+1}}$ on $Z_{l+1}$, we have
$Z_{l+1;\eta_{i_{l+1}}}(\eta_{i_{l+1}})\subseteq Z_{l+1}(\eta_{i_{l+1}})$ and
for any flag $\eta_{i_{l+1}}\ldots\eta_{i_k}$, $l+1<k<p$ on $Z_{l+1}$ we have
$Z_{l+1;\eta_{i_{l+1}}\ldots\eta_{i_{k}}}(\eta_{i_{k}})\subseteq
Z_{l+1;\eta_{i_{l+1}}\ldots\eta_{i_{k-1}}}(\eta_{i_{k}})$;
\item[(ii)]
for any flag $\eta_{i_{l+1}}\ldots\eta_{i_{k}}$, $l+1\le k<p$ on $Z_{l+1}$,
the subvariety
$Z_{l+1;\eta_{i_{l+1}}\ldots\eta_{i_{k}}}-Z_{l+1}$ contains the equidimensional subvariety
$$
E_{\eta_{i_{l+1}}\ldots\eta_{i_k}}=
\bigcup_{\eta_{l+1}\in Z_{l+1}^{(0)}}(D_{\eta_{l+1}\eta_{i_{l+1}}
\ldots\eta_{i_{k}}}-D_{\eta_{l+1}})
$$
of codimension $l+2$ in $X$ and the pair
$(E_{\eta_{i_{l+1}}\ldots\eta_{i_k}},
Z_{l+1;\eta_{i_{l+1}}\ldots\eta_{i_{k}}}-Z_{l+1})$ is strongly locally
effaceable with the freedom degree at least $p-k$.
\end{itemize}
\end{prop}

\begin{proof}
We construct the needed system of subvarieties $\{Z_{l+1;\eta_{i_{l+1}}\ldots\eta_{i_{k}}}\}$
by induction on $k$, $l+1\le k<p$.

Suppose that $k=l+1$. Since for each point $\eta_{l+1}\in Z_{l+1}^{(0)}$
the system of divisors
$\{D_{\eta_{l+1}},D_{\eta_{l+1}\eta_{i_{l+1}}\dots\eta_{i_k}}\}$, $l+1\le k<p$
on $\overline{\eta}_{l+1}$ satisfies condition $(*)$
from Proposition \ref{systemofd},
we get that for any point $\eta_{i_{l+1}}$ on $X$, the defined above subvariety
$E_{\eta_{i_{l+1}}}$ does not contain $\eta_{i_{l+1}}$. For each point $\eta_{i_{l+1}}$
on $Z_{l+1}$, we choose closed points on
each irreducible component of $Z_{l+1}$ and on $\overline{\eta}_{i_{l+1}}$
outside of $E_{\eta_{i_{l+1}}}$ and thus get a finite set of closed points
$T_{\eta_{i_{l+1}}}\subset X\backslash E_{\eta_{i_{l+1}}}$.
The application of Theorem \ref{freely_effaceable} with $f=p-l-1$,
$Z=E_{\eta_{i_{l+1}}}$,
and $T=T_{\eta_{i_{l+1}}}$ yields the existence of an
equidimensional subvariety $z_{\eta_{i_{l+1}}}=\widetilde{Z}$ of codimension
$l+1$ in $X$ such that $z_{\eta_{i_{l+1}}}$
does not contain $\overline{\eta}_{i_{l+1}}$, has no common
irreducible components with $Z_{l+1}$,
and $(E_{\eta_{i_{l+1}}},z_{\eta_{i_{l+1}}})$
is an $(p-l-1)$-s.l.e. pair. We put
$Z_{l+1;\eta_{i_{l+1}}}=Z_{l+1}\cup z_{\eta_{i_{l+1}}}$.

Now we do the induction step from $k-1$ to $k$, $l+1<k<p$. As before,
by condition $(*)$ from Proposition \ref{systemofd},
for any flag $\eta_{i_{l+1}}\ldots\eta_{i_{k}}$ on $Z_{l+1}$, the subvariety
$$
E_{\eta_{i_{l+1}}\ldots\eta_{i_{k}}}-E_{\eta_{i_{l+1}}\ldots\eta_{i_{k-1}}}=
\bigcup_{\eta_{l+1}\in
Z_{l+1}^{(0)}}(D_{\eta_{l+1}\eta_{i_{l+1}}\ldots\eta_{i_k}}-
(D_{\eta_{l+1}\eta_{i_{l+1}}\ldots\eta_{i_{k-1}}}\cup D_{\eta_{l+1}}))
$$
does not contain $\eta_{i_{k}}$. For each flag $\eta_{i_{l+1}}\ldots\eta_{i_k}$,
the application of Proposition \ref{addition} with
$Z_1=E_{\eta_{i_{l+1}}\ldots\eta_{i_{k-1}}}$,
$Z_2=E_{\eta_{i_{l+1}}\ldots\eta_{i_{k}}}-E_{\eta_{i_{l+1}}\ldots\eta_{i_{k-1}}}$,
$\widetilde{Z}_1=Z_{l+1;\eta_{i_{l+1}}\ldots\eta_{i_{k-1}}}-Z_{l+1}$,
$T=Z_{l+1}$, and $C=\overline{\eta}_{i_{k}}$ yields the existence
of an equidimensional
subvariety $z_{l+1;\eta_{i_{l+1}}\ldots\eta_{i_{k}}}=\widetilde{Z}_2$ of
codimension $l+1$ in $X$ such that
$z_{l+1;\eta_{i_{l+1}}\ldots\eta_{i_{k}}}$ does not contain
$\overline{\eta}_{i_{k}}$, has no common irreducible components
with $Z_{l+1}$, and
$$
(E_{\eta_{i_{l+1}}\ldots\eta_{i_{k-1}}}\cup E_{\eta_{i_{l+1}}\ldots\eta_{i_k}},
(Z_{l+1;\eta_{i_{l+1}}\ldots\eta_{i_{k-1}}}-Z_{l+1})\cup
z_{l+1;\eta_{i_{l+1}}\ldots\eta_{i_k}})
$$
is a $(p-k)$-s.l.e. pair. We put $Z_{l+1;\eta_{i_{l+1}}\ldots\eta_{i_{k}}}=
Z_{l+1;\eta_{i_{l+1}}\ldots\eta_{i_{k-1}}}\cup
z_{l+1;\eta_{i_{l+1}}\ldots\eta_{i_{k}}}$. By Remark \ref{subpair},
$$
(E_{\eta_{i_{l+1}}\ldots\eta_{i_k}},
Z_{l+1;\eta_{i_{l+1}}\ldots\eta_{i_{k}}}-Z_{l+1})
$$
is a $(p-k)$-s.l.e. pair and
$Z_{l+1;\eta_{i_{l+1}}\ldots\eta_{i_{k}}}(\eta_{i_k})\subseteq
Z_{l+1;\eta_{i_{l+1}}\ldots\eta_{i_{k-1}}}(\eta_k)$.
\end{proof}

\begin{corol}\label{transforming-collection}
For any flag $\eta_{i_{l+1}}\ldots\eta_{i_p}$ on $Z_{l+1}$, there exists a collection
$$
\{g_{\eta_{l+1};\eta_{i_{l+1}}\ldots\eta_{i_p}}\}\in
\mbox{$\bigoplus\limits_{\eta_{l+1}\in X^{(l+1)}_{\eta_{i_p}}}
F_{n-l-1}(k(\eta_{l+1}))$},
$$
satisfying the following conditions (note that the closure of the point
$\eta_{l+1}$ from the index of $g$
may not contain $\eta_{i_{l+1}}$):
\begin{itemize}
\item[(i)]
$d_{\eta_{i_p}}(\{g_{\eta_{l+1};\eta_{i_{l+1}}\ldots\eta_{i_{p}}}\})=0$;
\item[(ii)]
if $\overline{\eta}_{l+1}$ contains $\eta_{i_{l+1}}$, then
$g_{\eta_{l+1};\eta_{i_{l+1}}\ldots\eta_{i_p}}=
h_{\eta_{l+1}\eta_{i_{l+1}}\ldots\eta_{i_p}}$;
\item[(iii)]
the support of the
collection $\{g_{\eta_{l+1};\eta_{i_{l+1}}\ldots\eta_{i_p}}\}$ is contained in
the subvariety $Z_{l+1;\eta_{i_{l+1}}\ldots \eta_{i_{p-1}}}$.
\end{itemize}
\end{corol}

\begin{proof}
We use the notations from Proposition \ref{systemofc}.
Since the $\A'$-adele $f$
has type $(0\ldots l i_{l+1}\ldots i_p)$, $i_{l+1}>l+1$, by reciprocity law,
we have $d_{\eta_{i_{l+1}}}(\{h_{\eta_{l+1}\eta_{i_{l+1}}\ldots\eta_{i_p}}\})=0$
for any flag $\eta_{i_{l+1}}\ldots\eta_{i_p}$ on $Z_{l+1}$, where
$\{h_{\eta_{l+1}\eta_{i_{l+1}}\ldots\eta_{i_p}}\}$ is considered as a collection from
$\G(X_{\eta_{i_{l+1}}},F_*,n)^{l+1}$.
Therefore the support of
the collection $\{g_{\eta_{l+2};\eta_{i_{l+1}}\ldots\eta_{i_p}}\}=
d_{\eta_{i_{p}}}(\{h_{\eta_{l+1}\eta_{i_{l+1}}\ldots\eta_{i_p}}\})\in
\G(X_{\eta_{i_p}},F_*,n)^{l+2}$
is contained in the subvariety $E_{\eta_{i_{l+1}}\ldots\eta_{i_{p-1}}}$.
Hence, by Propositions \ref{differential} and \ref{systemofc},
there exists a collection $\{g'_{\eta_{l+1};\eta_{i_{l+1}}\ldots\eta_{i_p}}\}\in
\G(X_{\eta_p},F_*,n)^{l+1}$ with
the support on $Z_{l+1;\eta_{i_{l+1}}\ldots\eta_{i_{p-1}}}-Z_{l+1}$ such that
$d_{\eta_{i_p}}(\{g'_{\eta_{l+1};\eta_{i_{l+1}}\ldots\eta_{i_p}}\})=
g_{\eta_{l+2};\eta_{i_{l+1}}\ldots\eta_{i_p}}$.
Finally, we put
$g_{\eta_{l+1};\eta_{i_{l+1}}\ldots\eta_{i_p}}=
\{h_{\eta_{l+1}\eta_{i_{l+1}}\ldots
\eta_{i_p}}\}-\{g'_{\eta_{l+1};\eta_{i_{l+1}}\ldots\eta_{i_p}}\}$.
\end{proof}

{\it Step 2.} By Proposition \ref{patching systems}, there exists a patching
system $\{Z_{r}^{1,2}\}$, $1\le r\le l$ on $X$ for the subvariety $Z_{l+1}$ with
the freedom degree at least $p-l$. We extend this patching system to patching
systems for all subvarieties $Z_{l+1;\eta_{i_{l+1}}\ldots\eta_{i_k}}$, $l+1\le k<p$.

\begin{prop}\label{extension-patching-system}
For each $k$, $l+1\le k<p$, and for each flag $\eta_{i_{l+1}}\ldots\eta_{i_k}$
on $Z_{l+1}$, there exists a patching system
$Z^{1,2}_{r;\eta_{i_{l+1}}\ldots\eta_{i_k}}$, $1\le r\le l$ on $X$
for the subvariety $Z_{l+1;\eta_{i_{l+1}}\ldots\eta_{i_k}}$ with the freedom degree at least $p-k$, satisfying the following condition.
Put $D=Z^1_{1}\cup Z^2_1$ and
$D_{\eta_{i_{l+1}}\ldots\eta_{i_k}}=Z^1_{1;\eta_{i_{l+1}}\ldots\eta_{i_k}}\cup
Z^2_{1;\eta_{i_{l+1}}\ldots\eta_{i_k}}$ for any flag
$\eta_{i_{l+1}}\ldots\eta_{i_k}$, $l+1\le k< p$ on $Z_{l+1}$. Then for any point
$\eta_{i_{l+1}}$ on $Z_{l+1}$, we have
$D_{\eta_{i_{l+1}}}(\eta_{i_{l+1}})\subseteq D(\eta_{i_{l+1}})$
and for any flag $\eta_{i_{l+1}}\ldots\eta_{i_k}$, $l+1<k<p$ on $Z_{l+1}$, we have
$
D_{\eta_{i_{l+1}}\ldots\eta_{i_k}}(\eta_{i_k})\subseteq
D_{\eta_{i_{l+1}}\ldots\eta_{i_{k-1}}}(\eta_{i_k}).
$
\end{prop}
\begin{proof}
We use the notations from the proof of Proposition \ref{systemofc}. The proof
is by double induction on $k$ and $r$, $l+1\le k<p$, $1\le r\le l$ (the
induction on $r$ is decreasing, as in the proof of Proposition \ref{patching
systems}).

Suppose that $k=l+1$, $r=l$. For each point $\eta_{i_{l+1}}$ on $Z_{l+1}$, the
application of Proposition \ref{addition} with $Z_1=Z_{l+1}$,
$Z_2=z_{l+1;\eta_{i_{l+1}}}$, $\widetilde{Z}_1=Z_{l}^1$, $T=Z^2_{l}$, and
$C=\overline{\eta}_{i_{l+1}}$ yields the existence of an equidimensional subvariety
$z^1_{l;\eta_{i_{l+1}}}=\widetilde{Z}_2\subset X$ of codimension $l$ in $X$
such that $z^1_{l;\eta_{i_{l+1}}}$ does not contain
$\overline{\eta}_{i_{l+1}}$, has no common irreducible components with
$Z^2_l$, and $(Z_{l+1;\eta_{i_{l+1}}}, Z_{l}^1\cup
z^1_{l;\eta_{i_{l+1}}})$ is a $(p-l-1)$-s.l.e. pair.
We put $Z^1_{l;\eta_{i_{l+1}}}=Z^1_l\cup
z^1_{l;\eta_{i_{l+1}}}$. Using Proposition \ref{addition} with $Z_1=Z_{l+1}$,
$Z_2=z_{l+1;\eta_{i_{l+1}}}$, $\widetilde{Z}_1=Z_{l}^2$,
$T=Z_{l+1;\eta_{i_{l+1}}}^1$, and $C=\overline{\eta}_{i_{l+1}}$, we get an
equidimensional subvariety $z_{l;\eta_{i_{l+1}}}^2=\widetilde{Z}_2\subset X$ of
codimension $l$ in $X$ such that $z_{l;\eta_{i_{l+1}}}^2$ does not contain
$\overline{\eta}_{i_{l+1}}$, has no common irreducible components with
$Z_{l+1;\eta_{i_{l+1}}}^1$, and $(Z_{l+1;\eta_{i_{l+1}}},Z^2_l\cup
z^2_{l;\eta_{i_{l+1}}})$ is a $(p-l-1)$-s.l.e. pair.
We put $Z^2_{l;\eta_{i_{l+1}}}=Z^2_l\cup
z^2_{l;\eta_{i_{l+1}}}$.

The induction step from $r+1$ to $r$ for $k=l+1$, $1\le r< l$ is analogous
with $Z_{l+1}$ replaced by $Z^1_{r+1}\cup Z^2_{r+1}$, $z_{l+1;\eta_{i_{l+1}}}$
replaced by $z^1_{r+1;\eta_{i_{l+1}}}\cup z^2_{r+1;\eta_{i_{l+1}}}$, and
$Z^j_{l}$ replaced by $Z^j_r$, $j=1,2$.

The reasoning for arbitrary $k$, $l+1<k<p$ is the same as for $k=l+1$ with the
subvarieties $\overline{\eta}_{i_{l+1}}$, $Z_{l+1;\eta_{i_{l+1}}}$,
$z_{l+1;\eta_{i_{l+1}}}$ replaced by the subvarieties $\overline{\eta}_{i_k}$,
$Z_{l+1;\eta_{i_{l+1}}\ldots\eta_{i_k}}$,
$z_{l+1;\eta_{i_{l+1}}\ldots\eta_{i_k}}$, respectively, for each flag
$\eta_{i_{l+1}}\ldots\eta_{i_k}$ on $Z_{l+1}$ and with the patching system
$\{Z^{1,2}_{r}\}$, $1\le r\le l$ replaced by the inductively defined patching
system $\{Z^{1,2}_{r;\eta_{i_{l+1}}\ldots\eta_{i_{k-1}}}\}$.

By construction, the system of divisors
$\{D,D_{\eta_{i_{l+1},\ldots,\eta_{i_k}}}\}$, $l+1\le k<p$ on $X$ satisfies
the needed condition.
\end{proof}

\begin{corol}\label{patching-collection}
For any flag $(\eta_{i_{l+1}}\ldots\eta_{i_p})$ on $Z_{l+1}$, there exists a collection
$$
\{g_{\eta_{0};\eta_{i_{l+1}}\ldots\eta_{i_p}}\}\in
\mbox{$\bigoplus\limits_{\eta_{0}\in X^{(0)}}
F_{n}(k(\eta_{0}))$},
$$
such that $g_{\eta_{0};\eta_{i_{l+1}}\ldots\eta_{i_p}}\in
(\F_n^{X\backslash D_{\eta_{i_{l+1}}\ldots \eta_{i_{p-1}}}})_{\eta_{i_p}}$
and
$$
d_{\eta_{i_{l+1}}}\nu_{XZ^1_{1;F}\ldots Z^1_{l;F}}(\{g_{\eta_{0};\eta_{i_{l+1}}\ldots
\eta_{i_p}}\})=\{h_{\eta_{l+1}\eta_{i_{l+1}}\ldots\eta_{i_p}}\}.
$$
\end{corol}
\begin{proof}
The corollary follows from the direct application of
Proposition \ref{differential-patching-systems}
for the collection $\{g_{\eta_{l+1};\eta_{i_{l+1}}\ldots\eta_{i_k}}\}$
from Corollary \ref{transforming-collection} and the
patching system $\{Z^{1,2}_{r;\eta_{i_{l+1}}\ldots\eta_{i_{p-1}}}\}$, $1\le r\le l$ from Proposition
\ref{extension-patching-system}.
\end{proof}

Now we are ready to define the needed adele $g\in\A((i_0\ldots i_p),\F_n^X)$.
Let $\eta_0\ldots\eta_l\eta_{i_{l+1}}\ldots\eta_{i_p}$ be a flag of type
$(i_0\ldots i_p)$ on $X$. Then we put
$g_{\eta_0\ldots\eta_l\eta_{i_{l+1}}\ldots\eta_{i_p}}=
g_{\eta_0;(\eta_{i_{l+1}}\ldots\eta_{i_p})}$
if $\eta_{i_{l+1}}\ldots\eta_{i_p}$ is a flag on $Z_{l+1}$ and
$\eta_r\in (Z_{r;F}^1)^{(0)}$ for all
$1\le r\le l$. Otherwise we put
$g_{\eta_0\ldots\eta_l\eta_{i_{l+1}}\ldots \eta_p}=0$.

For any flag $\eta_0\ldots\eta_k$ of type $(i_0\ldots i_k)$, $0\le k\le l$ on $X$,
we put $D_{\eta_0\ldots\eta_k}=D$. For any flag
$\eta_0\ldots\eta_l\eta_{i_{l+1}}\ldots\eta_{i_k}$ of type
$(i_0\ldots i_k)$ on $X$, $l+1\le k< p$, we put
$D_{\eta_0\ldots\eta_{l}\eta_{i_{l+1}}\ldots\eta_k}=
D_{\eta_{i_{l+1}}\ldots\eta_{i_k}}$ if $\eta_{i_{l+1}}\ldots\eta_{i_k}$ is a
flag on $Z_{l+1}$. Otherwise we put
$D_{\eta_0\ldots\eta_{l}\eta_{i_{l+1}}\ldots\eta_k}=\emptyset$.

By Proposition \ref{extension-patching-system},
the system of divisors $D_{\eta_0\ldots\eta_k}$, $0\le k< p$ on $X$
satisfies condition $(*)$ from Proposition \ref{systemofd}.
By Corollary \ref{patching-collection}, the distribution
$g$ satisfies the adelic condition with respect to the system of
divisors $D_{\eta_0\ldots\eta_k}$, $0\le k< p$ on $X$
and we have $\nu_{0\ldots (l+1)}(g)=\nu'_{0\ldots (l+1)}(f)$.

This concludes the proof of Lemma \ref{approximation}.
\end{proof}

\subsection{Explicit cocycles}\label{explicit-classes}

In this section we construct explicitly certain cocycles in the
adelic complex corresponding to the given cocycles in the Gersten
complex. Suppose $k$ is an infinite perfect field, $F_*$ is an
l.a.f. homology theory over $k$, $X$ is an irreducible smooth
variety over the field $k$, and $Y\subset X$ is an equidimensional
subvariety of codimension $p$ in $X$. Consider a collection
$\{f_y\}\in \bigoplus\limits_{y\in Y^{(0)}}F_m(k(y))$ such that
$\{f_{\eta}\}$ is a cocycle in the Gersten complex
$\G(X,F_*,p+m)^{\bullet}$ on $X$. By Remark
\ref{global_patching_systems}, there exists a patching system
$\{Y^{1,2}_r\}$, $1\le r\le p-1$ on $X$ for the subvariety $Y$ with
the freedom degree at least zero. We put $Y^1_p=Y$.

\begin{prop}\label{adelic-class}
There exists an adele $f=[\{f_y\}]\in \A(X,\F^X_{p+m})^p$ such that $f$ is a
cocycle in the adelic complex $\A(X,\F^X_{p+m})^{\bullet}$ with
$\nu_X(f)=\{f_y\}$, where
$\nu_X:\A(X,\F^X_{p+m})^{\bullet}\to\G(X,F_*,p+m)^{\bullet}$
is the morphism from Theorem \ref{quasiis}, and $f$ satisfies the
following conditions for any flag $\eta_{i_0}\ldots\eta_{i_p}$ on $X$:
\begin{itemize}
\item[(i)]
$f_{\eta_{i_0}\ldots\eta_{i_p}}=0$ unless $\eta_{i_r}\in Y_r^1$ for all $r,1\le r\le
p$;
\item[(ii)]
suppose that $\eta_{i_0}\notin Y^1_1$, $\eta_{i_r}\in Y_r^1$,
$\eta_{i_r}\notin Y^2_{r}$ for all $r,1\le r\le p-1$ and
$\eta_{i_p}\in Y$; then
$f_{\eta_{i_0}\ldots\eta_{i_p}}=\widetilde{f}_{\eta_{i_p}}\in
F_{p+m}(k(X))$, where $\widetilde{f}_{\eta_{i_p}}$ depends only on
$\eta_{i_p}$ and satisfies the condition
$d_{\eta_{i_p}}\nu_{XY_1^1\ldots
Y_{p-1}^1}(\widetilde{f}_{\eta_{i_p}})=\{f_y\}_{\eta_{i_p}}$. Here
$d_{\eta_{i_p}}$ is the differential in the Gersten resolution on
$X_{\eta_{i_p}}=\Spec(\OO_{X,\eta_{i_p}})$, the notation $\nu_{X
Y^1_1\ldots Y^1_{p-1}}$ was introduced in Section
\ref{patching-systems-section}, and the index $\eta_{i_p}$ by a
collection means that we consider the restriction of the collection
to $X_{\eta_{i_p}}$;
\item[(iii)]
we have $\sing(\nu_{XY_1^1\ldots Y^1_{r-1}}(f_{\eta_{i_0}\ldots\eta_{i_{p}}}))
\subset Y_r^1\cup Y_r^2$ for all $r,1\le r\le l$, where $l$ is the depth of the
type $(i_0\ldots i_p)$.
\end{itemize}
\end{prop}

\begin{rmk}
Since $d(f)=0$, we have $\nu_{0\ldots(l+1)}(d(f))=0$ for any
integer $l,0\le l\le p-1$. Explicitly, for any flag
$\eta_{i_1}\ldots\eta_{i_{p-l+1}}$ of type
$(i_1\ldots i_{p-l+1})$ on $X$ with $i_1>l$, we have
$$
0=d_{\eta_{i_1}}\nu_{0\dots l}(\sum_{j=1}^{p-l+1}(-1)^{j+1}f_{01\ldots l
\eta_{i_1}\ldots\hat{\eta}_{i_{j}}\ldots\eta_{i_{p-l+1}}})\in
\mbox{$\bigoplus\limits_{\eta\in
X^{(l+1)}_{\eta_{i_1}}}F_{p+m-l-1}(k(\eta))$},\eqno (**)
$$
where for any flag $\Phi$ of
type $T$ on $X$, the index $(0\ldots l \Phi)$ means that we consider the set of all
flags $\eta_0\ldots\eta_{l}\Phi$ of type $(0\ldots l T)$ on $X$ with the
fixed $T$-part $\Phi$.
\end{rmk}

\begin{proof}[Proof of Proposition \ref{adelic-class}]

We define the components of the
adele $f$ by decreasing induction on the depth $l,-1\le l\le p$
of the type of a component. Moreover, we enlarge the induction hypothesis by
condition $(**)$.

Let $\eta_0\ldots\eta_{p-1}\eta_i$ be a flag on $X$ with the type depth
$l\ge p-1$, i.e.,
$i\ge p$. If $\eta_i\notin Y$, then we put
$f_{\eta_0\ldots\eta_{p-1}\eta_i}=0$. Suppose that $\eta_i\in Y$.
Then, by Proposition \ref{differential-patching-systems},
there exists an element $\widetilde{f}_{\eta_i}\in F_{p+m}(k(X))$ such that
$\sing(\nu_{XY^1_1\ldots Y_{r-1}^1}(\widetilde{f}_{\eta_i}))\subset Y_r^1\cup Y_r^2$
for all $r,1\le r\le p-1$ and
$$
d_{\eta_i}\nu_{X Y_1^1\ldots
Y_{p-1}^1}(\widetilde{f}_{\eta_i}) =\{f_y\}_{\eta_i}\in
\mbox{$\bigoplus\limits_{\eta\in X^{(p)}_{\eta_{i}}}F_{m}(k(\eta))$}.
$$
We put
$f_{\eta_0\ldots\eta_{p-1}\eta_i}=(-1)^{\frac{p(p+1)}{2}}\widetilde{f}_{\eta_i}$ if $\eta_r$ is
a generic point of some irreducible component of
$Y_r^1$ for all $r,1\le r\le p-1$. Otherwise, we put
$f_{\eta_0\ldots\eta_{p-1}\eta_i}=0$. It is readily seen that
conditions $(i)$, $(ii)$ are satisfied for the defined above
$(0\ldots p-1, i)$-type component of $f$,
and also condition $(**)$ holds for $l=p-1$ and any flag
$\eta_{i_1}\eta_{i_{2}}$ of type $(i_1 i_2)$ on $X$ such
that $i_1>p-1$.

Now we do the induction step from $l+1$ to $l$, $0\le l\le p-2$. Let
$\eta_0\ldots\eta_{l}\eta_{i_1}\ldots \eta_{i_{p-l}}$ be a flag of
type $(0\ldots l i_1 \ldots i_{p-l})$ on $X$, $i_1>l+1$. If
$\eta_{i_1}\notin Y_{l+1}^1$, then we put
$f_{\eta_0\ldots\eta_{l}\eta_{i_1}\ldots \eta_{i_{p-l}}}=0$. Suppose
that $\eta_{i_1}\in Y_{l+1}^1$. Then, by the inductive assumption
and by Proposition \ref{differential-patching-systems} applied to
the patching system $\{Y^{1,2}_r\}$, $1\le r\le l$ on $X$ for the
subvariety $Y^1_{l+1}$ and the point $\eta_{i_1}$, there exists an
element $\widetilde{f}_{\eta_{i_1}\ldots\eta_{i_{p-l}}} \in
F_{p+m}(k(X))$ such that $\sing(\nu_{XY_1^1\ldots
Y^1_{r-1}}(\widetilde{f}_{\eta_{i_1}\ldots\eta_{i_{p-l}}})) \subset
Y_r^1\cup Y_r^2$ for all $r,1\le r\le l$ and we have
$$
d_{\eta_{i_1}}\nu_{X Y_1^1\ldots
Y^1_l}(\widetilde{f}_{\eta_{i_1}\ldots\eta_{i_{p-l}}})=
\nu_{l+1}(\sum_{j=1}^{p-l}(-1)^{j+1}f_{01\ldots l+1
\eta_{i_1}\ldots\hat{\eta}_{i_{j}}\ldots\eta_{i_{p-l}}})_{\eta_{i_1}}\in
\mbox{$\bigoplus\limits_{\eta\in X^{(l+1)}_{\eta_{i_1}}}F_{p+m-l-1}(k(\eta))$}.
$$
We put
$f_{\eta_0\ldots\eta_{l}\eta_{i_1}\ldots \eta_{i_{p-l}}}=
\widetilde{f}_{\eta_{i_1}\ldots \eta_{i_{p-l}}}$ if $\eta_r$ is
a generic point of some irreducible component of
$Y_r^1$ for all $r,1\le r\le l$. Otherwise, we put
$f_{\eta_0\ldots\eta_{l}\eta_{i_1}\ldots \eta_{i_{p-l}}}=0$.

In the above notation suppose that
$\eta_{i_r}\notin Y^1_{l+r}$ for some $r,1<r\le p-l$. Since for any $j,1\le j<r$, we have
$\eta_{i_r}\notin Y^1_{l+r}$
and for any $j,r\le j\le p-l$, we have $\eta_{i_{r-1}}\notin Y^1_{l+r}$, by
the induction hypothesis, we get that $f_{0\ldots l+1\eta_{i_1}\ldots\hat{\eta}_{i_j}
\ldots\eta_{i_{p-l}}}=0$ for any $j,1\le j\le p-l$. Therefore we may put
$\widetilde{f}_{\eta_{i_1}\ldots\eta_{i_{p-l}}}=0$ and hence
the $(0\ldots l i_1\ldots i_{p-l})$-type component of $f$ satisfies condition $(i)$.

Further, suppose that $\eta_{i_r}\in Y^1_{l+r}$, $\eta_{i_r}\notin Y^2_{l+r}$
for all $r,1\le r\le p-1$ and $\eta_{i_p}\in Y$; then $\eta_{i_1}\notin
Y^1_{l+2}$ and, by the inductive hypothesis, $f_{0\ldots l+1\eta_{i_1}\ldots
\hat{\eta}_{i_j}\ldots\eta_{i_{p-l}}}=0$ for all $j,1<j\le p-l$ and
$\widetilde{f}_{0\ldots l+1\eta_{i_2}\ldots\eta_{i_{p-l}}}=
\widetilde{f}_{\eta_{i_{p-l}}}$, where
$\widetilde{f}_{\eta_{i_{p-l}}}\in F_{p+m}(k(X))$ satisfies condition $(ii)$.
Therefore we have the condition
$$
d_{\eta_{i_1}}\nu_{X Y_1^1\ldots
Y^1_l}(\widetilde{f}_{\eta_{i_1}\ldots\eta_{i_{p-l}}})=
\nu_{X Y^1_{1}\ldots Y^1_{l+1}}
(\widetilde{f}_{\eta_{i_{p-l}}}).
$$
By the inductive hypothesis, we have
$\sing(\nu_{XY^1_1\ldots Y^1_l}(\widetilde{f}_{\eta_{i_{p-l}}}))
\subset Y^1_{l+1}\cup Y^2_{l+1}$. Therefore we may put
$\widetilde{f}_{\eta_{i_1}\ldots\eta_{i_{p-l}}}=
\widetilde{f}_{\eta_{i_{p-l}}}$ and hence
the $(0\ldots l i_1\ldots i_{p-l})$-type component of $f$ satisfies
condition $(ii)$.

As above, it is a trivial check that condition $(**)$ holds for $l-1$ and any flag
$\eta_{i_1}\ldots\eta_{i_{p-l-1}}$ of type $(i_1\ldots i_{p-l-1})$ on $X$ such
that $i_1>l-1$.

Finally, we put $f_{\eta_{i_0}\ldots\eta_{i_p}}=
\displaystyle\sum_{j=0}^p(-1)^j f_{0\eta_{i_0}\ldots
\hat{\eta}_{i_{j}}\ldots\eta_{i_p}}$ for any flag $\eta_{i_0}\ldots\eta_{i_p}$
of type $(i_0\ldots i_p)$ on $X$ with $i_0>0$. By the induction hypothesis, we have
$f_{\eta_{i_0}\ldots\eta_{i_p}}\in(\F_{p+m}^X)_{\eta_{i_0}}$.
The same reasoning as above shows that conditions $(i)$ and $(ii)$ hold for the
$(i_0\ldots i_p)$-type component of $f$.

Since $\sing(f_{\eta_{i_0}\ldots\eta_{i_p}})\subset Y_1^1\cup Y_1^2$
for any flag $\eta_{i_0}\ldots\eta_{i_p}$ on $X$, we see that the distribution $f$ satisfies
the adelic condition with respect to the constant system of divisors $Y^1_1\cup Y^2_1$.
Thus we have defined the needed cocycle $f\in\A(X,\F_{p+m}^X)^p$.
\end{proof}

The following claim is necessary for the proof of Theorem
\ref{intersecting_cycles}.

\begin{claim}\label{integrality2}
Under the above assumptions, consider a schematic point $\eta\in Y$
and suppose that the cocycle $\{f_y\}_{\eta}\in
\G(X_{\eta},F_*,p+m)^p$ in the local Gersten resolution on
$X_{\eta}$ is the restriction of an element $\alpha\in
F_m(Y_{\eta})$. Then the element $\widetilde{f}_{\eta}\in
F_{p+m}(k(X))$ in condition $(ii)$ from Proposition
\ref{adelic-class} may be chosen such that the collection
$d_{\eta}\nu_{XY_1^1\ldots Y_{r-1}^1}({\widetilde{f}_{\eta}})$ is
the restriction of an element $\alpha_r\in F_{p+m-r}((Y^1_r\cup
Y^2_r)_{\eta})$ for all $r,1\le r\le p-1$. Moreover, for all $r,1\le
r< p-1$, the restriction of $\alpha_{r}$ to $F_{p+m-
r}((Y^1_r)_{\eta}\backslash Y^2_r)$ is equal to the restriction of
an element from $F_{p+m-r}((Y^1_r)_{\eta}\backslash (Y^1_{r+1}\cup
Y^2_{r+1}))$ and the restriction of $\alpha_{p-1}$ to
$F_{p+1}((Y^1_{p-1})_{\eta}\backslash Y^2_{p-1})$ is equal to the
restriction of an element from $F_{p+1}((Y^1_{p-1})_{\eta}\backslash
Y)$. Finally, $\widetilde{f}_{\eta}$ is the restriction of an
element $\alpha_0\in F_{p+m}(X_{\eta}\backslash (Y_1^1\cup Y_1^2))$.
\end{claim}

\begin{proof}
The proof is by decreasing induction on $r$, $1\le r\le p-1$.

Suppose that $r=p-1$. Since $\{Y^{1,2}_r\}$ is a patching system
with the freedom degree at least zero, we see that the natural maps
$F_m(Y_{\eta})\to F_m((Y^1_{p-1})_{\eta})$ and $F_m(Y_{\eta})\to
F_m((Y^2_{p-1})_{\eta})$ are equal to zero. Hence there are elements
$\alpha_{p-1}^i\in F_{m+1}((Y_{p-1}^i)_{\eta}\backslash Y)$, $i=1,2$
such that their coboundary is equal to $\alpha\in F_m(Y_{\eta})$.
The localization sequence associated to the closed embedding
$(Y^1_{p-1}\cap Y^2_{p-1}\hookrightarrow Y^1_{p-1}\cup Y^2_{p-1})$
implies that both elements $\alpha^1_{p-1}$ and $\alpha^2_{p-1}$ are
restrictions of an element $\alpha_{p-1}\in F_{m+1}((Y^1_{p-1}\cup
Y^2_{p-1})_{\eta})$.

The induction step from $r+1$ to $r$, $1\le r< p-1$ is analogous to the case $r=p-1$
with the subvarieties $Y$ and $Y^{1,2}_{p-1}$ replaced by the subvarieties
$Y^1_{r+1}\cup Y^2_{r+1}$
and $Y^{1,2}_r$, respectively. At the end, for $r=0$, we repeat the same with $Y$ and
$Y^{1,2}_{p-1}$ replaced by $Y^1_1\cup Y^2_1$ and $X$, respectively.
\end{proof}

\begin{rmk}
The condition from Claim \ref{integrality2} is satisfied for all
$\eta\in X^{(p)}$. Indeed, in this case one puts $\alpha=f_{\eta}\in
F_m(k(\eta))=F_m(Y_{\eta})$.
\end{rmk}

\begin{defin}\label{defin-goodcocycle}
Let $\{f_y\}\in\G(X,F_*,p+m)^p$ be a cocycle in the Gersten complex,
and $\{Y^{1,2}_r\}$, $1\le r\le p-1$ be a patching system on $X$ for
the support $Y$ of $\{f_y\}$ with the freedom degree ar least zero;
then a cocycle $[\{f_y\}]\in\A(X,\F^X_{p+m})^p$ is called a {\it
good cocycle} for $\{f_y\}\in\G(X,F_*,p+m)^p$ with respect to the
patching system $\{Y^{1,2}_r\}$, $1\le r\le p-1$, if it satisfies
all conditions from Proposition~\ref{adelic-class} and
Claim~\ref{integrality2} (for each point $\eta\in Y$).
\end{defin}

It follows from Proposition~\ref{adelic-class} and
Claim~\ref{integrality2} that good cocycles always exist.

\begin{claim}\label{claim-adelicgoodcocycles}
Let $X$ be a smooth variety over an infinite perfect field; then for
any cocycle \mbox{$\{f_y\}\in \G(X,F_*,m)^p$} in the Gersten complex
and a patching system $\{Y^{1,2}_r\}$, $1\le r\le p-1$ on $X$ for
the support $Y$ of $\{f_y\}$, there exists a good cocycle
$[\{f_y\}]\in\A(X,\F^X_m)^p$ for $\{f_y\}$ with respect to the
patching system $\{Y^{1,2}_r\}$.
\end{claim}

The next technical lemma illustrates the freedom of choice in
calculations with adeles.

\begin{lemma}\label{suitable_cycles}
Let $X$ be a smooth variety over an infinite perfect field and let
the collection $\{f_y\}\in\G(X,F_*,m)^{p}$ be supported on an
equidimensional subvariety $Y\subset X$. Suppose that
$d\{\widetilde{f}_{\widetilde{y}}\}=\{f_y\}$, where
$\{\widetilde{f}_{\widetilde{y}}\}\in \G(X,F_*,m)^{p-1}$. Suppose
that $f\in \A(X,\F^X_m)^{p}$ is such that $\nu_X(f)=\{f_y\}$ and
$f_U=0$, where $f_U\in \A(U,\F_m^U)^p$ is the restriction of $f$ to
$U=X\backslash Y$. Let $\{Y^{1,2}_r\}$ be a patching system on $X$
for $Y$; then there exists an adele $\widetilde{f}\in
\A(X,\F^X_m)^{p-1}$ such that $d\widetilde{f}=f$,
$\nu_X(\widetilde{f})=\{\widetilde{f}_{\widetilde{y}}\}$ and
$\widetilde{f}_U$ is a good cocycle on $U$ for
$\{\widetilde{f}_{\widetilde{y}}\}_U\in\G(U,F_*,m)^{p-1}$ with
respect to the restriction of the patching system $\{Y^{1,2}_r\}$ to
$U$.
\end{lemma}

\begin{proof}
We put $B^{\bullet}={\rm Ker}(\nu_X:\A(X,\F^X_m)^{\bullet}\to
\G(X,F_*,m)^{\bullet})$. Since by Lemma~\ref{surjective-nu}, $\nu_X$
is surjective, we see that the complex $B^{\bullet}$ is exact by
Theorem \ref{quasiis}. Let $\widetilde{f}_1\in\A(X,\F^X_m)^{p-1}$ be
such that
$\nu_X(\widetilde{f}_1)=\{\widetilde{f}_{\widetilde{x}}\}$. We have
$d(d(\widetilde{f}_1)-f)=0$ and $\nu_X(d(\widetilde{f}_1)-f)=0$,
hence there exists $h\in B^{p-1}$ such that $dh=
d(\widetilde{f}_1)-f$. The adele $\widetilde{f}_2=\widetilde{f}_1-h$
satisfies $d\widetilde{f}_2=f$,
$\nu_X(\widetilde{f}_2)=\{\widetilde{f}_{\widetilde{y}}\}$.

If $p=1$, then the adele $(\widetilde{f}_2)_U$ is a good cocycle for
$\{\widetilde{f}_{\widetilde{y}}\}_U\in\G(U,m)^{p-1}$ on $U$.
Suppose that $p\ge 2$. Let $\widetilde{f}_3\in\A(U,\F^U_m)^{p-1}$ be
a good cocycle for
$\{\widetilde{f}_{\widetilde{y}}\}_U\in\G(U,F_*,m)^{p-1}$ with
respect to the restriction of the patching system $\{Y^{1,2}_r\}$ to
$U$. We have $d_U(\widetilde{f}_3-(\widetilde{f}_2)_U)=-f_U=0$,
$\nu_U(\widetilde{f}_3-(\widetilde{f}_2)_U)=0$. Therefore there
exists $h\in B_U^{p-2}$ such that $d_U
(h)=\widetilde{f}_3-(\widetilde{f}_2)_U$. Here we put
$B^{\bullet}_U=\ker(\nu_U)$ and $d_U$ denotes the differential in
the adelic complex on $U$. Let $h'\in\A(X,\F^X_m)^{p-2}$ be the
extension by zero of $h$ from $U$ to $X$, i.e., we put
$h'_{\eta_0\ldots\eta_{p-2}}=h_{\eta_0\ldots\eta_{p-2}}$ if
$\eta_0\ldots\eta_{p-2}$ is a flag on $U$ and, otherwise, we put
$h'_{\eta_0\ldots\eta_{p-2}}=0$ (see Corollary
\ref{open-local}$(i)$). It follows easily that $h'\in B^{p-2}$ and
the restriction of $h'$ from $X$ to $U$ is equal to $h$. Thus the
adele $\widetilde{f}=\widetilde{f}_2+dh'\in \A(X,\F_m^X)^{p-1}$
satisfies all needed conditions.
\end{proof}

\section{Applications to $K$-cohomology}\label{explicit_products}

\subsection{Generalities on $K$-cohomology and $K$-adeles}\label{section-K-generalities}

Recall several standard facts on sheaves of $K$-groups and $K$-cohomology.

Consider a weak homology theory for Noetherian schemes given by $F_*=K'_*$.
The corresponding Zariski homology sheaves
will be denoted by $\K'_n$, $n\in\Z$. We put $\G(X,n)^{\bullet}=
\G(X,K'_*,n)^{\bullet}$, i.e.,
$\G(X,n)^p=\bigoplus\limits_{\eta\in X^{(p)}}K_{n-p}(k(\eta))$.
For a scheme $X$ and an integer $n\ge 0$,
let $\K^X_n$ be the sheaf associated to the presheaf given by the formula
$U\mapsto K_n(U)$ for any open subset $U\subset X$, where
$K_n(U)=\pi_{n+1}(BQ{\mathcal P}(U))$ and ${\mathcal P}(U)$
is the exact category of coherent locally free sheaves on $U$.
The Zariski cohomology groups $H^{\bullet}(X,\K^X_n)$ are called the
{\it $K$-cohomology} of $X$. Evidently, there is a morphism of sheaves
$\K_n^X\to (\K'_n)^X$ for any $n\ge 0$, which is an isomorphism if $X$
is regular and separated.
The sheaf $\K^X=\bigoplus\limits_{n\ge 0}\K^X_n$ is the sheaf of
supercommutative associative rings. Any morphism of schemes
$f:X\to Y$ defines a homomorphism of sheaves of
algebras $f^*:\K^Y\to f_*\K^X$.

For any integers $m,n\ge 0$, there is a morphism of complexes of
sheaves $\underline{\G}(X,m)^{\bullet}\otimes \K^X_n\to
\underline{\G}(X,m+n)^{\bullet}$ given by the formula
$\{f_{\eta}\}\otimes g\mapsto \{f_{\eta}\cdot i_{\eta}^*{g}\}$,
where $i^*_{\eta}$ is the natural morphism of sheaves
$i^*_{\eta}:\K_n^X\to (i_{\overline{\eta}})_*K_n(k(\eta))$,
$i_{\overline{\eta}}:\overline{\eta}\hookrightarrow X$. Thus the
complex of sheaves
$\underline{\G}(X)^{\bullet}=\bigoplus\limits_{m\ge
0}\underline{\G}(X,m)^{\bullet}$ is a right module over the sheaf of
associative rings $\K^X$ and the natural morphism
$\bigoplus\limits_{m\ge 0}(\K'_m)^X\to \underline{\G}(X)^{\bullet}$
is a homomorphism of $\K^X$-modules. For any proper morphism $f:X\to
Y$ of irreducible schemes, there is a canonical morphism of
complexes of sheaves
$Rf_*\underline{\G}(X)^{\bullet}[d]=f_*\underline{\G}(X)^{\bullet}[d]\stackrel{f_*}
\longrightarrow\underline{\G}(Y)^{\bullet}$, where
$d=\dim(f)=\dim(X)-\dim(Y)$. The projection formula tells that this
morphism is a homomorphism of $\K^Y$-modules via the homomorphism
$\K^Y\to f_*\K^X$. Therefore general properties of resolutions of
sheaves imply the following fact.

\begin{lemma}\label{product_formula_Massey}
Let $f:X\to Y$ be a proper morphism of schemes.
Let $a_1\in H^{p_1}(X,\underline{\G}(X)^{\bullet})=
H^{p_1}(Y,Rf_*\underline{\G}(X)^{\bullet})$, $b_i\in
H^{p_i}(Y,\K^Y)$, $2\le i\le k$ be classes in $K$-cohomology
groups such that their $k$-th higher product
$m_k(a_1,b_2,\ldots,b_k)$ is well defined, where we consider
$Rf_*\underline{\G}(X)^{\bullet}$ as a module over $\K^Y$ via the homomorphism
$\K^Y\to f_*\K^X\to Rf_*\K^X$.
Then the higher products $m_k(a_1,f^*(b_2),\ldots,f^*(b_k))$ and
$m_k(f_*(a_1),b_2,\ldots,b_k)$ are also well defined and
there is an equality
$$
f_*(m_k(a_1,f^*(b_2),\ldots,f^*(b_k)))=(-1)^{d(p_2+\ldots+p_k)}
m_k(f_*(a_1),b_2,\ldots,b_k).
$$
\end{lemma}

\begin{rmk}
An alternative, more direct, way to show Lemma \ref{product_formula_Massey}
for smooth varieties over an infinite perfect field is to use Theorem \ref{quasiis}
and the adelic projection formula from Proposition \ref{product_formula}.
\end{rmk}

\begin{rmk}\label{rmk-Massey}
Let us recall that Massey higher products for a right DG-module
$M^{\bullet}$ over a DG-ring $A^{\bullet}$ are defined via the
higher differentials in the spectral sequence associated with the
Hochschild bicomplex $(M^{\bullet}\otimes (A^{\bullet})^{\otimes
(p-1)})^q$. More precisely, Massey higher products have the form
$$
m_k:(H^{i_1}(M^{\bullet})\otimes H^{i_2}(A^{\bullet})\otimes\ldots
\otimes H^{i_k}(A^{\bullet}))^{\circ}\to
{}^{\circ}(H^{i_1+\ldots+i_k-k}(M^{\bullet})),
$$
where for a group $G$, the notation $(G)^{\circ}$ means that we take
a certain subgroup in $G$ and ${}^{\circ}(G)$ means that we take a
certain quotient of $G$. In particular, for a sheaf of associative
algebras $\Ac$ on a topological space $X$ and a sheaf $\M$ of right
modules over $\Ac$, there are Massey higher products in cohomology
groups $H^{\bullet}(X,\Ac)$ and $H^{\bullet}(X,\M)$; to define them
one should take multiplicative resolutions for sheaves $\Ac$ and
$\M$ on $X$ (e.g., Godement resolutions), see more details
in~\cite{Den}.
\end{rmk}

If $X$ is a regular scheme of finite type over a field, then
$\K^X_n=(\K'_n)^X$, the complex of sheaves
$\underline{\G}(X,n)^{\bullet}$ is quasiisomorphic to $\K^X_n$, and
$H^n(X,\K^X_n)=CH^n(X)$ for any $n\ge 0$ (see \cite{Q} and also
Proposition \ref{Gersten-resolution}).

By Section \ref{defin-basic-adeles}, the complex
$\A(X,\K^X)^{\bullet}$ is a DG-ring, any morphism of schemes $f:X\to
Y$ defines a DG-homomorphism $\A(Y,\K^Y)^{\bullet}\to
\A(X,\K^X)^{\bullet}$, and the complex $\G(X)^{\bullet}$ is a right
DG-module over the DG-ring $\A(X,\K^X)^{\bullet}$. By Proposition
\ref{product_formula}, for any proper morphism $f:X\to Y$ of
irreducible schemes, the morphism
$\G(X)^{\bullet}[d]\to\G(Y)^{\bullet}$ is a homomorphism of right
DG-modules over $\A(Y,\K^Y)^{\bullet}$ via the homomorphism
$\A(Y,\K^Y)^{\bullet}\to\A(X,\K^X)^{\bullet}$, where $d=\dim(f)$.

\begin{rmk}
It seems that there is no way to define a direct image map on the
adelic complexes $\A(X,\K^X)^{\bullet}$ for proper morphisms of
smooth varieties. This fact can be already seen in the simplest
cases of finite morphisms or a closed embeddings. Nevertheless it is
expected that there exists a {\it complete} version of $K$-adeles
such that the complete adelic complex would have a (non-canonical)
direct image map. Also, completed $K$-adeles should correspond to
the global class field theory of arithmetical schemes, see
\cite{Par78}. Some particular cases were treated in \cite{Osi}.
However the ``complete'' theory is still to be built.
\end{rmk}

\begin{rmk}
It follows from what is said above that for each $p\ge 0$, there is
a canonical map $\alpha_p$ from $H^p(\A(X,\K_p^X)^{\bullet})$ to the
bivariant Chow group $A^p(X\stackrel{id}\longrightarrow X)$ (see
\cite{Ful}). In addition, the natural map $\beta_p:H^p(X,\K_p)\to
A^p(X\stackrel{id}\longrightarrow X)$ factors through $\alpha_p$.
\end{rmk}

\begin{quest}
Does there exist a singular variety $X$ such that the image ${\rm
Im}(\alpha_p:H^p(\A(X,\K_p^X)^{\bullet})\to
A^p(X\stackrel{id}\longrightarrow X))$ is strictly bigger than the
image ${\rm Im}(\beta_p:H^p(X,\K_p)\to
A^p(X\stackrel{id}\longrightarrow X))$, i.e., such that adelic
cocycles define new elements in the bivariant Chow groups?
\end{quest}

Now let us fix an infinite perfect field $k$ and consider $K'_*$ as
an l.a.f. homology theory over $k$, see Example
\ref{homology_examples}, 1). Let $X$ be an irreducible smooth
variety over $k$; then by Proposition \ref{adelic-class} and Claim
\ref{integrality2}, for any algebraic cycle $Y=\sum n_i Y_i$ of
codimension $p$ on $X$, there is a good cocycle $[Y]=[\{1_Y\}]\in
\A(X,\K_{p}^X)^p$, where $\{1_Y\}$ denotes a collection from
$\bigoplus\limits_{\eta\in X^{(p)}} \Z$ that equals $n_i\in\Z$ at
the generic point $\eta_i$ of $Y_i$ for each $i$ and equals $0\in\Z$
at all other schematic points $\eta\in X^{(p)}$. Let us give two
examples for adelic classes of subvarieties.

Let $D$ be a (not necessary reduced or effective) divisor on $X$,
$d=\dim(X)$. For each schematic point $\eta\in X$, consider a local
equation $s_{\eta}\in k(X)^*$ of $D$ in
$X_{\eta}=\Spec(\OO_{X,\eta})$. Evidently, $s_{\xi}/s_{\eta}\in
\OO_{X,\eta}^*$ whenever $\xi\in \overline{\eta}$. Thus we get a
1-cocycle $[D]\in \A(X,\K_1^X)^1$ such that the $(X\eta)$-component
of $[D]$ is $s^{-1}_{\eta}$ for $\eta\ne X$ and the
$(\eta\xi)$-component of $[D]$ is $s_{\eta}/s_{\xi}$ for $\eta\ne
X$, $\xi\in\overline{\eta}$, $\xi\ne\eta$. By construction, the
class of $[D]$ in $H^1(\A(X,\K_1^X)^{\bullet})=CH^1(X)$ coincides
with the class of $D$ in the first Chow group under the map $\nu_X$.
In \cite{Gil}, \cite{Gra}, and Corollary \ref{Gersten_product} it is
proved that the intersection product in Chow groups coincides up to
sign with the natural product in the corresponding $K$-cohomology
groups. Thus we get the following adelic formula for the
intersection index of divisors $D_1,\ldots,D_d$ when $X$ is proper:
$$
(D_1,\ldots,D_d)=
-\sum_{\eta_0\ldots\eta_d}[k(\eta_d):k] \nu_{\eta_0\ldots\eta_d}
\{s^{-1}_{1,\eta_1},s_{2,\eta_1}/s_{2,\eta_2},\ldots,
s_{d,\eta_{d-1}}/s_{d,\eta_{d}}\}= $$ $$
=-\sum_{\eta_0\ldots\eta_d}[k(\eta_d):k] \nu_{\eta_0\ldots\eta_d}
\{s^{-1}_{1,\eta_1},s^{-1}_{2,\eta_2},\ldots,s^{-1}_{d,\eta_d}\},
$$
where the last identity follows from reciprocity law. This formula
was proved by different methods
first for $d=2$ in \cite{Par} and for arbitrary $d$ in \cite{Lom}.
We generalize the explicit computations from \cite{Par} and \cite{Lom}
in the proof of Theorem \ref{intersecting_cycles}.

The next example is the intersection of a 1-cycle $C$ and a divisor
$D$ in the three-dimensional irreducible smooth variety $X$ over
$k$. We describe explicitly a 2-cocycle $[C]$ in the adelic complex
$\A(X,\K_2^X)^{\bullet}$ that represents $C$. Let us choose an
effective reduced divisor $E$ with the following properties: for
each schematic point $\eta\in X$ of codimension at least two in $X$,
there exists an element $t_{\eta}\in K_2(k(X))$ and a subdivisor
$E_{\eta}\subset E$ such that $\sing(t_{\eta})\subset E$ and
$d_{\eta}(\nu_{XE_{\eta}}(t_{\eta}))=C_{\eta}$. Recall that
$\nu_{XE_{\eta}}$ denotes the residue map from $K_2(k(X))$ to the
direct sum of multiplicative groups of fields of rational functions
on all irreducible components of $E_{\eta}$ (see Section
\ref{patching-systems-section}), $C_{\eta}$ is the restriction of
$C$ to $X_{\eta}$, and $d_{\eta}$ is the differential in the local
Gersten resolution on $X_{\eta}$. The existence of such divisor $E$
follows from Proposition \ref{differential} and Remark
\ref{global_patching_systems}. Further, we define the adeles
$f_{012}$ and $f_{013}$ such that
$f_{XE_{\eta}\eta}=t^{-1}_{\eta}\in K_2(k(X))$, where $t_{\eta}$ is
as above. We put all the other components of $f_{012}$ and $f_{013}$
to be any elements from $(\K_2^X)_{\eta}=K_2(\OO_{X,\eta})$. For
each flag $\eta\xi$ of type $(23)$, we have
$$
d_{\eta}(\nu_{XE_{\eta}}(t_{\eta})/\nu_{XE_{\xi}}(t_{\xi})) =0,
$$
hence there exists an element $t_{\eta\xi}\in K_2(k(X))$ such that
$$
d_{\eta}(t_{\eta\xi})=\nu_{XE_{\eta}}(t_{\eta})/
\nu_{XE_{\xi}}(t_{\xi}).
$$
This defines the adele $f_{023}$.
Finally, we see that for each flag $\mu\eta\xi$ of type $(123)$,
the product $f_{\mu\eta\xi}=f_{X\eta\xi}f^{-1}_{X\mu\xi}f_{X\mu\eta}$ belongs to
$(\K_2^X)_{\mu}$ and is also an adele. Thus we have defined the
cocycle $[C]=f\in\A(X,\K_2^X)^2$ such that $[C]$ represents the class of
$C$ in $H^2(\A(X,\K_2^X)^{\bullet})=CH^2(X)$ with respect to the map $\nu_X$.
From this we get the following intersection
formula when $X$ is proper:
$$
(D,C)=\sum_{\mu\eta\xi}[k(\xi):k]\nu_{X\mu\eta\xi}\{s^{-1}_{\mu} ,
f_{\mu\eta\xi}\}=
\sum_{\mu\eta\xi}[k(\xi):k]\nu_{X\mu\eta\xi}\{s^{-1}_{\mu},
t_{\eta\xi}\},
$$
where, as above, $s_{\mu}$ is the local equation of the
divisor $D$ at the point $\mu$ and $\mu\eta\xi$ ranges over all flags of type $(123)$ on $X$.
As in the previous case, the last equality follows from reciprocity law.

Further, let us indicate a link between the adelic complex $\A(X,\K_n^X)^{\bullet}$
and coherent adeles.

\begin{prop}\label{Parshin-Beilinson}
Let $X$ be a smooth variety over a field $k$;
then for any $n\ge 0$, there is a natural morphism of complexes
$$
{\rm dlog}:\A(X,\K_n^X)^{\bullet}\to a(X,\Omega_X^n)^{\bullet},
$$
where $a(X,\Omega_X^n)^{\bullet}$ is the complex of rational coherent
adeles (see \cite{Par76} and \cite{Hub}, Proposition 5.2.1)
and the local component of this morphism for a flag
$\eta_0\ldots\eta_p$ is equal to the natural map
$K_n(\OO_{\eta_0})\to \Omega_{\OO_{\eta_0}/k}^n$ (see \cite{Blo77}).
\end{prop}
\begin{proof}
Let us prove by induction on $p$ that for any natural number $p\ge 0$, any subset $M\subset S(X)_p$,
and any open subset
$U\subset X$, the map ${\rm dlog}:\A(M,\K_n^U)\to a(M,\Omega_{U}^n)$ is
well defined. Since the sheaf $\Omega^n_X$ is locally free, it is $1$-pure and
we may suppose that $X\backslash U=D$ is a divisor.

Suppose that $p=0$. We have $\A(M,\K_n^U)=\prod\limits_{\eta\in M}(\K_n^U)_{\eta}$
and $\A(M,\Omega^n_U)=\lim\limits_{\longrightarrow\atop l\ge 0}
\prod\limits_{\eta\in M}(\Omega^n_X(lD))_{\eta}$ and, by Lemma
\ref{Tate_map}, we get the needed result. For $p>0$, we have
$$
a(M,\Omega_{U}^n)=\prod_{\eta\in P(X)}a({_\eta M},(\Omega_{U}^n)_{\eta})=
\prod_{\eta\in P(X)} \lim_{\longrightarrow\atop V}a({_\eta M},\Omega_{U\cap V}^n),
$$
where for each schematic point $\eta\in P(X)$ the limit is taken
over all open subsets $V\subset X$ containing $\eta$ (for the second equality we
use that the adelic functor commutes with direct limits of
quasicoherent sheaves). Since the same equality holds for the adelic groups
for the sheaf $\K^X_n$, the induction step is proved.
\end{proof}

The author is grateful to C.\,Soul\'e for explaining the proof of
the following lemma.

\begin{lemma}\label{Tate_map}
Rational differential forms from the image of the map
$K_n(k(X))\to \Omega^n_{k(X)/k}$ have pole of order at most one
along each irreducible divisor $D\subset X$.
\end{lemma}
\begin{proof}
Let us recall the construction of the map $K_n(R)\to \Omega^n_{R/\Z}$
and its properties.
There are universal classes
$c_n\in\lim\limits_{\longleftarrow}H^n(GL_m(R),\Omega^n_{R/ \Z})$, where
the limit is taken over $m\ge 0$; they define the canonical maps
$K_n(R)\to H_n(GL(R),\Z)\stackrel{c_n}\longrightarrow
\Omega^n_{R/\Z}$.  The map $c_n$ is trivial on
$H_{n}(GL_{n-1}(R),\Z)$. Moreover, the composition
$R^*\times\ldots\times R^*\to H_1(GL_1(R),\Z)\times\ldots\times
H_1(GL_1(R),\Z)\to H_n(GL_n(R),\Z)\to H_n(GL(R),\Z)
\stackrel{c_n}\longrightarrow \Omega^n_{R/\Z}$ is given by the formula
$(r_1,\ldots,r_n)\mapsto
\frac{dr_1}{r_1}\wedge\ldots\wedge\frac{dr_n}{r_n}$.
Since one may suppose that $\dim X>0$, the field $F=k(X)$ is
infinite. By the results of Suslin, see \cite{Sus}, there is an
isomorphism $H_n(GL_n(F),\Z)\cong H_n(GL(F),\Z)$ and the natural
map constructed above $F^*\times\ldots\times F^*\to
H_n(GL_n(F),\Z)$ induces an isomorphism $K_n^M(F)\cong
H_n(GL_n(F),\Z)/H_n(GL_{n-1}(F),\Z)$. Since for any non-zero
rational functions $f_1,\ldots,f_n\in k(X)^*$, the differential form
$\frac{df_1}{f_1}\wedge\ldots\wedge\frac{df_n}{f_n}$ has pole of
order at most one along each irreducible divisor $D\subset X$, the
lemma is proved.
\end{proof}

\begin{rmk}
It follows from \cite{Sus} that for any field $F$ the natural
composition $K_n^M(F)\to K_n(F)\to
\Omega^n_{F/\Z}$ is given by the formula $\{f_1,\ldots,f_n\}\mapsto
(-1)^n(n-1)!\frac{df_1}{f_1}\wedge\ldots\wedge\frac{df_n}{f_n}$.
\end{rmk}

\begin{rmk}
There is an equality ${\rm dlog}(f\cdot g)=
-\frac{(m+n-1)!}{(m-1)!(n-1)!}{\rm dlog}(f)\cdot{\rm dlog}(g)$,
where $f\in \A(X,\K^X_m)^{\bullet}$, $g\in \A(X,\K^X_n)^{\bullet}$,
and in the right hand side we consider the product in the DG-ring
$\bigoplus\limits_{n\ge 0}a(X,\Omega^n_X)$.
\end{rmk}

\begin{rmk}
Let $Y$ be an algebraic cycle of codimension $p$ on
a smooth variety $X$ over an infinite perfect field $k$.
Then there is an explicit construction for the class of $Y$ in the
rational adelic group $a(X,\Omega_X^p)^p$. Indeed, one
should take the image under the map ${\rm dlog}$
of the explicit (good) class $[Y]$ of $Y$ in $\A(X,\K_p^X)^p$ constructed in
Proposition \ref{adelic-class}.
\end{rmk}

\subsection{Euler characteristic with support for $K$-groups}\label{Euler_charact}

The construction and the results of this section are needed for the proof of Theorem
\ref{intersecting_cycles} given in the next section. These results are
not new; for example, they follow from Waldhausen $K$-theory of
perfect complexes, developed in \cite{Tho} or they may be obtained
by using $R$-spaces constructed in \cite{Blo84}. However the author
did not find a reference for an explicit construction,
that is why this section is written.

By $\Omega S$ denote the loop space of a
pointed space $(S,s_0)$. Let $f:(X,x_0)\to (Y,y_0)$ be a continuous map of pointed
topological spaces. Consider the mapping path fibration
$$
M(f)=\{(x,\varphi)|x\in X,\varphi:I\to Y, \varphi(0)=f(x)\},
$$
where $I$ is the interval $[0,1]$. Recall that the homotopy fiber $F(f)$
is the fiber over $y_0$ of the natural map $M(f)\to Y$,
$(x,\varphi)\mapsto \varphi(1)$. Notice that $F(f)$ and $M(f)$ are
pointed spaces with the point $(x_0,\varphi_0)$, where $\varphi_0$
is the constant map to $y_0$. There is a natural map
$\Omega Y\to F(f)$, defined by $\gamma\mapsto (x_0,\gamma)$.
Moreover, the composition $\Omega X\to \Omega Y\to F(f)$ is
canonically homotopic to the constant map to $(x_0,\varphi_0)\in
F(f)$. Indeed, the homotopy
$$
G:\Omega X\times I\to F(f)
$$
is
given by
$$
(\gamma,t)\mapsto (\gamma(t),\varphi_t),
$$
where $\varphi_t(s)=(f\circ\gamma)(t+s(1-t))$.

Let $\M$ be an exact category, $\E_3$ be the exact category of
exact triples of objects in $\M$. The exact functors
$$
\{0\to M'\to M\to M''\to 0\}\mapsto (M',M'')
$$
and
$$
(M',M'')\mapsto\{0\to M'\to M'\oplus M''\to M''\to 0\}
$$
induce the maps $BQ\E_3\to
BQ\M\times BQ\M$ and $BQ\M\times BQ\M\to BQ\E_3$, respectively.
The well-known result of Quillen says that these two maps of pointed spaces are homotopy
inverse (see \cite{Q}, Theorem 2). Furthermore, let $\M'$ and $\M''$
be two exact subcategories in
$\M$ and let $\E'_3$ be the category of exact triples in $\M$ such
that in the above notations the object $M'$ is in $\M'$ and the object $M''$ is in $\M''$.
Then, analogously, $BQ\E'_3$ is
homotopy equivalent to $BQ\M'\times BQ\M''$.

In what follows we suppose for simplicity that $\M$ is an
abelian category (which is enough for further applications).

\begin{lemma}\label{Thom}
Let $\Cc_n$ be the exact category of length $n$ complexes of
objects in $\M$ and let $\E_n$ be the full subcategory in $\Cc_n$ consisting
of all exact complexes. We put $B^i=\Imm(M^{i-1}\to M^{i})$ for a
complex $M^{\bullet}$. Then the natural maps $BQ\Cc_n\to
BQ\M^{n+1}$, $BQ\E_n\to BQ\M^{n}$ induced by the exact functors
$$
\{0\to M^0\to\ldots\to M^n\to 0\}\mapsto (M^0,\ldots,M^n),
$$
$$
\{0\to M^0\to\ldots\to M^n\to 0\}\mapsto (B^1,\ldots, B^{n}),
$$
respectively, are homotopy equivalences. Moreover, the
following diagram of pointed spaces is commutative up to homotopy:
$$
\begin{array}{ccc}
BQ\E_n&\longrightarrow&BQ\M^{n}\\
\downarrow&&\downarrow\lefteqn{i}\\
BQ\Cc_n&\longrightarrow&BQ\M^{n+1},
\end{array}
$$
where the horizontal maps are as defined above, the left
vertical arrow is the natural inclusion, and $i$ is induced by the exact functor
$$
(B^1,\ldots,B^{n})\mapsto(B^1,B^1\oplus B^2,\ldots, B^{n-1}\oplus
B^n,B^n).
$$
\end{lemma}
\begin{proof}
We follow the proof of Theorem 1.11.7 from \cite{Tho}.
Nevertheless we do not use the language of $K$-theory
spectra of Waldhausen categories.

The proof is by induction on $n\ge 3$. The case $n=3$ is the result of Quillen
mentioned above. For arbitrary $n\ge 4$, consider the natural inclusions of categories $\E_{n-1}\hookrightarrow \E_n$ and
$\M \hookrightarrow \E_n$
given by the functors
$$
\{0 \to M^0\to\ldots M^{n-1}\to 0\}\mapsto\{0\to M^0\to\ldots\to
M^{n-1}\to 0\to 0\}
$$
and
$$
M \mapsto\{0\to 0\ldots\to 0\to M\to M \to 0\},
$$
respectively. The category $\E_n$ is equivalent to the category of
exact triples in $\E_n$
that start with an object from $\E_{n-1}\hookrightarrow \E_{n}$ and end with an object
in $\E_2=\M\hookrightarrow\E_n$. Indeed, the explicit equivalence is given by
the functor
$$
M^{\bullet}\mapsto \{0\to \tau_{\le (n-1)}(M^{\bullet})\to
M^{\bullet}\to \{0\to B^n\to B^n\to 0\}\},
$$
where $\tau_{\le i}$
is the usual truncation functor associated to the canonical
filtration on complexes. Thus, applying the result of Quillen
modified above, we get that $BQ\E_n$ is homotopy equivalent to
$BQ\E_{n-1}\times BQ\E$. Combining the explicit view of this
homotopy and the inductive hypothesis,
we get the desired result for $\E_n$.

The analogous reasoning leads to the needed result for
$\Cc_n$. In this case we should replace the canonical filtration on complexes
by the ``b\^ete'' filtration and consider the inclusion of categories
$\M \hookrightarrow \Cc_n$ given by the functor
$$
M \mapsto\{0\to \ldots\to 0\to M \to 0\}.
$$

Finally, for any exact complex $M^{\bullet}$ from $\E_n$, we have the exact sequences
$$
0\to B^{i}\to M^i\to B^{i+1}\to 0
$$
for all $0< i< n$. It follows from the proof of Corollary 1, $\S$3, \cite{Q} that
this leads to the needed homotopy equivalence in the diagram from the lemma.
\end{proof}

Let $F$ be the homotopy fiber of the natural map $BQ\E_n\to BQ\Cc_n$ and
put $\Kb\N=\Omega BQ{\mathcal N}$ for any exact category ${\mathcal N}$.

\begin{corol}\label{equivalence-1}
The inclusion of the categories $\M\hookrightarrow \Cc_n$ given by the functor
$M\mapsto \{0\to M\to 0\to\ldots\to 0\}$ induces the map $\Kb\M\to\Kb\Cc_n\to F$ such that the
composition is a homotopy equivalence.
\end{corol}
\begin{proof}
Let us compute the induced map on homotopy groups. By Lemma
\ref{Thom}, for each $i\ge 0$, there is a commutative diagram
$$
\begin{array}{ccc}
\pi_i(\Kb\E_n)&\longrightarrow&\pi_{i+1}(BQ\M)^{n}\\
\downarrow&&\downarrow\lefteqn{i_*}\\
\pi_i(\Kb\Cc_n)&\longrightarrow&\pi_{i+1}(BQ\M)^{n+1},\\
\end{array}
$$
where the horizontal arrows are the isomorphisms induced by the maps
described in Lemma \ref{Thom}. Thus there is a canonical isomorphism
$\pi_i(F)\cong \pi_{i+1}(BQ\M)$ given by the alternated sum of
projections $\pi_{i+1}(BQ\M)^{n+1}\to \pi_{i+1}(BQ\M)$. Moreover,
the composition
$\pi_{i+1}(BQ\M)\cong\pi_{i}(\Kb\M)\to\pi_i(\Kb\Cc_n)\to
\pi_i(F)\cong \pi_{i+1}(BQ\M)$ is the identity map. Since by
Milnor's result, $\Kb\M$ has the homotopy type of a CW-complex, we
conclude by the well known theorem of Whitehead.
\end{proof}

By $\M(S)$ denote the abelian category of coherent sheaves on a
scheme $S$. We put $\E_n(S)=\E_n$, $\Cc_n(S)=\Cc_n$, $F(S)=F$,
and $\Kb(S)=\Kb\M$ for $\M=\M(S)$.

\begin{prop}\label{chi}
Let $S$ be a closed subscheme in the
scheme $T$ and let $\Cc_n(T,S)$ be a full subcategory in $\Cc_n(T)$ consisting of
complexes whose cohomology sheaves have support on $S$, i.e.,
complexes whose restriction to $T\backslash S$ is in
$\E_n(T\backslash S)$. Then there exists a well defined up to
homotopy the ``Euler characteristic with support'' map $\chi:\Kb\Cc_n(T,S)\to
\Kb(S)$ with the following properties:
\begin{itemize}
\item[(i)]
the induced homomorphism $\chi_*:K_0(\Cc_n(T,S))\to K'_0(S)$ is
equal to
$$
[\F^{\bullet}]\mapsto \sum_{i=0}^n(-1)^i [H^i(\F^{\bullet})],
$$
where $\F^{\bullet}$ is in $\Cc_n(T,S)$ (here we imply the canonical
isomorphism $K_0'(S)\cong K_0'(\widetilde{S})$, induced by the
closed embeddings $S_{red}\hookrightarrow S$ and
$\widetilde{S}_{red}\hookrightarrow \widetilde{S}$,
 where $\widetilde{S}$ is any closed subscheme in
$T$ such that $S_{red}=\widetilde{S}_{red}$);
\item[(ii)]
$\chi$ commutes with the direct image under closed embeddings; namely consider a closed
subscheme $i:T'\hookrightarrow T$ and put $S'=S\times_{T}T'$. Then the
following diagram of pointed spaces is commutative up to homotopy:
$$
\begin{array}{ccc}
\Kb\Cc_n(T,S)&\stackrel{\chi}\longrightarrow& \Kb(S)\\
\uparrow\lefteqn{i_*}&&\uparrow\lefteqn{i_*}\\
\Kb\Cc_n(T',S')&\stackrel{\chi}\longrightarrow& \Kb(S');
\end{array}
$$
\item[(iii)]
$\chi$ commutes with the restriction to open subsets; namely consider
an open subset $j:U\hookrightarrow T$ and put $V=S\times_T U$. Then the
following diagram of pointed spaces is commutative up to homotopy:
$$
\begin{array}{ccc}
\Kb\Cc_n(T,S)&\stackrel{\chi}\longrightarrow& \Kb(S)\\
\downarrow\lefteqn{j^*}&&\downarrow\lefteqn{j^*}\\
\Kb\Cc_n(U,V)&\stackrel{\chi}\longrightarrow& \Kb(V).
\end{array}
$$
\end{itemize}
\end{prop}

\begin{proof}
The natural map $\Kb\Cc_n(T,S)\to F(T\backslash S)$, defined by
the diagram
$$
\begin{array}{ccc}
\Kb\Cc_n(T,S)&\longrightarrow&\Kb\E_n(T\backslash S)\\
\downarrow&&\downarrow\\
\Kb\Cc_n(T)&\longrightarrow&\Kb\Cc_n(T\backslash S)\\
\downarrow&&\downarrow\\ F(T)&\longrightarrow& F(T\backslash S),
\end{array}
$$
is canonically homotopy trivial. Hence there is a well defined
map
$$
\Kb\Cc_n(T,S)\to F\left\{F(T)\to F(T\backslash S)\right\}.
$$
On the other hand, by Corollary \ref{equivalence-1} and by Quillen's localization
lemma (see \cite{Q}), the diagram
$$
\begin{array}{ccc}
\Kb(T)&\longrightarrow&\Kb(T\backslash S)\\
\downarrow&&\downarrow\\ F(T)&\longrightarrow&F(T\backslash S)
\end{array}
$$
induces a homotopy equivalence $\Kb(S)\to F\left\{F(T)\to
F(T\backslash S)\right\}$. This defines the map $\chi:\Kb\Cc_n(T,S)\to\Kb(S)$
uniquely up to homotopy.

Now we prove $(i)$, i.e., we compute explicitly $\chi_*$ on
$\pi_0$-groups. Consider a point $[\F^{\bullet}]$ in
$\Kb\Cc_n(T,S)$ corresponding to a loop in $BQ\Cc_n(T,S)$ defined
by a complex $\F^{\bullet}$ from $\Cc_n(T,S)$ in the standard way.
There exists a homotopy inside
$BQ\Cc_n(T,S)$ between the loop $[\F^{\bullet}]$ and the sum of
loops
$$
[\tau_{\le (n-1)}\F^{\bullet}]+ [\{0\to\ldots\to B^n\to
B^n\to 0\}]+ [H^n(\F^{\bullet})[-n]].
$$
In addition,
$[H^n(\F^{\bullet})[-n]]$ is homotopic inside $BQ\Cc_n(T,S)$ to
the sum
$$
(-1)^n[H^n(\F^{\bullet})]+\sum_{j=0}^{n-1}(-1)^{j}
[\{0\to H^n(\F^{\bullet})\to H^n(\F^{\bullet})\to 0\}[-j]],
$$
where the short complexes have support in degrees 0 and 1.
Continuing, we show by induction that the initial loop may be
homotoped inside $BQ\Cc_n(T,S)$ to the sum of $\sum(-1)^i
[H^i(\F^{\bullet})]$ and some loops in $BQ\E_n(T)$. The
classes of points in $\Kb\Cc_n(T,S)$, corresponding to loops in $BQ\E_n(T)$,
evidently have zero image under $\chi_*$ on $\pi_0$-groups, and we
are done.

For the proof of $(ii)$ and $(iii)$ one uses that the natural maps
$$
T'\backslash
S'=(T\backslash S)\times_T T'\hookrightarrow T\backslash S
$$
and
$$
U\backslash V=(T\backslash S)\times_T
U\hookrightarrow T\backslash S
$$
are closed embedding and open embedding, respectively. Also, one
uses the fact that if in the commutative diagram of pointed
spaces
$$
\begin{array}{ccc}
X&\longrightarrow&Y\\ \downarrow&&\downarrow\\ Z&\longrightarrow&W
\end{array}
$$
the vertical arrows are homotopy equivalences, then the diagram remains
commutative up to homotopy after
we take the homotopy inverse to the vertical homotopy equivalences.
\end{proof}

\begin{rmk}
A similar way to define the Euler characteristic with support on $K$-groups
$K_*(\Cc_n(T,S))\to K_*(S)$ is to use the $R$-space construction from
\cite{Blo84}, Section 1. Recall that for any exact category $\M$,
the space $R\M$ has the same homotopy type as $BQ\M$ and the $H$-space $R\M$
has a canonical inverse map, not just an inverse up to homotopy.
Thus there is the Euler characteristic map
$R\Cc_n(T,S)\to R\M(T)$ such that its composition with the natural map
$R\M(T)\to R\M(T\backslash S)$ is canonically homotopy trivial.
This gives a well defined map $R\Cc_n(T,S)\to R\M(S)$ up to homotopy.
\end{rmk}

Now consider a complex $\Pc^{\bullet}$ from $\Cc_n(T,S)$ such that $\Pc^{\bullet}$
consists of flat sheaves on $T$. Let the map $*\cdot \Pc^{\bullet}:\Kb(T)\to
\Kb(S)$ be the composition of the map induced by the exact functor
$*\otimes_{\OO_T}\Pc^{\bullet}:\M(T)\to \Cc_n(T,S),\F\mapsto
\F\otimes_{\OO_T}\Pc^{\bullet}$ and the map $\chi:\Kb\Cc_n(T,S)\to \Kb(S)$
from Proposition \ref{chi}. The map $*\cdot \Pc^{\bullet}$ is well defined up to
homotopy. A standard argument shows that the maps $*\cdot \Pc^{\bullet}$ and
$*\cdot (\Pc')^{\bullet}$ are homotopy equivalent for quasiisomorphic complexes
$\Pc^{\bullet}$ and $(\Pc')^{\bullet}$.
By Proposition \ref{chi}$(i)$, for any class $[\F]$ from
$K_0'(T)$ we have $[\F]\cdot\Pc^{\bullet}=\sum(-1)^i
H^i(\F\otimes_{\OO_T}\Pc^{\bullet})\in K_0'(S)$.

\begin{prop}
Let $\Pc^{\bullet}$ be a finite flat resolution of $\OO_S$ on
$T$ (here the complex $\Pc^{\bullet}$ has
support in non-positive terms) and suppose that $T$ admits an ample line bundle
(for example, $T$ is quasi-projective). Then the map $*\cdot \Pc^{\bullet}$ is
homotopic to the map $f^*$, where $f:S\hookrightarrow T$ is the
closed embedding (see \cite{Q}, $\S$7, 2.5).
\end{prop}

\begin{proof}
By Quillen resolution theorem (see \cite{Q}, $\S$4, Corollary 1),
the space $BQ\M(T)$ is homotopy equivalent to its subspace $BQ\F(T)$,
where $\F(T)$ is the exact
category of coherent sheaves on $T$ that are Tor-independent with $\OO_S$. For
$BQ\F(T)$ we have the exact sequence of functors from $\F(T)$ to $\Cc_n(T,S)$
$$
0\to\tau'_{\le
0}(*\otimes_{\OO_T}\Pc^{\bullet})\to *\otimes_{\OO_T}\Pc^{\bullet}\to
*\otimes_{\OO_T}\OO_S\to 0
$$
where we put
$$
\tau'_{\le 0}(A^{\bullet})=\{\ldots\to A^i\to\ldots\to
A^{-1}\to B^0\to 0\},
$$
for any complex $A^{\bullet}$. It follows from the proof of
Corollary 1, $\S$3, \cite{Q} that the map $\Kb(*\otimes_{\OO_T}\Pc^{\bullet})$
is homotopic to the map
$\Kb(\tau'_{\le 0}(*\otimes_{\OO_T}\Pc^{\bullet}))+\Kb(*\otimes_{\OO_T}\OO_S)$,
where the sum is taken with respect to the natural $H$-structures on
$\Kb$-spaces of exact categories. Moreover, the image of the map
$\Kb(\tau'_{\le 0}(*\otimes_{\OO_T}\Pc^{\bullet}))$ is in the space
$\Kb(\E_n(T))\subset \Kb(\Cc_n(T,S))$ and the map $\Kb(*\otimes_{\OO_T}\OO_S)$ equals $f^*$
on $\Kb(\F(T))$. This concludes the proof.
\end{proof}

\begin{rmk}\label{K-groups_multiplication}
For the elements from $K_*(T)$, the map $*\cdot\Pc^{\bullet}$ is
equal to the composition of the usual restriction to $S$ with
multiplication by $\chi_*([\Pc^{\bullet}])\in K'_0(S)$.
\end{rmk}

\begin{prop}\label{residue_restrict}
Let $i:T'\hookrightarrow T$ be a closed embedding and
put $S'=S\times_T T'$, $j:U=T\backslash T'\hookrightarrow T$,
$V=S\times_T U$. Consider arbitrary elements $x\in K_m(S)$, $y\in
K_n'(U)$, where $m,n\ge 0$, $m+n\ge 1$. Then we have
$$
\nu(x\cdot (y\cdot \Pc^{\bullet}))=(-1)^{m}x\cdot (\nu(y)\cdot
i^*\Pc^{\bullet})\in K'_{m+n-1}(S'),
$$
where $\nu:K'_*(U)\to
K'_{*-1}(T')$ denotes the usual boundary map (the same for $V$
and $S'$).
\end{prop}
\begin{proof}
By Proposition \ref{chi}$(ii)$,$(iii)$, both squares in the following diagram of spaces are commutative up
to homotopy:
$$
\begin{array}{ccc}
\Kb(\Pc(S))\wedge \Kb(T')&\stackrel{*\cdot \Pc_{\bullet}\otimes
*}\longrightarrow& \Kb(S')\\
\downarrow\lefteqn{i_*\times {\rm id}}&&\downarrow\lefteqn{i_*}\\
\Kb(\Pc(S))\wedge \Kb(T)
&\stackrel{*\cdot \Pc_{\bullet}\otimes
*}\longrightarrow& \Kb(S)\\ \downarrow\lefteqn{j^*\times {\rm
id}}&&\downarrow\lefteqn{j^*}\\
\Kb(\Pc(S))\wedge\Kb(U)
&\stackrel{*\cdot \Pc_{\bullet}\otimes
*}\longrightarrow& \Kb(V),
\end{array}
$$
where $\Pc(S)$ denotes the exact category of locally free
coherent sheaves on $S$ and ``$\wedge$'' denotes the wedge
product of pointed topological spaces. Thus rest is a direct application
of Theorem 2.5 from \cite{Gra}.
\end{proof}

Let us mention the following simple fact.

\begin{lemma}\label{reduction}
Let $S\hookrightarrow T$ be a closed embedding, $S_{red}$ be the
reduced scheme, $j:S_{red}\to S$ be a natural embedding. Then
$j_*\circ\nu_{red}=\nu$, where $\nu:K'_*(T\backslash S)\to
K'_{*-1}(S)$, $\nu_{red}:K'_*(T\backslash S)\to
K'_{*-1}(S_{red})$ are the boundary maps.
\end{lemma}
\begin{proof}
This follows immediately from the commutativity of the diagram of
CW-spaces
$$
\begin{array}{ccc}
BQ\M(S_{red})&\stackrel{j_*}\longrightarrow&BQ\M(S)\\
\downarrow&&\downarrow\\
BQ\M(T)&\stackrel{=}\longrightarrow&BQ\M(T)\\
\downarrow&&\downarrow\\ BQ\M(T\backslash
S_{red})&\stackrel{=}\longrightarrow&BQ\M(T\backslash S)
\end{array}
$$
after we pass to the long homotopy sequences, associated to the
vertical sequences, which are fibrations up to homotopy.
\end{proof}

Lemma \ref{reduction} implies the following statement.

\begin{corol}\label{reduced_residue_restrict}
Proposition \ref{residue_restrict} remains true after we change
the schematic intersection $S'=S\times_T T'$ by its reduced part $S'_{red}$.
\end{corol}

\begin{examp}
Let $k$ be a field, $T$ be the local scheme $(\Ab_k^2)_{(0,0)}$ with coordinates
$(x,y)$, $T'=\{xy=0\}$, $S=\{x+y=0\}$,
$\Pc^{\bullet}=\{\OO_T\stackrel{x+y}\longrightarrow \OO_T\}$. Let
$f(x),g(y)$ be rational functions on the corresponding
irreducible components of $T'$ such that $f(x)$ and $g(y)$
have the opposite valuations at the origin.
These functions naturally define an element $\alpha\in K'_1(T')$.
We have $\alpha\cdot
i^*\Pc^{\bullet}=a/b\in K_1(k)$, where $a$ and $b$ are the main
parts of $f(x)$ and $g(y)$ in $x$ and $y$, respectively.
\end{examp}

\subsection{Formula for product in $K$-cohomology}\label{product_cocycles_section}

Let $X$ be an irreducible smooth quasiprojective variety over an infinite perfect field,
$Y$, $Z$ be two equidimensional subvarieties in $X$ of codimensions $p$ and $q$,
respectively. Consider two cocycles
$\{f_y\}\in\bigoplus\limits_{y\in Y^{(0)}}K_m(k(y))$ and
$\{g_z\}\in\bigoplus\limits_{z\in Z^{(0)}}K_n(k(z))$
in the Gersten complexes $\G(X,p+m)^p$ and $\G(X,q+n)^q$, respectively.

Suppose that the subvarieties $Y$ and $Z$ intersect properly.
In addition, suppose that for any
irreducible component $w$ of the intersection $W=Y\cap Z$, the
collections $\{f_y\}_w$ and $\{g_z\}_w$ are represented by some
elements $\alpha_w$ and $\beta_w$ from the groups $K_m(Y_w)$ and $K'_n(Z_w)$,
respectively (recall that the index $w$ means the restriction of
a collection to $X_w=\Spec(\OO_{X,w})$).

By Remark~\ref{global_patching_systems} and
Remark~\ref{rmk-irreducible-patchingsyst}, there are patching
systems $\{Y^{1,2}_r\}$, $1\le r\le p-1$ and $\{Z^{1,2}_s\}$, $1\le
s\le q-1$ on $X$ for subvarieties $Y$ and $Z$, respectively, with
the freedom degree at least zero such that the following conditions
are satisfied:
\begin{itemize}
\item[(i)]
for any $s,1\le s\le q-1$, no irreducible component in $Y\cap Z^1_s$
is contained in $Z^2_s$;
\item[(ii)]
each irreducible component in $Y^{1,2}_{p-1}$ contains some
irreducible component in $Y$, each irreducible component in
$Y^{1,2}_r$, $2\le r\le p-2$, contains some irreducible component in
$Y^1_{r+1}\cup Y^2_{r+1}$, and the analogous is true for the
patching system $Z^{1,2}_{s}$, $1\le s\le q-1$; in particular, the
subvarieties $Y^{1,2}_r$ and $Y$ meet the subvarieties $Z^{1,2}_s$
and $Z$ properly.
\end{itemize}

Let $f\in \A(X,\K^X_{p+m})^p$ and $g\in \A(X,\K^X_{q+n})^q$ be good
cocycles for the collections $\{f_y\}$ and $\{g_z\}$ with respect to
the patching systems $\{Y^{1,2}_r\}$, $1\le r\le p-1$ and
$\{Z^{1,2}_s\}$, $1\le s\le q-1$, respectively, such that
$\{Y^{1,2}_r\}$ and $\{Z^{1,2}_s\}$ satisfy two conditions from
above (see Section \ref{explicit-classes}).

\begin{theor}\label{intersecting_cycles}
Let $\Pc^{\bullet}\to \OO_Y$ be a finite flat resolution of
$\OO_Y$ on $X$. Under the above assumptions we have:
$$
\nu_{X}(f\cdot g)=(-1)^{(p+m)q}\{\overline{\alpha}_w\cdot (\beta_w\cdot
i^*_Z\Pc^{\bullet})\}\in \mbox{$\bigoplus\limits_{w\in W^{(0)}}K_{m+n}(k(w))$},
$$
where $i_Z:Z\hookrightarrow X$ is the closed embedding (considered
locally around $w$ for each summand) and the bar over $\alpha_w$
denotes the image under the natural homomorphisms $K_m(Y_w)\to
K_m(k(w))$ for each point $w$.
\end{theor}

\begin{proof}
Let $\eta_0\ldots\eta_{p+q}$ be a flag of type $(0\ldots p+q)$ on $X$. By
condition $(i)$ from Proposition \ref{adelic-class}, we have
$f_{\eta_0\ldots\eta_p}\cdot g_{\eta_{p}\ldots\eta_{p+q}}=0$ unless
$\eta_r$ is the generic point of an irreducible component of $Y^1_r$ for all
$r,1\le r\le p-1$, $\eta_{p}$ is the generic point of an irreducible component of $Y$,
$\eta_{p+s}$ is the generic point of an irreducible component of the intersection
$Y\cap Z^1_s$ for all
$s,1\le s\le q-1$, and $\eta_{p+q}$ is the generic point of an irreducible component
of the intersection $Y\cap Z$.

Suppose that the flag $\eta_0\ldots\eta_{p+q}$ enjoys this property.
Combining the assumption on the patching systems $\{Y_r^{1,2}\}$,
$1\le r\le p-1$ and $\{Z^{1,2}_{s}\}$, $1\le s\le q-1$, condition
$(ii)$ from Proposition \ref{adelic-class}, and
Claim~\ref{integrality2}, we see that
$$
f_{\eta_0\ldots\eta_p}=\widetilde{f}_{\eta_p}\in
K_{p+m}(X_{\eta_p}\backslash (Y^1_1\cup Y^2_1)),\,
g_{\eta_p\ldots\eta_{p+q}}=\widetilde{g}_{\eta_{p+q}}\in
K_{q+n}(X_{\eta_{p+q}}\backslash (Z_1^1\cup Z_1^2)),
$$
where
$$
d_{\eta_p}\nu_{XY^1_1\ldots Y^1_{p-1}}(\widetilde{f}_{\eta_p})=
(-1)^{\frac{p(p+1)}{2}}f_{\eta_p}= \{f_y\}_{\eta_p},
$$
$$
d_{\eta_{p+q}}\nu_{XZ^1_1\ldots
Z^1_{q-1}}(\widetilde{g}_{\eta_{p+q}})=
(-1)^{\frac{q(q+1)}{2}}\{g_z\}_{\eta_{p+q}}.
$$
Henceforth,
$\nu_{0\ldots p}(f\cdot g)_{\eta_p\ldots\eta_{p+q}}=
\nu_{XY^1_1\ldots
Y^1_{p-1}\eta_p}(\widetilde{f}_{\eta_p}\cdot\widetilde{g}_{\eta_{p+q}})$.


We claim that the residue $\nu_{XY^1_1\ldots
Y^1_{p-1}\eta_p}(\widetilde{f}_{\eta_p}\cdot\widetilde{g}_{\eta_{p+q}})$
is equal to the product
\mbox{$(-1)^{\frac{p(p+1)}{2}}f_{\eta_p}\cdot
i_{\eta_p}^*\widetilde{g}_{\eta_{p+q}}\in K_{m+q+n}(k(\eta_p))$},
where $i_{\eta_p}:\Spec (k(\eta_p))\to X$ is the natural morphism
(notice that $\eta_p$ does not belong to $Z^1_1\cup Z^2_1$). This
can be shown using Proposition \ref{residue_restrict} first with
$S=T=X_{\eta_p}$, $\Pc^{\bullet}=\OO_S$, $T'=(Y_1^1\cup
Y^2_1)_{\eta_p}$ and then, inductively, with $S=T=(Y^1_s)_{\eta_p}$,
$\Pc^{\bullet}=\OO_S$, $T'=(Y_{s+1}^1\cup Y^2_{s+1})_{\eta_p}$ for
$1\le s\le p-1$ (more precisely, for the induction step we use that
the cocycles $f$ and $g$ satisfy the conditions from
Claim~\ref{integrality2}).

In addition, the product
$f_{\eta_p}\cdot i_{\eta_p}^*\widetilde{g}_{\eta_{p+q}}\in K_{m+q+n}(k(\eta_p))$
equals the restriction to $\Spec (k(\eta_p))$ of the product
$$
\alpha_{\eta_{p+q}}\cdot i_{Y}^*\widetilde{g}_{\eta_{p+q}}\in
K_{m+q+n}(Y_{\eta_{p+q}}\backslash (Z_1^1\cup Z_1^2)),
$$
where $i_Y:Y\hookrightarrow X$ is the closed embedding. Consequently,
we have $\nu_{p+q}(f\cdot g)_{w}=(-1)^{\frac{p(p+1)}{2}}
\nu_{Y,Y\cap Z^1_1,\ldots,Y\cap Z^1_{q-1},w}
(\alpha_{w}\cdot i_Y^*\widetilde{g}_w)$ if $w$ is
the generic point of an irreducible
component of the intersection $W=Y\cap Z$. Otherwise,
$\nu_{p+q}(f\cdot g)_{w}=0$.

Let $w$ be the generic point of an irreducible component of $W$. In
what follows we consider all schemes on $X$ locally around the
schematic point $w$, i.e., we consider their restrictions to
$X_w=\Spec(\OO_{X,w})$, though we denote them by the same letter.
Put $Z_0^1=X$, $Z_q^1=Z$ and by $i_s:Z^1_s\hookrightarrow X$ denote
the natural closed embedding for each $s,0\le s\le q$. By
Proposition \ref{chi} and Corollary \ref{reduced_residue_restrict},
the following diagram commutes up to sign $(-1)^m$ for all $s,0\le
s\le q-1$:
$$
\begin{array}{ccc}
K_{*+m}'(Y\cap(Z^1_s\backslash (Z^1_{s+1}\cup
Z^2_{s+1})))&\longleftarrow& K'_*(Z^1_s\backslash (Z^1_{s+1}\cup
Z^2_{s+1}))\\
\downarrow&&\downarrow\\
K_{*+m-1}'(Y\cap
(Z^1_{s+1}\backslash Z^2_{s+1}))&\longleftarrow&
K_{*-1}'(Z^1_{s+1}\backslash Z^2_{s+1})),
\end{array}
$$
where the vertical arrows are the compositions of the boundary maps with
the restrictions to open subsets
and the horizontal arrows are the compositions
of the map $*\cdot i^*_s\Pc^{\bullet}$ (respectively, $*\cdot
i^*_{s+1}\Pc^{\bullet}$) with the multiplication on the right by
the restriction of $\alpha_{\eta_{p+q}}\in K_m(Y)$ to the
corresponding closed subsets in $Y$. Therefore, by Claim \ref{integrality2} and
Remark \ref{K-groups_multiplication}, we get
$$
\nu_{Y,Y\cap Z^1_1,\ldots,Y\cap Z^1_{q-1},w} (\alpha_{w}\cdot
i_Y^*\widetilde{g}_{w})= \nu_{Y,Y\cap Z^1_1,\ldots,Y\cap Z^1_{q-1},w}
(\alpha_{w}\cdot (\widetilde{g}_w\cdot
\Pc^{\bullet}))=
$$
$$
=\overline{\alpha}_{w}\cdot((-1)^{mq}\nu_{X Z^1_1\ldots Z^1_{q-1}w}
(\widetilde{g}_{w})\cdot i_Z^*\Pc^{\bullet})=
(-1)^{mq+\frac{q(q+1)}{2}}\overline{\alpha}_{w}\cdot (\beta_{w}\cdot i^*_Z\Pc^{\bullet}).
$$
Combining this with the equality $\nu_X=(-1)^{\frac{(p+q)(p+q+1)}{2}}\nu_{p+q}$,
we conclude the proof of the theorem.
\end{proof}

Let $\{f_y\}\in \G(X,p+m)^p$ and $\{g_z\}\in \G(X,q+n)^q$ be two Gersten cocycles as above
with one additional property: for any irreducible component
$w$ of the intersection $W=Y\cap Z$, the collection $\{g_z\}_w$ is represented by
an element $\beta_w$ from the group $K_n(Z_w)$.

\begin{corol}\label{Gersten_product}
Under the above assumptions, the product of the
classes of $\{g_y\}$ and $\{h_z\}$ in $K$-cohomology groups
is represented by the Gersten cocycle
$$
(-1)^{(p+m)q}\{(Y,Z;w)\overline{\alpha}_w\cdot\overline{\beta}_w\}\in
\mbox{$\bigoplus\limits_{w\in W^{(0)}}K_{m+n}(k(w))$},
$$
where $(Y,Z;w)$ is the
local intersection index of $Y$ and $Z$ at the component $w$.
In particular, the intersection product
in Chow groups coincides up to sign with the natural product in the corresponding
$K$-cohomology groups (the last assertion had been proved by different methods in
\cite{Gil} and \cite{Gra}).
\end{corol}
\begin{proof}
The composition of the morphisms of complexes
$\K_n(\OO_X)\to
\underline{\A}(X,\K_n)^{\bullet}\stackrel{\nu_X}\longrightarrow
\underline{\G}(X,n)^{\bullet}$ is identity on the (hyper)cohomology groups.
Therefore the product of the classes of Gersten cocycles in $K$-cohomology groups is
represented by the image under the map $\nu_X$ of the adelic product
of the corresponding adelic cocycles.

On the other hand, we have $(Y,Z;w)=\displaystyle\sum_{i\ge 0}
(-1)^il(\Tor_i^{\OO_{X,w}}(\OO_{Y,w},\OO_{Z,w}))$,
where $l(\cdot)$ is the length of an $\OO_{X,w}$-module, i.e., the length of a
filtration whose adjoint quotients are one-dimensional vector spaces over
the field $k(w)$. Thus the lemma follows directly from Theorem
\ref{intersecting_cycles} and Remark \ref{K-groups_multiplication}.
\end{proof}

\begin{rmk}\label{regular-condition}
If each generic point $w$ of an irreducible component of the intersection $W=Y\cap Z$
is regular on $Y$ and $Z$, then the conditions of Corollary \ref{Gersten_product}
are satisfied.
\end{rmk}

\subsection{Massey triple product and the Weil pairing}\label{section-triple-formula}

In this section we apply the adelic resolution to computation of
some Massey triple product. Let $X$ be a smooth variety of dimension
$d$ over an infinite perfect field $k$. Consider elements $\alpha\in
CH^p(X)_l=H^p(X,\K_p^X)_l$, $\beta\in CH^q(X)_l=H^q(X,\K_q^X)_l$,
and $l\in H^0(X,\K^X_0)=\Z$ such that $p+q=d+1$. The triple
$(\alpha,l,\beta)$ satisfies $\alpha\cdot l=l\cdot \beta=0$ in
$K$-cohomology groups, hence there is a triple product
$$
m_3(\alpha,l,\beta)\in H^d(X,\K_{d+1})/(\alpha\cdot
H^{q-1}(X,\K_q^X)+ H^{p-1}(X,\K_p^X)\cdot \beta).
$$
We compute this product explicitly. Let us represent the classes
$\alpha\in CH^p(X)_l$ and $\beta\in CH^q(X)_l$ by cycles $Y=\sum_i
m_i\cdot Y_i$ and $Z=\sum_j n_j\cdot Z_j$, respectively, where $Y_i$
and $Z_j$ are irreducible subvarieties in $X$ of codimensions $p$
and $q$, respectively, and $m_i,n_j\in\Z$. Since $p+q=d+1$, it can
be assumed that the supports $|Y|=\cup_iY_i$ and $|Z|=\cup_j Z_j$ do
not intersect. Let $\{f_{\widetilde{y}}\}\in \G(X,p)^{p-1}$ and
$\{g_{\widetilde{z}}\}\in \G(X,q)^{q-1}$ be two collections such
that $d\{f_{\widetilde{y}}\}=lY$, $d\{g_{\widetilde{z}}\}=lZ$ and
let $\widetilde{Y}\subset X$ and $\widetilde{Z}\subset X$ be the
supports of $\{f_{\widetilde{y}}\}$ and $\{g_{\widetilde{z}}\}$,
respectively. It follows the moving lemma for higher Chow groups
(see \cite{Blo86} and also \cite{Lev}) that for $X$ either affine or
projective, one can choose $\{f_{\widetilde{y}}\}$ and
$\{g_{\widetilde{z}}\}$ such that the intersections
$\widetilde{Y}\cap |Z|$ and $|Y|\cap\widetilde{Z}$ are proper and
the rational functions $f_{\widetilde{y}}$ and $g_{\widetilde{z}}$
are regular at all points from the intersections $\widetilde{Y}\cap
|Z|$ and $|Y|\cap \widetilde{Z}$, respectively. For each point $x\in
\widetilde{Y}\cap |Z|$, we put $f(x)=\prod_{\widetilde{y}}
f_{\widetilde{y}}^{(\widetilde{y},Z;x)}(x)$. Similarly, we define
$g(x)$ for each point $x\in |Y|\cap\widetilde{Z}$.

\begin{prop}\label{Massey_comp}
Under the above assumptions, the triple product
$m_3(\alpha,l,\beta)$ is represented by the Gersten cocycle
$$
(-1)^{pq}(\sum_{x\in \widetilde{Y}\cap |Z|}f(x)\cdot x+\sum_{x\in
|Y|\cap \widetilde{Z}}g^{-1}(x)\cdot x)\in \G(X,d+1)^d.
$$
\end{prop}
\begin{proof}
By Remark~\ref{global_patching_systems} and
Remark~\ref{rmk-irreducible-patchingsyst}, there are patching
systems $\{Y^{1,2}_r\}$, $\{\widetilde{Y}^{1,2}_r\}$,
$\{Z^{1,2}_s\}$, and $\{\widetilde{Z}^{1,2}_s\}$ on $X$ for
subvarieties $|Y|$, $\widetilde{Y}$, $|Z|$, and $\widetilde{Z}$,
respectively, such that the pairs of patching systems
$(\{Y^{1,2}_r\},\{\widetilde{Z}^{1,2}_s\})$ and
$(\{\widetilde{Y}^{1,2}_r\},\{Z^{1,2}_s\})$ satisfy conditions of
Theorem~\ref{intersecting_cycles}. Let $[Y]\in \A(X,\K_p)^p$ and
$[Z]\in\A(X,q)^{q}$ be good cocycles for $Y$ and $Z$ with respect to
the patching systems $\{Y^{1,2}_r\}$ and $\{Z^{1,2}_s\}$ on $X$,
respectively. By Lemma~\ref{suitable_cycles}, there are adeles
$f\in\A(X,\K^X_p)^{p-1}$ and $g\in \A(X,\K^X_q)^{q-1}$ such that
$df=l[Y]$, $dg=l[Z]$, and the restrictions $f_{U}$ and $g_{V}$ are
good cocycles for $\{f_{\widetilde{y}}\}_U$ and
$\{g_{\widetilde{z}}\}_V$ with respect to the patching systems
$\{\widetilde{Y}^{1,2}_r\}_U$ and $\{\widetilde{Z}^{1,2}_s\}_V$
respectively, where $U=X\backslash |Y|$ and $V=X\backslash |Z|$. By
definition, $m_3(\alpha,l,\beta)$ is represented by the Gersten
cocycle $\nu_X(f\cdot[Z]-(-1)^p[Y]\cdot g)$. Since
$\nu_X(f\cdot[Z])=\nu_U((f\cdot[Z])_U)$ and $\nu_X([Y]\cdot
g)=\nu_V(([Y]\cdot g)_V)$, we conclude by
Theorem~\ref{intersecting_cycles} and
Corollary~\ref{Gersten_product}.
\end{proof}

Let $X$ be a smooth projective variety of dimension $d$ over a field
$k$. Evidently, for any degree zero element $\alpha\in CH^d(X)$, we
have $\pi_*(\alpha\cdot
H^0(X,\K_1^X))=\pi_*(\alpha\cdot\OO^*(X))=1$, where
$\pi_*:H^d(X,\K_{d+1}^X)\to k^*$ is the direct image map associated
with the structure morphism $\pi: X\to {\rm Spec}(k)$.

\begin{prop}\label{Bloch-vanishing-direct}
The subgroup $H^{d-1}(X,\K_d^X)\cdot \Pic^0(X)\subset
H^d(X,\K^X_{d+1})$ is in the kernel of the direct image map
$\pi_*:H^{d}(X,\K_{d+1}^X)\to k^*$.
\end{prop}
\begin{proof}
After the base change, one may assume that the ground field $k$ is
algebraically closed. Consider a \mbox{$K_1$-chain} $\{f_{j}\}\in
\G(X,d)^{d-1}$ with $\sum_j\div(f_j)=0$ and a group homomorphism
$\Pic^0(X)(k)\to k^*$ given by $\beta\mapsto
\pi_*(\{f_j\}\cdot\beta)$. We claim that this homomorphism is
induced by a regular morphism from the Picard variety $\Pic^0(X)$ to
the algebraic group ${\mathbb G}_m$ and therefore is identically
equal to $1\in k^*$.

For each $j$, by $C'_j$ denote the complement to the divisor of the
function $f_j$ on $C_j$. For any closed point $\gamma\in \Pic^0(X)$,
there exists a rational section $s$ of the Poincar\'e line bundle on
the product $X\times \Pic^0(X)$ such that the restriction
$D_{\gamma}$ to $X\times \{\gamma\}$ of the divisor $D=\sum m_i D_i$
of the section $s$ meets the curve $C\times \{\gamma\}$ properly and
this intersection is contained in $C'_j\times\{\gamma\}$. Clearly,
this condition holds for all closed points $\beta$ from a
sufficiently small open neighborhood $U$ of the point $\gamma$ in
$\Pic^0(X)$. Let $a_{ij}$ be the degree of the natural finite map
$D_i\cap(C_j\times U)\to U$. We get a regular morphism
$$
\widetilde{f}:U\to
\prod_{i}\Sym^{a_{ij}}(C'_j)\stackrel{f}\longrightarrow {\mathbb
G}_m,
$$
where $f=\prod_{i}(\Sym^{a_{ij}}(f_j))^{m_i}$. Moreover, for any
point $\beta\in U$, there is an equality
$\widetilde{f}(\beta)=\prod\limits_{i,j} \prod\limits_{x\in
(D_i)_{\beta}\cap C_j}f_j^{m_i a_{ij}(x)}(x)$, where $a_{ij}(x)$ is
the intersection index of the divisor $(D_i)_{\beta}$ with $C_j$ at
a point $x$. Therefore by Corollary~\ref{Gersten_product},
$\pi_*(\{f_j\}\cdot\beta)=\widetilde{f}(\beta)^{(-1)^{d}}$. This
proves the needed statement.
\end{proof}

By Proposition~\ref{Bloch-vanishing-direct}, for any elements
$\alpha\in CH^d(X)_l$, $\beta\in\Pic^0(X)_l$, the direct image of
the Massey triple product
$\overline{m}_3(\alpha,l,\beta)=\pi_*(m_3(\alpha,l,\beta))$ is well
defined.

\begin{prop}\label{Weil_Massey}
Suppose that $l$ is prime to ${\rm char}(k)$; then for any elements
$\alpha\in CH^d(X)_l$, $\beta\in\Pic^0(X)_l$, there is an equality
$$
\overline{m}_3(\alpha,l,\beta)=\psi_l(\alpha,\beta)^{(-1)^d},
$$
where $\psi_l$ is the Weil pairing between the $l$-torsion of
Albanese and Picard varieties of $X$.
\end{prop}
\begin{proof}
Since both sides of the equality evidently do not change under
extensions of the ground field, we can assume that the field $k$ is
algebraically closed. First, suppose that $\dim(X)=1$ and consider
elements $\LL,\M\in\Pic^0(X)_l$. Choose two adeles (in fact, ideles)
$f,g\in\A(X,\K_1^X)^1$ that correspond to $\LL$ and $\M$,
respectively. There are two adeles
$\widetilde{f},\widetilde{g}\in\A(X,\K_1^X)^0$ such that
$d\widetilde{f}=f^l$, $d\widetilde{g}=g^l$ (we write the group law
for $K_1$-groups in the multiplicative way). By definition, we have
$$
\overline{m}_3(\LL,l,\M)=\prod_{x\in X}
(\widetilde{f}_X,g_{Xx})_x(f_{Xx},\widetilde{g}_x)_x,
$$
where $(\cdot,\cdot)_x$ is the tame symbol in the discrete valuation
ring $\OO_{X,x}$. We may assume that the supports of the divisors
$D=-\div(f)$ and $E=-\div(g)$ do not intersect. In this case we get
$\phi_l(\LL,\M)=f_{X}(E)\cdot g_X(-D)$. This formula for the Weil
pairing $\psi_l(\LL,\M)$ of $\LL$ and $\M$ is well known. The proof
can be found in~\cite{How}, \cite{Maz}, and \cite{Gor} (these three
proofs use different methods).

Now suppose that $X$ is not a curve. Let $i:C\hookrightarrow X$ be a
general $(d-1)$-th hyperplane section of $X$. Consider two elements
$\M\in \Pic^0(C)_l$, $\LL\in \Pic^0(X)_l$. The projection formula
for Weil pairing (following, for example, from that for \'etale
cohomology) implies that $\psi_l(i_*(\M),\LL)=\psi_l(\M,i^*(\LL))$.
The projection formula for Massey higher products, stated in Lemma
\ref{product_formula_Massey}, implies that
$\phi_l(i_*(\M),\LL)=\phi_l(\M,i^*(\LL))^{(-1)^{d-1}}$. On the other
hand, it is well known that the map $i_*:\Pic^0(C)_l\to \Alb(X)_l$
is surjective. Thus, by the previous step, we get the needed result.
\end{proof}

Propositions \ref{Massey_comp} and \ref{Weil_Massey} imply the
following formula.

\begin{corol}
Let the class $\alpha\in CH^d(X)_l$ be represented by a zero-cycle
$z=\sum_i m_i\cdot z_i$ and let the class $\beta\in\Pic^0(X)_l$ be
represented by a divisor $D$ such that $z$ does not intersect with
$D$. Suppose that $\div(g)=lD$ with $g\in k(X)^*$ and
$d(\{f_m\})=lz$, where $d$ is the differential in the Gersten
complex on $X$, $f_m\in k(C_m)^*$, and $C_m$ are irreducible curves
on $X$. Assume that for any $m$ the rational function $f_m$ is
regular at all points from the intersection of $D$ with $C_m$. Then
in the notations from Proposition \ref{Massey_comp} we have the
following formula for the Weil pairing of $\alpha$ and $\LL$:
$$
\psi_l(\alpha,\LL)=(\prod_{x\in C\cap D}f(x)\cdot
\prod_{i}g^{-m_i}(z_i))^{(-1)^d}.
$$
\end{corol}

\end{document}